\documentclass[a4paper,twoside,10pt]{amsart}
\usepackage{amsmath,amsfonts,amssymb,amsthm}
\usepackage{mathtools,mathrsfs}
\usepackage{color}
\usepackage[a4paper]{geometry}
\usepackage[colorlinks=true,pdfstartview={XYZ null null 1.00},pdfview=FitH,citecolor=blue,urlcolor=cyan]{hyperref}
\usepackage{caption,enumerate}
\usepackage{bm}
\newtheorem{theorem}{Theorem}[section]
\newtheorem{lemma}[theorem]{Lemma}
\numberwithin{equation}{section}
\def\SzQ{\mathrm{Q}}
\def\DQ{\mathsf{Q}}
\def\im{\mathrm{i}}
\def\d{\mathrm{d}}
\def\id{\,\mathrm{d}}
\def\e{\mathrm{e}}

\def\Re{\operatorname{Re}}

\def\Im{\operatorname{Im}}
\newcommand\mytop[2]{\genfrac{}{}{0pt}{}{#1}{#2}}

\allowdisplaybreaks

\author[G. Nemes]{Gerg\H{o} Nemes}
\address{School of Mathematics, Harbin Institute of Technology, Xidazhi Street, Harbin 150001,\newline Heilongjiang, People's Republic of China\medskip}

\email{g.nemes@hit.edu.cn}

\keywords{asymptotic expansions, error bounds, Jacobi functions, hypergeometric function}
\subjclass[2020]{41A60, 33C05, 33C45}

\begin{document}

\title[Large-degree asymptotics for Jacobi and related functions]{Large-degree asymptotic expansions for\\ the Jacobi and related functions}
\dedicatory{Dedicated to Roderick S. C. Wong on the occasion of his $80^{\text{th}}$ birthday}

\begin{abstract} Simple asymptotic expansions for the Jacobi functions $P_\nu^{(\alpha, \beta)}(z)$ and $Q_\nu^{(\alpha, \beta)}(z)$ for large degree $\nu$, with fixed parameters $\alpha$ and $\beta$, are surprisingly rare in the literature, with only a few special cases covered. This paper addresses this notable gap by deriving simple (inverse) factorial expansions for these functions, complemented by explicit and computable error bounds. Additionally, we provide analogous results for the associated functions $\SzQ_\nu^{(\alpha, \beta)}(x)$ and $\DQ_\nu^{(\alpha, \beta)}(x)$.
\end{abstract}

\maketitle

\section{Introduction and main result}

In this paper, we derive asymptotic expansions for the Jacobi functions $P_\nu^{(\alpha ,\beta )} (z)$ and $Q_\nu^{(\alpha ,\beta )} (z)$, as well as for the associated functions $\SzQ_\nu^{(\alpha ,\beta )} (x)$ and $\DQ_\nu^{(\alpha ,\beta )}(x)$, in the regime where the degree $\nu$ becomes large while $\alpha$ and $\beta$ are held fixed. These functions are solutions to the Jacobi differential equation
\begin{equation}\label{JacobiEq}
(1 - z^2 )\frac{\d^2 w(z)}{\d z^2 } + (\beta  - \alpha  - z(\alpha  + \beta  + 2))\frac{\d w(z)}{\d z} + \nu (\nu  + \alpha  + \beta  + 1)w(z) = 0,
\end{equation}
which is characterised by regular singularities at $z = \pm 1$ and $z = \infty$.

In recent decades, significant effort has been dedicated to deriving asymptotic expansions for these functions. Generally, there are two types of expansions: those in terms of elementary functions and uniform expansions involving Bessel functions (a brief summary is provided later in this section). A common feature of these results is the complexity of their coefficients. In contrast, the expansions we present in this paper are simple (inverse) factorial expansions with elementary coefficients. Additionally, in Section \ref{bounds}, we provide explicit and computable error bounds for these new expansions under certain conditions on the degree $\nu$ and the parameters $\alpha$ and $\beta$. It is important to note that, unlike their uniform counterparts, our expansions break down near the transition points $z = \pm 1$ of \eqref{JacobiEq}. Surprisingly, these simple expansions have not been extensively documented in the literature, except for a few special cases. This omission may be because previous researchers primarily focused on deriving uniform asymptotic results, which naturally possess a more complicated structure. We note that (inverse) factorial series were recently obtained for the associated Legendre and Ferrers functions in \cite{Nemes2020}.

To begin, we define the relevant functions. The Jacobi functions of the first and second kind, with complex degree $\nu$ and parameters $\alpha$ and $\beta$, are defined in terms of the (regularised) hypergeometric function as follows (cf. \cite{Cohl2023}, \cite[Eq. (10.8.18)]{Erdelyi1981}):
\begin{equation}\label{Pdef}
P_\nu ^{(\alpha ,\beta )} (z) = \frac{\Gamma (\nu  + \alpha  + 1)}{\Gamma (\nu  + 1)}\mathbf{F}\!\left( \mytop{- \nu ,\nu  + \alpha  + \beta  + 1}{\alpha  + 1};\frac{1 - z}{2} \right)
\end{equation}
for $z\in \mathbb{C}\setminus \left(-\infty,-1\right]$ and $\nu+\alpha \notin \mathbb{Z}^{-}$, and
\begin{equation}\label{Qdef}
Q_\nu ^{(\alpha ,\beta )} (z) = \frac{2^{\nu  + \alpha  + \beta } \Gamma (\nu  + \alpha  + 1)\Gamma (\nu  + \beta  + 1)}{\left(z - 1\right)^{\nu  + \alpha  + 1} \left(z + 1\right)^\beta  }\mathbf{F}\!\left( \mytop{\nu  + 1,\nu  + \alpha  + 1}{2\nu  + \alpha  + \beta  + 2};\frac{2}{1 - z} \right)
\end{equation}
for $z\in \mathbb{C}\setminus \left(-\infty,1\right]$ and $\nu+\alpha, \nu+\beta \notin \mathbb{Z}^{-}$. For analytic continuation to other Riemann sheets, see \cite{Dunster1999}. In the case where $\nu = n$ is a non-negative integer, $P_n^{(\alpha, \beta)}(z)$ corresponds to the classical Jacobi polynomials. We will later use the following relation between the two types of Jacobi functions:
\begin{equation}\label{PQrel}
P_\nu ^{(\alpha ,\beta )} (x) = \lim_{\varepsilon \to 0^+} \tfrac{\im}{\pi}\left( \e^{\pi \alpha \im} Q_\nu ^{(\alpha ,\beta )} (x + \im\varepsilon ) - \e^{ - \pi \alpha \im} Q_\nu ^{(\alpha ,\beta )} (x - \im\varepsilon ) \right),
\end{equation}
which holds for $-1 < x < 1$ and $\nu + \alpha, \nu + \beta \notin \mathbb{Z}^{-}$ \cite[Eq. (2.3)]{Durand1978}. The literature offers (at least) two distinct definitions for extending $Q_\nu ^{(\alpha ,\beta )} (x)$ into the interval $-1<x<1$. Szeg\H{o} defines the associated function $\SzQ_\nu^{(\alpha ,\beta )}(x)$, with complex degree $\nu$ and parameters $\alpha$ and $\beta$, as follows (cf. \cite[Eq. (10.8.22)]{Erdelyi1981} and \cite[Eq. (4.62.9)]{Szego1939}):
\begin{equation}\label{Qfuncdef}
\SzQ_\nu ^{(\alpha ,\beta )} (x) = \lim_{\varepsilon\to 0^+} \tfrac{1}{2}\left( Q_\nu ^{(\alpha ,\beta)} (x + \im\varepsilon ) + Q_\nu ^{(\alpha ,\beta )} (x - \im\varepsilon ) \right),
\end{equation}
for $-1<x<1$ and $\nu+\alpha, \nu+\beta \notin \mathbb{Z}^{-}$. Alternatively, Durand \cite[Eq. (2.4)]{Durand1978} defines the associated function $\DQ_\nu^{(\alpha ,\beta )}(x)$ as
\begin{equation}\label{Qfuncdef2}
\DQ_\nu ^{(\alpha ,\beta )} (x) = \lim_{\varepsilon \to 0^+} \tfrac{1}{2}\left( \e^{\pi \alpha \im} Q_\nu ^{(\alpha ,\beta )} (x + \im\varepsilon ) + \e^{ - \pi \alpha \im} Q_\nu ^{(\alpha ,\beta )} (x - \im\varepsilon ) \right),
\end{equation}
for $-1<x<1$ and $\nu+\alpha, \nu+\beta \notin \mathbb{Z}^{-}$. Both functions are real when $\nu$, $\alpha$ and $\beta$ are real. Relevant applications and important properties of these functions are discussed in, for example, \cite{Wimp1997} and \cite{Cohl2025}.

The associated Legendre functions $P_\nu^\mu (z)$ and $Q_\nu^\mu (z)$ can be expressed in terms of Jacobi functions as
\[
P_\nu^\mu (z) = \frac{\Gamma(\nu + 1)}{\Gamma(\nu - \mu + 1)} \left( \frac{z + 1}{z - 1} \right)^{\mu / 2} P_\nu^{(-\mu, \mu)}(z),
\]
and
\[
Q_\nu^\mu (z) = \e^{\pi \im \mu} \frac{\Gamma(\nu + 1)}{\Gamma(\nu - \mu + 1)} \left( \frac{z + 1}{z - 1} \right)^{\mu / 2} Q_\nu^{(-\mu, \mu)}(z).
\]
Similarly, the Ferrers functions $\mathsf{P}_\nu^\mu (x)$ and $\mathsf{Q}_\nu ^\mu  (x)$ are related to the Jacobi function of the first kind and the associated function $\DQ_\nu ^{(\alpha ,\beta )} (x)$ via
\[
\mathsf{P}_\nu^\mu (x) = \frac{\Gamma(\nu + 1)}{\Gamma(\nu - \mu + 1)} \left( \frac{1+x}{1-x} \right)^{\mu / 2} P_\nu^{(-\mu, \mu)}(x),
\]
and
\[
\mathsf{Q}_\nu ^\mu  (x) = \frac{\Gamma (\nu  + 1)}{\Gamma (\nu  - \mu  + 1)}\left( \frac{1 + x}{1 - x} \right)^{\mu /2} \DQ_\nu ^{(-\mu ,\mu )} (x).
\]
For these representations, as well as alternative forms involving symmetric Jacobi functions or the choice $\beta = \pm\frac{1}{2}$, see \cite[\S 2.3]{Cohl2025}. Thus, the results of this paper can also be directly applied to these functions. However, the resulting expansions differ from those given in \cite{Nemes2020}.

Prior to presenting our main results, we briefly review existing asymptotic results in the literature. Darboux \cite{Darboux1878} established the leading-order asymptotic behaviour of the Jacobi polynomials $P_n ^{(\alpha ,\beta )} (z)$ for $-1 < z < 1$ and $z \in \mathbb{C} \setminus \left[-1,1\right]$, away from the points $z = \pm 1$, in the case of large $n$ with fixed $\alpha$ and $\beta$. Vinogradov \cite{Vinogradov2013} derived an estimate for the absolute error of Darboux's approximation. Watson \cite{Watson1918} provided complete asymptotic expansions for the associated hypergeometric functions, with the coefficients in these expansions being notably complicated. In the case of Jacobi polynomials, Watson's results were presented by Szeg\H{o} \cite[Theorem 8.21.9]{Szego1939}, who also derived a complicated uniform asymptotic expansion in terms of Bessel functions for $P_n^{(\alpha ,\beta)} (x)$ for large $n$ over the interval $-1 < x \le 1$ \cite{Szego1933}. Hahn \cite{Hahn1980} derived factorial series expansions for the Jacobi function $P_n^{(\alpha ,\beta )}(z)$ for $-1 < z < 1$ and $\Im(z) < 0$, away from the points $z = \pm 1$, in the case of large positive $n$ with fixed $\alpha$ and $\beta$. These expansions were complemented with error bounds for $-\frac{1}{2} < \alpha, \beta < \frac{1}{2}$. McCabe and Connor \cite{McCabe1996} provided the leading-order asymptotic behaviour of $P_n^{(\alpha ,\beta )}(x)$, $\SzQ_n^{(\alpha ,\beta )}(x)$, and $\DQ_n^{(\alpha ,\beta )}(x)$ for $0 < x < 1$ in the case of large $n$ with fixed $\alpha$ and $\beta$. Several authors have derived uniform asymptotic approximations or expansions of the Jacobi functions in terms of Bessel functions, which remain valid at the endpoint $z = 1$ of the interval $-1 \le z \le 1$, with fixed parameters $\alpha$ and $\beta$ in specific real or complex domains \cite{Bai2007,Baratella1991,Durand2019,Elliott1971,Frenzen1985,Gatteschi1994,Wong1992,Wong1996,Wong2003,Wong2004,Zhu2006}. In many of these works, the expansions are accompanied by error bounds. Jones \cite{Jones2001} obtained large-$\nu$ uniform asymptotic expansions for the hypergeometric functions associated with $P_\nu^{(\alpha ,\beta )}(z)$ for $z$ outside the interval $-1 < z < 1$. Finally, several studies have focused on obtaining large-$n$ asymptotic results for the Jacobi polynomials $P_n^{(\alpha ,\beta )}(z)$ that are uniform in $\alpha$ and $\beta$ within certain real or complex domains \cite{Chen1991,Kuijlaars2004,Aoki2019,Gil2021,Szehr2022}.

Before stating our main results, we introduce some necessary notation. For any non-negative integer $n$ and $\mu \in \mathbb{C}$, we define
\begin{equation}\label{andef}
a_n (\mu) = \frac{\left(4\mu^2 - 1^2\right) \left(4\mu^2 - 3^2\right) \cdots \left(4\mu^2 - (2n - 1)^2\right)}{8^n n!} = \frac{\left(\frac{1}{2} + \mu\right)_n \left(\frac{1}{2} - \mu\right)_n}{\left(-2\right)^n  n!},
\end{equation}
where $(w)_n$ denotes the Pochhammer symbol, defined by $(w)_n = w(w + 1)\cdots(w + n - 1)$ for $n \ge 1$, and $(w)_0 = 1$. We will use $B(w_1,w_2)$ to denote the beta function \cite[\href{http://dlmf.nist.gov/5.12}{\S5.12}]{NIST:DLMF} and $\mathcal{D}$ to represent the domain
\begin{equation}\label{Ddef}
\mathcal{D} = \left\{ \xi :\Re(\xi ) > 0,\; \left| \Im(\xi ) \right| < \tfrac{\pi }{2} \right\}.
\end{equation}
Note that $\cosh(2\xi)$ is a biholomorphic bijection between $\mathcal{D}$ and $\mathbb{C}\setminus \left(-\infty,1\right]$. We further define, for $\xi \in \mathcal{D}$ and $\alpha \in \mathbb{C}$,
\begin{equation}\label{Cdef}
C(\xi,\alpha )  = \begin{cases} \sin (\pi \alpha ) & \text{if }\; \xi>0, \\ \mp \im \e^{ \pm \pi \im\alpha } & \text{if }\; 0 <  \pm \Im(\xi ) < \frac{\pi}{2}. \end{cases}
\end{equation} 

The two theorems below present inverse factorial expansions for the Jacobi and associated functions as the order $\nu$ becomes large. Unlike standard asymptotic power series, which are expressed in negative powers of the large variable, inverse factorial expansions are asymptotic series involving gamma functions with decreasing arguments that depend on the large variable. One advantage of using these expansions in the case of Jacobi and associated functions is that the coefficients can be expressed in simple, closed forms. As Szeg\H{o} noted \cite[p. 196]{Szego1939}, the corresponding asymptotic power series involve considerably more complicated coefficients. Moreover, as will be shown in Section \ref{bounds}, these inverse factorial expansions allow for explicit and computable error bounds, which would be much more difficult to obtain using standard asymptotic power series. However, as mentioned in the introduction, a drawback of inverse factorial expansions compared to uniform asymptotic expansions in terms of special functions is their breakdown near the transition points $z = \pm 1$ of \eqref{JacobiEq}. Finally, we note that inverse factorial expansions also arise naturally in exponential asymptotics; see, for example, \cite{OldeDaalhuis1998,OldeDaalhuis2004}.

\begin{theorem}\label{thm1} Assume that $\xi  \in  \mathcal{D}$, that $\left| \sinh \xi \right|,\left| \cosh \xi \right|  \geq  \varepsilon$ ($> 0$), and that $\alpha,\beta  \in \mathbb{C}$ are bounded. Under these conditions, the Jacobi functions admit the following inverse factorial expansions:
\begin{gather}\label{Pifacasymp}
\begin{split}
2^{2\nu + \alpha + \beta + 1} & B(\nu+1,\nu + \alpha + \beta + 1)P_\nu^{(\alpha ,\beta )} (\cosh(2\xi) ) \sim \\ & \left( \frac{\e^{\xi } }{\sinh \xi} \right)^{\alpha + 1/2} \left( \frac{\e^{\xi } }{\cosh \xi} \right)^{\beta + 1/2} \e^{2\nu\xi }  \sum_{n = 0}^\infty g_n ( -\xi,\alpha ,\beta )\frac{\Gamma (2\nu + \alpha + \beta  -n + 1)}{\Gamma (2\nu  + \alpha  + \beta  + 2)} 
\\ & - C(\xi ,\alpha )\left( \frac{\e^{ - \xi } }{\sinh \xi} \right)^{\alpha + 1/2} \left( \frac{\e^{ - \xi } }{\cosh \xi} \right)^{\beta + 1/2} \e^{ - 2\nu\xi } \sum_{m = 0}^\infty g_m ( \xi,\alpha ,\beta )\frac{\Gamma (2\nu + \alpha + \beta  -m + 1)}{\Gamma (2\nu  + \alpha  + \beta  + 2)} ,
\end{split}
\end{gather}
as $\nu\to\infty$ within the sector $\left|\arg \nu \right|\le \frac{\pi}{2}-\delta$ ($<\frac{\pi}{2}$), and
\begin{gather}\label{Qifacasymp}
\begin{split}
\frac{2^{2\nu + \alpha + \beta + 1}}{\pi} & B(\nu + 1, \nu + \alpha + \beta + 1) Q_\nu^{(\alpha, \beta)}(\cosh(2\xi)) \sim\\ &\left( \frac{\e^{-\xi}}{\sinh \xi} \right)^{\alpha + 1/2} \left( \frac{\e^{-\xi}}{\cosh \xi} \right)^{\beta + 1/2} \e^{-2\nu \xi}  \sum_{n = 0}^\infty g_n(\xi, \alpha, \beta) \frac{\Gamma(2\nu + \alpha + \beta - n + 1)}{\Gamma(2\nu + \alpha + \beta + 2)},
\end{split}
\end{gather}
as $\nu\to\infty$ in the sector $\left|\arg \nu \right|\le \pi-\delta$ ($<\pi$). The coefficients $g_n(\pm \xi,\alpha ,\beta )$ are given by
\[
g_n (\pm \xi,\alpha ,\beta ) = \sum\limits_{\ell  = 0}^n   (\pm 1)^\ell a_\ell  (\alpha )a_{n - \ell } (\beta )\left( \frac{\e^{ \pm \xi} }{\sinh \xi} \right)^\ell  \left( \frac{\e^{ \pm \xi} }{\cosh \xi} \right)^{n - \ell } .
\]
The fractional powers are defined to be positive for positive real $\xi$ and are defined by continuity elsewhere.
\end{theorem}

\begin{theorem}\label{thm2}
Suppose that ($0 <$) $\varepsilon  \leq  \zeta  \leq  \frac{\pi}{2} - \varepsilon$ ($< \frac{\pi}{2}$), and let $\alpha,\beta  \in \mathbb{C}$ be bounded. Define
\begin{gather}\label{zetadef1}
\begin{split}
& \zeta^{(1)}_{\nu,n,\ell} = (2\nu + \alpha + \beta - n + 1)\zeta - \left( \alpha - \ell + \tfrac{1}{2} \right) \tfrac{\pi}{2},\\ & \zeta^{(2)}_{\nu ,n,\ell }= (2\nu + \alpha + \beta - n + 1)\zeta  + \left( \alpha  + \ell  + \tfrac{1}{2} \right)\tfrac{\pi}{2},
\end{split}
\end{gather}
where $n$ and $\ell$ are non-negative integers. Then, the Jacobi function and the associated functions admit the following inverse factorial expansions:
\begin{gather}\label{invfexp1}
\begin{split}
2^{2\nu + \alpha + \beta} & B(\nu + 1,\nu + \alpha + \beta + 1)P_\nu^{(\alpha, \beta)} (\cos(2\zeta)) \sim\\ & \frac{1}{\sin^{\alpha + 1/2} \zeta \cos^{\beta + 1/2} \zeta}  \sum\limits_{n = 0}^\infty \left(\sum\limits_{\ell = 0}^n \frac{a_\ell (\alpha) a_{n - \ell} (\beta)}{\sin^\ell \zeta \cos^{n- \ell} \zeta} \cos \left(\zeta^{(1)}_{\nu,n,\ell}\right)\right) \frac{\Gamma(2\nu + \alpha + \beta - n + 1)}{\Gamma(2\nu+\alpha + \beta+2)},
\end{split}
\end{gather}
\begin{gather}\label{invfexp2}
\begin{split}
\frac{2^{2\nu + \alpha + \beta + 1}}{\pi} & B(\nu + 1,\nu + \alpha + \beta + 1) \SzQ_\nu^{(\alpha, \beta)} (\cos(2\zeta)) \sim \\ & \frac{1}{\sin^{\alpha + 1/2} \zeta \cos^{\beta + 1/2} \zeta} \sum_{n = 0}^\infty \left(\sum\limits_{\ell = 0}^n \frac{a_\ell (\alpha) a_{n - \ell} (\beta)}{\sin^\ell \zeta \cos^{n- \ell} \zeta} \cos \left(\zeta^{(2)}_{\nu,n,\ell}\right)\right) \frac{\Gamma(2\nu + \alpha + \beta - n + 1)}{\Gamma(2\nu + \alpha + \beta + 2)} ,
\end{split}
\end{gather}
\begin{gather}\label{invfexp3}
\begin{split}
\frac{2^{2\nu + \alpha + \beta + 1}}{\pi} & B(\nu + 1, \nu + \alpha + \beta + 1) \DQ_\nu^{(\alpha, \beta)} (\cos(2\zeta)) \sim \\ & \frac{-1}{\sin^{\alpha + 1/2} \zeta \cos^{\beta + 1/2} \zeta} \sum_{n = 0}^\infty \left(\sum\limits_{\ell = 0}^n \frac{a_\ell (\alpha) a_{n - \ell} (\beta)}{\sin^\ell \zeta \cos^{n- \ell} \zeta} \sin \left(\zeta^{(1)}_{\nu,n,\ell}\right)\right) \frac{\Gamma(2\nu + \alpha + \beta - n + 1)}{\Gamma(2\nu + \alpha + \beta + 2)},
\end{split}
\end{gather}
as $\nu\to\infty$ in the sector $\left|\arg \nu \right|\le \pi-\delta$ ($<\pi$).
\end{theorem}

The proofs of Theorems \ref{thm1} and \ref{thm2} are presented in Section \ref{asympproofs}. To our knowledge, inverse factorial-type expansions for these functions have not previously been established in the literature.

We note that if $2\alpha$ and $2\beta$ are odd integers, these inverse factorial expansions terminate, providing an exact representation of the corresponding function (see Sections \ref{infproof1} and \ref{infproof2}).

The next theorem presents factorial expansions for the Jacobi functions. Like inverse factorial expansions, factorial expansions involve gamma functions; however, in this case, the gamma functions have increasing arguments. A key distinction is that, for certain values of $z$, these expansions are not merely asymptotic but actually convergent. As before, the coefficients admit simple representations, and explicit, computable error bounds can be obtained (see Section \ref{bounds}). However, a limitation is that factorial expansions, like their inverse counterparts, break down near the transition points $z = \pm 1$ of \eqref{JacobiEq}. Factorial expansions have a long history and are well-established in asymptotic analysis. For further details, see, for example, \cite[Ch. 7]{Lauwerier1974} and \cite{Delabaere2007}.

\begin{theorem}\label{thm3} Suppose that $\xi  \in  \mathcal{D}$, that $\left| \sinh \xi \right|,\left| \cosh \xi \right|  \geq  \varepsilon$ ($> 0$), and that $\alpha,\beta  \in \mathbb{C}$ are bounded. Under these conditions, the Jacobi functions admit factorial expansions given by
\begin{gather}\label{Pfacasymp}
\begin{split}
 &\frac{\pi P_\nu^{(\alpha ,\beta )} (\cosh (2\xi ))}{2^{2\nu + \alpha  + \beta } B(\nu + \alpha  + 1,\nu + \beta  + 1)}
 \sim  \\ & \left( \frac{\e^{ \xi }}{\sinh \xi} \right)^{\alpha  + 1/2} \left( \frac{\e^{ \xi }}{\cosh \xi} \right)^{\beta  + 1/2} \e^{2\nu \xi } \sum_{n = 0}^\infty (-1)^n g_n(\xi,\alpha,\beta)\frac{\Gamma (2\nu+\alpha+\beta+2)}{\Gamma(2\nu+\alpha+\beta+n+2)} 
\\ & - C(\xi ,\alpha ) \left( \frac{\e^{ - \xi }}{\sinh \xi} \right)^{\alpha  + 1/2} \left( \frac{\e^{ - \xi }}{\cosh \xi} \right)^{\beta  + 1/2} \e^{ - 2\nu \xi } \sum_{m = 0}^\infty (-1)^m g_m(-\xi,\alpha,\beta)\frac{\Gamma (2\nu+\alpha+\beta+2)}{\Gamma(2\nu+\alpha+\beta+m+2)} 
\end{split}
\end{gather}
as $\nu\to\infty$ in the sector $\left|\arg \nu \right|\le \frac{\pi}{2}-\delta$ ($<\frac{\pi}{2}$), and
\begin{gather}\label{Qfacasymp}
\begin{split}
& \frac{ Q_\nu ^{(\alpha ,\beta )} (\cosh (2\xi) )}{2^{2\nu  + \alpha  + \beta } B(\nu  + \alpha  + 1,\nu  + \beta  + 1)} \sim\\ & \left( \frac{\e^{ - \xi }}{\sinh \xi} \right)^{\alpha  + 1/2} \left( \frac{\e^{ - \xi }}{\cosh \xi} \right)^{\beta  + 1/2} \e^{ - 2\nu \xi } \sum_{n = 0}^\infty (-1)^n g_n(-\xi,\alpha,\beta)\frac{\Gamma (2\nu+\alpha+\beta+2)}{\Gamma(2\nu+\alpha+\beta+n+2)} 
\end{split}
\end{gather}
as $\nu\to\infty$ within the sector $\left|\arg \nu \right|\le \pi-\delta$ ($<\pi$). The coefficients $g_n(\pm \xi,\alpha ,\beta )$ are the same as those appearing in the corresponding inverse factorial series. The fractional powers are defined to be positive for positive real $\xi$ and are defined by continuity elsewhere.
\end{theorem}

The proof of Theorem \ref{thm3} is given in Section \ref{asympproofs}. In the special case where $\Re(\xi)>0$, $-\frac{\pi}{2} < \Im(\xi)<0$ (or equivalently, $\Im(z)<0$), with real parameters $\alpha$ and $\beta$ and $\nu$ approaching infinity along positive real values, the factorial expansion \eqref{Pfacasymp} was derived by Hahn \cite{Hahn1980}. Note that some typographical errors appear in the expressions for the error term and coefficients in \cite{Hahn1980}.

We note that when $\Re(2\nu + \alpha + \beta + 1) > \Re(\alpha' + \beta')$ and $\left|\Re\left(\e^{\pm 2\xi}\right)\right| < \frac{1}{2}$, the infinite series in \eqref{Pfacasymp} converges to its left-hand side, and similarly, when $\Re(2\nu + \alpha + \beta + 1) > \Re(\alpha' + \beta')$ and $\left|\Re\left(\e^{-2\xi}\right)\right| < \frac{1}{2}$, the series in \eqref{Qfacasymp} converges to the left-hand side. Here, $\alpha'$ and $\beta'$ are defined in \eqref{abprime}. For a short proof, refer to the end of Section \ref{fproof1}.

The following theorem gives analogous factorial expansions for the Jacobi function $P_\nu ^{(\alpha ,\beta )} (x)$ and the associated functions as $\nu \to \infty$.

\begin{theorem}\label{thm4}
Suppose that ($0 <$) $\varepsilon  \leq  \zeta  \leq  \frac{\pi}{2} - \varepsilon$ ($< \frac{\pi}{2}$), and let $\alpha,\beta  \in \mathbb{C}$ be bounded. Define
\begin{gather}\label{zetadef2}
\begin{split}
& \zeta^{(3)}_{\nu ,n,\ell }  = (2\nu  + \alpha  + \beta  + n + 1)\zeta  - \left( \alpha  + \ell  + \tfrac{1}{2} \right)\tfrac{\pi }{2},\\ & \zeta^{(4)}_{\nu ,n,\ell }  = (2\nu  + \alpha  + \beta  + n + 1)\zeta  + \left( \alpha  - \ell  + \tfrac{1}{2} \right)\tfrac{\pi }{2},
\end{split}
\end{gather}
where $n$ and $\ell$ are non-negative integers. Then, the Jacobi function and the associated functions admit the following factorial expansions:
\begin{gather}\label{Pfacasymp2}
\begin{split}
& \frac{ \pi P_\nu ^{(\alpha ,\beta )} (\cos (2\zeta) )}{2^{2\nu  + \alpha  + \beta +1} B(\nu  + \alpha  + 1,\nu  + \beta  + 1)} \sim \\& 
\frac{1}{\sin ^{\alpha  + 1/2} \zeta \cos ^{\beta  + 1/2} \zeta }\sum_{n = 0}^\infty  ( - 1)^n \left(\sum\limits_{\ell = 0}^n \frac{a_\ell (\alpha) a_{n - \ell} (\beta)}{\sin^\ell \zeta \cos^{n- \ell} \zeta} \cos \left(\zeta^{(3)}_{\nu ,n,\ell }\right)\right)\frac{\Gamma (2\nu  + \alpha  + \beta  + 2)}{\Gamma (2\nu  + \alpha  + \beta  + n + 2)} ,
\end{split}
\end{gather}
\begin{gather}\label{Qfacasymp2}
\begin{split}
& \frac{ \SzQ_\nu ^{(\alpha ,\beta )} (\cos (2\zeta) )}{2^{2\nu  + \alpha  + \beta } B(\nu  + \alpha  + 1,\nu  + \beta  + 1)} \sim \\& 
\frac{1}{\sin ^{\alpha  + 1/2} \zeta \cos ^{\beta  + 1/2} \zeta }\sum_{n = 0}^\infty  ( - 1)^n \left(\sum\limits_{\ell = 0}^n \frac{a_\ell (\alpha) a_{n - \ell} (\beta)}{\sin^\ell \zeta \cos^{n- \ell} \zeta} \cos \left(\zeta^{(4)}_{\nu ,n,\ell }\right)\right)\frac{\Gamma (2\nu  + \alpha  + \beta  + 2)}{\Gamma (2\nu  + \alpha  + \beta  + n + 2)} ,
\end{split}
\end{gather}
\begin{gather}\label{Qfacasymp3}
\begin{split}
& \frac{ \DQ_\nu ^{(\alpha ,\beta )} (\cos (2\zeta) )}{2^{2\nu  + \alpha  + \beta } B(\nu  + \alpha  + 1,\nu  + \beta  + 1)} \sim \\& 
\frac{-1}{\sin ^{\alpha  + 1/2} \zeta \cos ^{\beta  + 1/2} \zeta }\sum_{n = 0}^\infty  ( - 1)^n \left(\sum\limits_{\ell = 0}^n \frac{a_\ell (\alpha) a_{n - \ell} (\beta)}{\sin^\ell \zeta \cos^{n- \ell} \zeta} \sin \left(\zeta^{(3)}_{\nu ,n,\ell }\right)\right)\frac{\Gamma (2\nu  + \alpha  + \beta  + 2)}{\Gamma (2\nu  + \alpha  + \beta  + n + 2)} ,
\end{split}
\end{gather}
as $\nu\to\infty$ in the sector $\left|\arg \nu \right|\le \pi-\delta$ ($<\pi$).
\end{theorem}

The proof of Theorem \ref{thm4} can be found in Section \ref{asympproofs}. In the special case where $\alpha$ and $\beta$ are real and $\nu$ approaches infinity along positive real values, the factorial expansion \eqref{Pfacasymp2} was derived by Hahn \cite{Hahn1980}.

In the case that $\Re(2\nu + \alpha + \beta + 1) > \Re(\alpha' + \beta')$ and $\frac{\pi }{6} < \zeta  < \frac{\pi }{3}$, the infinite series on the right-hand sides of \eqref{Pfacasymp2}, \eqref{Qfacasymp2} and \eqref{Qfacasymp3} converge to the corresponding functions on the left-hand sides. For a short proof, refer to the end of Section \ref{fproof3}.

We note that if $2\alpha$ and $2\beta$ are odd integers, the factorial expansions given in Theorems \ref{thm3} and \ref{thm4} terminate, yielding exact representations of the corresponding functions (see Sections \ref{fproof1} and \ref{fproof3}).

The proofs of Theorems \ref{thm1}--\ref{thm4} rely on the connection between these functions and hypergeometric functions, Watson's large-parameter asymptotic expansions for hypergeometric functions, and our error-bound theorems presented in Section \ref{bounds}. 

To prove our error bounds, we first derive explicit integral representations for the remainder terms in the expansions given in Theorems \ref{thm1}--\ref{thm4}. These are based on suitable representations of the Jacobi functions. In the following, we briefly discuss such representations for the Jacobi function of the second kind, $Q_\nu^{(\alpha, \beta)}(z)$, and give a rough outline of how they are used in our analysis.

(i) To study the remainder term in the inverse factorial expansion \eqref{Qifacasymp} of the Jacobi function of the second kind, we shall use the following, less commonly employed representation involving the modified Bessel function of the second kind:
\begin{equation}\label{integralrep1}
\begin{split}
2^{2\nu + \alpha + \beta} & B(\nu + 1, \nu + \alpha + \beta + 1)  Q_\nu^{(\alpha, \beta)} (\cosh (2\xi)) = \frac{1}{\sinh^\alpha \xi \cosh^\beta \xi} \\ &\times \frac{1}{\Gamma (2\nu + \alpha + \beta + 2)}\int_0^{+\infty} t^{2\nu + \alpha + \beta + 1} K_\alpha (t \sinh \xi) K_\beta (t \cosh \xi) \id t ,
\end{split}
\end{equation}
which holds for $\xi  \in \mathcal{D}$ and $\Re (2\nu+\alpha+\beta+2) > \left|\Re(\alpha)\right|+\left|\Re(\beta)\right|$\footnote{Note that this condition ensures that none of $ \nu + 1 $, $ \nu + \alpha $, $ \nu + \beta $, or $ \nu + \alpha + \beta $ are negative integers, so the left-hand side of \eqref{integralrep1} is well-defined.} (combine \cite[Eq. (21)]{Cohl2023} and \cite[Eq. (6.8.47)]{Erdelyi1954}). An immediate consequence of the factor $t^{2\nu + \alpha + \beta + 1}$ in the integrand is that, when $\Re(\nu)$ is large and positive, the main contribution to the integral comes from large values of $t$. Thus, we can obtain a large-$\nu$ asymptotic expansion by using the well-known large-argument asymptotic expansion of the Bessel function. This approach also allows for the straightforward derivation of explicit and computable error bounds by employing the known error estimates for the large-argument expansion of the Bessel function. For the corresponding representation of $P_\nu^{(\alpha, \beta)} (\cosh (2\xi))$, see equation \eqref{eq9}.

(ii) To obtain an explicit remainder term for the factorial expansion \eqref{Qfacasymp} of the Jacobi function of the second kind, we shall employ the following alternative hypergeometric representation:
\begin{equation}\label{Qhypergeom}
Q_\nu ^{(\alpha ,\beta )} (\cosh (2\xi )) = \frac{2\Gamma (\nu  + \alpha  + 1)\Gamma (\nu  + \beta  + 1)}{\left( \sinh \xi \right)^{2\alpha } \left(\cosh \xi \right)^{2\nu  + 2\beta  + 2} }\mathbf{F}\!\left( \mytop{\nu  + 1,\nu  + \beta  + 1}{2\nu  + \alpha  + \beta  + 2};\frac{1}{\cosh ^2 \xi } \right),
\end{equation}
valid for $\xi \in \mathcal{D}$ and $\nu + \alpha, \nu + \beta \notin \mathbb{Z}^{-}$ (see \cite[Eq. (22)]{Cohl2023} with $z = \cosh(2\xi)$). Building on ideas due to Hahn \cite{Hahn1980}, we shall reformulate the hypergeometric function as a double integral, making it suitable for deriving a factorial expansion with an explicit remainder term.

The remainder of the paper is organised as follows. In Section \ref{bounds}, we provide error bounds for the inverse factorial and factorial expansions presented in Theorems \ref{thm1}–\ref{thm4}. Section \ref{asympproofs} contains the proofs of Theorems \ref{thm1}–\ref{thm4}. The proofs of the error bounds for the inverse factorial series are given in Sections \ref{infproof1} and \ref{infproof2}, while those for the factorial series are provided in Sections \ref{fproof1}--\ref{fproof3}. Section \ref{numerics} presents numerical examples that illustrate the sharpness of the bounds and the accuracy of the expansions.

\section{Error bounds}\label{bounds}

In this section, we provide several explicit and computable bounds for the remainder terms of the inverse factorial and factorial expansions stated in Theorems \ref{thm1}–\ref{thm4}, truncated after finitely many terms, under specific conditions on the degree $\nu$ and the parameters $\alpha$ and $\beta$.

In our first two theorems below, we provide computable error bounds for the inverse factorial expansions \eqref{Pifacasymp} and \eqref{Qifacasymp} of the Jacobi functions $P_\nu^{(\alpha ,\beta )}(z)$ and $Q_\nu^{(\alpha ,\beta )}(z)$, respectively. In Theorem \ref{Pinfacthm}, we introduce the notation
\begin{equation}\label{chidef}
\chi (p) = \sqrt \pi  \frac{\Gamma\left(\frac{p}{2} + 1\right)}{\Gamma \left(\frac{p}{2} + \frac{1}{2}\right)},\quad p>0.
\end{equation}
We remark that $\chi (p) \sim \sqrt{\frac{\pi}{2} p}$ as $p\to+\infty$ \cite[\href{http://dlmf.nist.gov/9.7.E4}{Eq. (9.7.4)}]{NIST:DLMF}. For the definition of the domain $\mathcal{D}$, see \eqref{Ddef}.

\begin{theorem}\label{Pinfacthm} Let $\xi\in \mathcal{D}$ and let $N$, $M$ both be arbitrary non-negative integers. Let $\nu$, $\alpha$, and $\beta$ be complex numbers satisfying $\Re(2\nu + \alpha + \beta + 1) > \max(N, M)$. Then we have
\begin{gather}\label{Pinfac}
\begin{split}
& 2^{2\nu + \alpha + \beta + 1} B(\nu+1,\nu + \alpha + \beta + 1)P_\nu^{(\alpha ,\beta )} (\cosh(2\xi) ) = \\ & \left( \frac{\e^{\xi } }{\sinh \xi} \right)^{\alpha + 1/2} \left( \frac{\e^{\xi } }{\cosh \xi} \right)^{\beta + 1/2} \e^{2\nu\xi } \left( \sum_{n = 0}^{N - 1} g_n ( -\xi,\alpha ,\beta )\frac{\Gamma (2\nu + \alpha + \beta  -n + 1)}{\Gamma (2\nu  + \alpha  + \beta  + 2)} + R_N^{(P1)} (\nu,\xi ,\alpha ,\beta ) \right)
\\ & - C(\xi ,\alpha )\left( \frac{\e^{ - \xi } }{\sinh \xi} \right)^{\alpha + 1/2} \left( \frac{\e^{ - \xi } }{\cosh \xi} \right)^{\beta + 1/2} \e^{ - 2\nu\xi } \left( \sum_{m = 0}^{M - 1} g_m ( \xi,\alpha ,\beta )\frac{\Gamma (2\nu + \alpha + \beta  -m + 1)}{\Gamma (2\nu  + \alpha  + \beta  + 2)} + R_M^{(P2)} (\nu,\xi ,\alpha ,\beta ) \right),
\end{split}
\end{gather}
where $C(\xi ,\alpha )$ is defined in \eqref{Cdef}, and the remainder terms satisfy the following estimates:
\begin{gather}\label{eq40}
\begin{split}
&\left|R_N^{(P1)} (\nu,\xi ,\alpha ,\beta )\right|  \le \left| \frac{\cos(\pi \alpha) \cos(\pi \beta)}{\cos(\pi \Re(\alpha)) \cos(\pi \Re(\beta))} a_N(\Re(\alpha)) \left( \frac{\e^{-\xi}}{\sinh \xi} \right)^N \right| \frac{\Gamma (\Re (2\nu  + \alpha  + \beta ) - N + 1)}{\left| \Gamma (2\nu  + \alpha  + \beta  + 2)\right|} \\ & \times \begin{cases} 1 & \text{if }\; \Re \left(\e^{2\xi } \right) \le 1, \\ 
\min\big(\left| 1 - \e^{ - 2\xi } \right|\left| \csc (2\Im(\xi) ) \right|,1 + \chi\big(N+\tfrac{1}{2}\big)\big) & \text{if }\; \Re \left(\e^{2\xi } \right) > 1  \end{cases} \\ &
+ \sum_{\ell = 0}^{N-1} \left| \frac{\cos(\pi \beta)}{\cos(\pi \Re(\beta))} a_\ell(\alpha) a_{N-\ell}(\Re(\beta)) \left( \frac{\e^{-\xi}}{\sinh \xi} \right)^\ell \left( \frac{\e^{-\xi}}{\cosh \xi} \right)^{N-\ell} \right| \frac{\Gamma (\Re (2\nu  + \alpha  + \beta ) - N + 1)}{\left| \Gamma (2\nu  + \alpha  + \beta  + 2)\right|},
\end{split}
\end{gather}
provided $\left|\Re(\alpha)\right|<N+ \frac{1}{2}$ and $\left|\Re(\beta)\right| < \frac{1}{2}$,
\begin{gather}\label{eq41}
\begin{split}
&\left|R_N^{(P1)} (\nu,\xi ,\alpha ,\beta )\right|  \le  \left(\left| \frac{\cos(\pi \alpha) \cos(\pi \beta)}{\cos(\pi \Re(\alpha)) \cos(\pi \Re(\beta))} a_N(\Re(\beta)) \left( \frac{\e^{-\xi}}{\cosh \xi} \right)^N \right|  \right. \\ &+ \left. \sum_{\ell = 1}^N \left| \frac{\cos(\pi \alpha)}{\cos(\pi \Re(\alpha))} a_\ell(\Re(\alpha)) a_{N-\ell}(\beta) \left( \frac{\e^{-\xi}}{\sinh \xi} \right)^\ell \left( \frac{\e^{-\xi}}{\cosh \xi} \right)^{N-\ell} \right| \right) \frac{\Gamma (\Re (2\nu  + \alpha  + \beta ) - N + 1)}{\left| \Gamma (2\nu  + \alpha  + \beta  + 2)\right|}
\\ & \times \begin{cases} 1 & \text{if }\; \Re \left(\e^{2\xi } \right)\le 1, \\ 
\min\big(\left| 1 - \e^{ - 2\xi } \right|\left| \csc (2\Im(\xi) ) \right|,1 + \chi\big(N+\tfrac{1}{2}\big)\big) & \text{if }\; \Re \left(\e^{2\xi } \right) > 1 , \end{cases}
\end{split}
\end{gather}
provided $\left|\Re(\alpha)\right|<\frac{1}{2}$ and $\left|\Re(\beta)\right| < N+ \frac{1}{2}$, and
\begin{gather}\label{eq44}
\begin{split}
&\left|R_N^{(P1)} (\nu,\xi ,\alpha ,\beta )\right|  \le \\ & \sum_{\ell = 0}^N \left| \frac{\cos(\pi \alpha) \cos(\pi \beta)}{\cos(\pi \Re(\alpha)) \cos(\pi \Re(\beta))} a_\ell(\Re(\alpha)) a_{N-\ell}(\Re(\beta)) \left( \frac{\e^{-\xi}}{\sinh \xi} \right)^\ell \left( \frac{\e^{-\xi}}{\cosh \xi} \right)^{N-\ell} \right|  \frac{\Gamma (\Re (2\nu  + \alpha  + \beta ) - N + 1)}{\left| \Gamma (2\nu  + \alpha  + \beta  + 2)\right|}
\\ & \times \begin{cases} 1 & \text{if }\; \Re \left(\e^{2\xi } \right) \le 1, \\ 
\min\big(\left| 1 - \e^{ - 2\xi } \right|\left| \csc (2\Im(\xi) ) \right|,1 + \chi\big(N+\tfrac{1}{2}\big)\big) & \text{if }\; \Re \left(\e^{2\xi } \right) > 1 , \end{cases}
\end{split}
\end{gather}
provided $\left|\Re(\alpha)\right|, \left|\Re(\beta)\right| <\frac{1}{2}$. Similarly,
\begin{gather}\label{eq42}
\begin{split}
&\left|R_M^{(P2)} (\nu,\xi ,\alpha ,\beta )\right|  \le \left( \left| \frac{\cos(\pi \alpha) \cos(\pi \beta)}{\cos(\pi \Re(\alpha)) \cos(\pi \Re(\beta))} a_M(\Re(\alpha)) \left( \frac{\e^{\xi}}{\sinh \xi} \right)^M \right| \right. \\ &
+\left. \sum_{\ell = 0}^{M-1} \left| \frac{\cos(\pi \beta)}{\cos(\pi \Re(\beta))} a_\ell(\alpha) a_{M-\ell}(\Re(\beta)) \left( \frac{\e^{\xi}}{\sinh \xi} \right)^\ell \left( \frac{\e^{\xi}}{\cosh \xi} \right)^{M-\ell} \right| \right)\frac{\Gamma (\Re (2\nu  + \alpha  + \beta ) - M + 1)}{\left| \Gamma (2\nu  + \alpha  + \beta  + 2)\right|},
\end{split}
\end{gather}
provided $\left|\Re(\alpha)\right|<M+ \frac{1}{2}$ and $\left|\Re(\beta)\right| < \frac{1}{2}$,
\begin{gather}\label{eq43}
\begin{split}
&\left|R_M^{(P2)} (\nu,\xi ,\alpha ,\beta )\right|  \le  \left(\left| \frac{\cos(\pi \alpha) \cos(\pi \beta)}{\cos(\pi \Re(\alpha)) \cos(\pi \Re(\beta))} a_M(\Re(\beta)) \left( \frac{\e^{\xi}}{\cosh \xi} \right)^M \right| \right. \\ &+ \left.\sum_{\ell = 1}^M \left| \frac{\cos(\pi \alpha)}{\cos(\pi \Re(\alpha))} a_\ell(\Re(\alpha)) a_{M-\ell}(\beta) \left( \frac{\e^{\xi}}{\sinh \xi} \right)^\ell \left( \frac{\e^{\xi}}{\cosh \xi} \right)^{M-\ell} \right|\right) \frac{\Gamma (\Re (2\nu  + \alpha  + \beta ) - M + 1)}{\left| \Gamma (2\nu  + \alpha  + \beta  + 2)\right|},
\end{split}
\end{gather}
provided $\left|\Re(\alpha)\right|<\frac{1}{2}$ and $\left|\Re(\beta)\right| < M+ \frac{1}{2}$, and
\begin{gather}\label{eq45}
\begin{split}
&\left|R_M^{(P2)} (\nu,\xi ,\alpha ,\beta )\right|  \le \\ & \sum_{\ell = 0}^M \left| \frac{\cos(\pi \alpha) \cos(\pi \beta)}{\cos(\pi \Re(\alpha)) \cos(\pi \Re(\beta))} a_\ell(\Re(\alpha)) a_{M-\ell}(\Re(\beta)) \left( \frac{\e^{\xi}}{\sinh \xi} \right)^\ell \left( \frac{\e^{\xi}}{\cosh \xi} \right)^{M-\ell} \right| \frac{\Gamma (\Re (2\nu  + \alpha  + \beta ) - M + 1)}{\left| \Gamma (2\nu  + \alpha  + \beta  + 2)\right|},
\end{split}
\end{gather}
provided $\left|\Re(\alpha)\right|, \left|\Re(\beta)\right| < \frac{1}{2}$. If $2\Re(\alpha)$ or $2\Re(\beta)$ is an odd integer, then the limiting values must to be taken in these bounds. The fractional powers are defined to be positive for positive real $\xi$ and are defined by continuity elsewhere. Furthermore, the remainder term $R_M^{(P2)} (\nu,\xi ,\alpha ,\beta )$ is bounded in absolute value by the corresponding first neglected term and retains the same sign, provided that $\xi$ is positive, $\nu$, $\alpha$, and $\beta$ are real, $|\alpha|, |\beta| < \frac{1}{2}$, and $2\nu + \alpha + \beta + 1 > M$.
\end{theorem}

\begin{theorem}\label{Qinfacthm} Let $\xi\in \mathcal{D}$ and let $N$ be an arbitrary non-negative integer. Let $\nu$, $\alpha$, and $\beta$ be complex numbers satisfying $\Re(2\nu + \alpha + \beta + 1) > N$. Then we have
\begin{gather}\label{Qinfac}
\begin{split}
&\frac{2^{2\nu + \alpha + \beta + 1}}{\pi} B(\nu + 1, \nu + \alpha + \beta + 1) Q_\nu^{(\alpha, \beta)}(\cosh(2\xi)) =\\ &\left( \frac{\e^{-\xi}}{\sinh \xi} \right)^{\alpha + 1/2} \left( \frac{\e^{-\xi}}{\cosh \xi} \right)^{\beta + 1/2} \e^{-2\nu \xi} \left( \sum_{n = 0}^{N - 1} g_n(\xi, \alpha, \beta) \frac{\Gamma(2\nu + \alpha + \beta - n + 1)}{\Gamma(2\nu + \alpha + \beta + 2)} + R_N^{(Q)}(\nu, \xi, \alpha, \beta) \right),
\end{split}
\end{gather}
where the remainder term satisfies the following estimates:
\begin{gather}\label{boundnum}
\begin{split}
&\left|R_N^{(Q)} (\nu,\xi ,\alpha ,\beta )\right|  \le  \left(\left| \frac{\cos(\pi \alpha) \cos(\pi \beta)}{\cos(\pi \Re(\alpha)) \cos(\pi \Re(\beta))} a_N(\Re(\alpha)) \left( \frac{\e^{\xi}}{\sinh \xi} \right)^N \right| \right. \\ &
+\left. \sum_{\ell = 0}^{N-1} \left| \frac{\cos(\pi \beta)}{\cos(\pi \Re(\beta))} a_\ell(\alpha) a_{N-\ell}(\Re(\beta)) \left( \frac{\e^{\xi}}{\sinh \xi} \right)^\ell \left( \frac{\e^{\xi}}{\cosh \xi} \right)^{N-\ell} \right|\right)\frac{\Gamma (\Re (2\nu  + \alpha  + \beta ) -N + 1)}{\left| \Gamma (2\nu  + \alpha  + \beta  + 2)\right|},
\end{split}
\end{gather}
provided $\left|\Re(\alpha)\right|<N+ \frac{1}{2}$ and $\left|\Re(\beta)\right| < \frac{1}{2}$,
\begin{align*}
&\left|R_N^{(Q)} (\nu,\xi ,\alpha ,\beta )\right|  \le  \left(\left| \frac{\cos(\pi \alpha) \cos(\pi \beta)}{\cos(\pi \Re(\alpha)) \cos(\pi \Re(\beta))} a_N(\Re(\beta)) \left( \frac{\e^{\xi}}{\cosh \xi} \right)^N \right| \right. \\ &+\left. \sum_{\ell = 1}^N \left| \frac{\cos(\pi \alpha)}{\cos(\pi \Re(\alpha))} a_\ell(\Re(\alpha)) a_{N-\ell}(\beta) \left( \frac{\e^{\xi}}{\sinh \xi} \right)^\ell \left( \frac{\e^{\xi}}{\cosh \xi} \right)^{N-\ell} \right|\right) \frac{\Gamma (\Re (2\nu  + \alpha  + \beta ) - N + 1)}{\left| \Gamma (2\nu  + \alpha  + \beta  + 2)\right|},
\end{align*}
provided $\left|\Re(\alpha)\right|<\frac{1}{2}$ and $\left|\Re(\beta)\right| < N+ \frac{1}{2}$, and
\begin{align*}
&\left|R_N^{(Q)} (\nu,\xi ,\alpha ,\beta )\right|  \le \\ & \sum_{\ell = 0}^N \left| \frac{\cos(\pi \alpha) \cos(\pi \beta)}{\cos(\pi \Re(\alpha)) \cos(\pi \Re(\beta))} a_\ell(\Re(\alpha)) a_{N-\ell}(\Re(\beta)) \left( \frac{\e^{\xi}}{\sinh \xi} \right)^\ell \left( \frac{\e^{\xi}}{\cosh \xi} \right)^{N-\ell} \right| \frac{\Gamma (\Re (2\nu  + \alpha  + \beta ) - N + 1)}{\left| \Gamma (2\nu  + \alpha  + \beta  + 2)\right|},
\end{align*}
provided $\left|\Re(\alpha)\right|, \left|\Re(\beta)\right| < \frac{1}{2}$. If $2\Re(\alpha)$ or $2\Re(\beta)$ is an odd integer, then the limiting values must to be taken in these bounds. The fractional powers are defined to be positive for positive real $\xi$ and are defined by continuity elsewhere. Furthermore, the remainder term $R_N^{(Q)} (\nu,\xi ,\alpha ,\beta )$ is bounded in absolute value by the corresponding first neglected term and retains the same sign, provided that $\xi$ is positive, $\nu$, $\alpha$, and $\beta$ are real, $|\alpha|, |\beta| < \frac{1}{2}$, and $2\nu + \alpha + \beta + 1 > N$.
\end{theorem}

The proofs of Theorems \ref{Pinfacthm} and \ref{Qinfacthm} are given in Section \ref{infproof1}.

The following theorem provides computable error bounds for the inverse factorial expansions \eqref{invfexp1}, \eqref{invfexp2}, and \eqref{invfexp3} of the Jacobi function $P_\nu^{(\alpha, \beta)}(x)$ and the associated functions $\SzQ_\nu^{(\alpha, \beta)}(x)$ and $\DQ_\nu^{(\alpha, \beta)}(x)$.

\begin{theorem}\label{cutinfacthm} Let $0 < \zeta  < \frac{\pi}{2}$ and let $N$ be an arbitrary non-negative integer. Let $\nu$, $\alpha$, and $\beta$ be complex numbers satisfying $\Re(2\nu + \alpha + \beta + 1) > N$. Then we have
\begin{gather}\label{Pinfac2}
\begin{split}
& 2^{2\nu + \alpha + \beta} B(\nu + 1,\nu + \alpha + \beta + 1)P_\nu^{(\alpha, \beta)} (\cos(2\zeta)) = \frac{1}{\sin^{\alpha + 1/2} \zeta \cos^{\beta + 1/2} \zeta} \\ & \times \left( \sum\limits_{n = 0}^{N - 1} \left(\sum\limits_{\ell = 0}^n \frac{a_\ell (\alpha) a_{n - \ell} (\beta)}{\sin^\ell \zeta \cos^{n- \ell} \zeta} \cos \left(\zeta^{(1)}_{\nu,n,\ell}\right)\right) \frac{\Gamma(2\nu + \alpha + \beta - n + 1)}{\Gamma(2\nu+\alpha + \beta+2)} + R_N^{(P)} (\nu,\zeta, \alpha, \beta) \right),
\end{split}
\end{gather}
\begin{gather}\label{Qinfac2}
\begin{split}
& \frac{2^{2\nu + \alpha + \beta + 1}}{\pi} B(\nu + 1, \nu + \alpha + \beta + 1) \SzQ_\nu^{(\alpha, \beta)} (\cos(2\zeta)) = \frac{1}{\sin^{\alpha + 1/2} \zeta \cos^{\beta + 1/2} \zeta} \\ & \times \left( \sum_{n = 0}^{N - 1} \left(\sum\limits_{\ell = 0}^n \frac{a_\ell (\alpha) a_{n - \ell} (\beta)}{\sin^\ell \zeta \cos^{n- \ell} \zeta} \cos \left(\zeta^{(2)}_{\nu,n,\ell}\right)\right) \frac{\Gamma(2\nu + \alpha + \beta - n + 1)}{\Gamma(2\nu + \alpha + \beta + 2)} + R_N^{(\SzQ)}(\nu, \zeta, \alpha, \beta) \right),
\end{split}
\end{gather}
\begin{gather}\label{Qinfac3}
\begin{split}
& \frac{2^{2\nu + \alpha + \beta + 1}}{\pi} B(\nu + 1, \nu + \alpha + \beta + 1) \DQ_\nu^{(\alpha, \beta)} (\cos(2\zeta)) = \frac{-1}{\sin^{\alpha + 1/2} \zeta \cos^{\beta + 1/2} \zeta}\\ &  \times \left( \sum_{n = 0}^{N - 1} \left(\sum\limits_{\ell = 0}^n \frac{a_\ell (\alpha) a_{n - \ell} (\beta)}{\sin^\ell \zeta \cos^{n- \ell} \zeta} \sin \left(\zeta^{(1)}_{\nu,n,\ell}\right)\right) \frac{\Gamma(2\nu + \alpha + \beta - n + 1)}{\Gamma(2\nu + \alpha + \beta + 2)} + R_N^{(\DQ)}(\nu, \zeta, \alpha, \beta) \right),
\end{split}
\end{gather}
where $\zeta^{(1)}_{\nu,n,\ell}$ and $\zeta^{(2)}_{\nu,n,\ell}$ are defined in \eqref{zetadef1}, and the remainder terms satisfy the following estimates:
\begin{gather}\label{cutbound1}
\begin{split}
& \left| R_N^{(P)} (\nu, \zeta, \alpha, \beta) \right|, \left|R_N^{(\DQ)}(\nu, \zeta, \alpha, \beta)\right| \le \left( \left| \frac{\cos (\pi \alpha) \cos (\pi \beta)}{\cos (\pi \Re(\alpha)) \cos (\pi \Re(\beta))} \frac{a_N (\Re(\alpha))}{\sin^N \zeta} \right| \cosh \left(\Im\left(\zeta^{(1)}_{\nu, N, N}\right)\right) \right. \\ &+\left. \sum_{\ell = 0}^{N - 1} \left| \frac{\cos (\pi \beta)}{\cos (\pi \Re(\beta))} \frac{a_\ell (\alpha) a_{N - \ell} (\Re(\beta))}{\sin^\ell \zeta \cos^{N - \ell} \zeta} \right| \cosh \left(\Im\left(\zeta^{(1)}_{\nu, N, \ell}\right)\right) \right) \frac{\Gamma (\Re(2\nu + \alpha + \beta) - N + 1)}{\left| \Gamma (2\nu + \alpha + \beta + 2) \right|},
\end{split}
\end{gather}
provided $\left|\Re(\alpha)\right|<N+ \frac{1}{2}$ and $\left|\Re(\beta)\right| < \frac{1}{2}$,
\begin{gather}\label{cutbound2}
\begin{split}
& \left| R_N^{(P)} (\nu, \zeta, \alpha, \beta) \right|, \left|R_N^{(\DQ)}(\nu, \zeta, \alpha, \beta)\right| \le \left( \left| \frac{\cos (\pi \alpha) \cos (\pi \beta)}{\cos (\pi \Re(\alpha)) \cos (\pi \Re(\beta))} \frac{a_N (\Re(\beta))}{\cos^N \zeta} \right| \cosh \left(\Im\left(\zeta^{(1)}_{\nu, N, 0}\right)\right) \right. \\ &+\left. \sum_{\ell = 1}^{N} \left| \frac{\cos (\pi \alpha)}{\cos (\pi \Re(\alpha))} \frac{a_\ell (\Re(\alpha)) a_{N - \ell} (\beta)}{\sin^\ell \zeta \cos^{N - \ell} \zeta} \right| \cosh \left(\Im\left(\zeta^{(1)}_{\nu, N, \ell}\right)\right) \right) \frac{\Gamma (\Re(2\nu + \alpha + \beta) - N + 1)}{\left| \Gamma (2\nu + \alpha + \beta + 2) \right|},
\end{split}
\end{gather}
provided $\left|\Re(\alpha)\right|<\frac{1}{2}$ and $\left|\Re(\beta)\right| < N+ \frac{1}{2}$, and
\begin{gather}\label{cutbound3}
\begin{split}
& \left| R_N^{(P)} (\nu, \zeta, \alpha, \beta) \right|, \left|R_N^{(\DQ)}(\nu, \zeta, \alpha, \beta)\right| \le \\ & \sum_{\ell = 0}^N \left| \frac{\cos (\pi \alpha) \cos (\pi \beta)}{\cos (\pi \Re(\alpha)) \cos (\pi \Re(\beta))} \frac{a_\ell (\Re(\alpha)) a_{N - \ell} (\Re(\beta))}{\sin^\ell \zeta \cos^{N - \ell} \zeta} \right| \cosh \left(\Im\left(\zeta^{(1)}_{\nu, N, \ell}\right)\right) \frac{\Gamma (\Re(2\nu + \alpha + \beta) - N + 1)}{\left| \Gamma (2\nu + \alpha + \beta + 2) \right|},
\end{split}
\end{gather}
provided $\left|\Re(\alpha)\right|, \left|\Re(\beta)\right| < \frac{1}{2}$, with analogous bounds for $R_N^{(\SzQ)}(\nu, \zeta, \alpha, \beta)$ in which each instance of $\zeta^{(1)}_{\nu, N, \ell}$ ($0\le \ell \le N$) is replaced by $\zeta^{(2)}_{\nu, N, \ell}$. If $2\Re(\alpha)$ or $2\Re(\beta)$ is an odd integer, then the limiting values must to be taken in these bounds.
\end{theorem}

The proof of Theorem \ref{cutinfacthm} is given in Section \ref{infproof2}.

We now turn to the error bounds for the factorial-type expansions. In the following two theorems, we provide computable bounds for the factorial expansions \eqref{Pfacasymp} and \eqref{Qfacasymp} of the Jacobi functions $P_\nu^{(\alpha, \beta)}(z)$ and $Q_\nu^{(\alpha, \beta)}(z)$, respectively. In these theorems, we use the following notation:
\begin{equation}\label{abprime}
\alpha' = \begin{cases}
   \alpha, & \text{if }\; \Re(\alpha) \ge \frac{1}{2}, \\
   \pm \alpha, & \text{if }\; -\frac{1}{2} < \Re(\alpha) < \frac{1}{2}, \\
   -\alpha, & \text{if }\; \Re(\alpha) \le -\frac{1}{2},\end{cases} \qquad \beta' = \begin{cases}
   \beta, & \text{if }\; \Re(\beta) \ge \frac{1}{2}, \\
   \pm \beta, & \text{if }\; -\frac{1}{2} < \Re(\beta) < \frac{1}{2}, \\
   -\beta, & \text{if }\; \Re(\beta) \le -\frac{1}{2}.
\end{cases}
\end{equation}
The symbol $\pm \alpha$ indicates that either $\alpha$ or $-\alpha$ can be chosen arbitrarily, with the same meaning applying to $\pm \beta$.

\begin{theorem}\label{Pfacthm} Let $\xi\in \mathcal{D}\setminus \mathbb{R}^+$ and let $N$, $M$ both be arbitrary non-negative integers. Let $\nu$, $\alpha$, and $\beta$ be complex numbers satisfying $\Re(2\nu + \alpha + \beta + 1) > \Re(\alpha' + \beta')$. Then we have
\begin{gather}\label{Pfac}
\begin{split}
 &\frac{\pi P_\nu^{(\alpha ,\beta )} (\cosh (2\xi ))}{2^{2\nu + \alpha  + \beta } B(\nu + \alpha  + 1,\nu + \beta  + 1)}
 = \\ & \left( \frac{\e^{ \xi }}{\sinh \xi} \right)^{\alpha  + 1/2} \left( \frac{\e^{ \xi }}{\cosh \xi} \right)^{\beta  + 1/2} \e^{2\nu \xi } \left(\sum_{n = 0}^{N - 1} (-1)^n g_n(\xi,\alpha,\beta)\frac{\Gamma (2\nu+\alpha+\beta+2)}{\Gamma(2\nu+\alpha+\beta+n+2)} + \widehat{R}_N^{(P1)} (\nu,\xi ,\alpha ,\beta ) \right)
\\ & - C(\xi ,\alpha ) \left( \frac{\e^{ - \xi }}{\sinh \xi} \right)^{\alpha  + 1/2} \left( \frac{\e^{ - \xi }}{\cosh \xi} \right)^{\beta  + 1/2} \e^{ - 2\nu \xi }\\ & \times \left( \sum_{m = 0}^{M - 1} (-1)^m g_m(-\xi,\alpha,\beta)\frac{\Gamma (2\nu+\alpha+\beta+2)}{\Gamma(2\nu+\alpha+\beta+m+2)} + \widehat{R}_M^{(P2)} (\nu,\xi ,\alpha ,\beta ) \right),
\end{split}
\end{gather}
where $C(\xi ,\alpha )$ is defined in \eqref{Cdef}, and the remainder terms satisfy the following estimates:
\begin{gather}\label{Pfacbound1}
\begin{split}
&\left| \widehat{R}_N^{(P_1)}(\nu, \xi, \alpha, \beta) \right| \le \frac{\Gamma( \Re(2\nu+\alpha+\beta-\alpha'-\beta') + 1)}{\left| \Gamma(2\nu+\alpha+\beta-\alpha'-\beta' + 1)\right|} \frac{\Gamma\left( \Re(\alpha') + \frac{1}{2} \right) \Gamma\left( \Re(\beta') + \frac{1}{2} \right)}{\left|\Gamma\left( \alpha' + \frac{1}{2} \right) \Gamma\left( \beta' + \frac{1}{2} \right) \right|}
\\ & \times \left| \frac{\cos(\pi \alpha) \cos(\pi \beta)}{\cos(\pi \Re(\alpha)) \cos(\pi \Re(\beta))} a_N(\Re(\alpha)) \left( \frac{\e^\xi}{\sinh \xi} \right)^N \right| \frac{\left| \Gamma(2\nu + \alpha + \beta + 2) \right|}{\Gamma(\Re(2\nu + \alpha + \beta) + N + 2)} 
\\ & \times \left| 1 - \e^{4\xi} \right| \left| \csc(2 \Im(\xi)) \right|  + \frac{\Gamma(\Re(2\nu+\alpha+\beta-\alpha'-\beta') + 1)}{\left| \Gamma(2\nu+\alpha+\beta-\alpha'-\beta' + 1)\right|} \frac{\Gamma\left( \Re(\beta') + \frac{1}{2} \right)}{\left|\Gamma\left( \beta' + \frac{1}{2} \right) \right|} \\ & \times \sum_{\ell = 0}^{N - 1} \left| \frac{\cos(\pi \beta)}{\cos(\pi \Re(\beta))} \frac{\Gamma\left( \Re(\alpha') + \ell + \frac{1}{2} \right)}{\Gamma\left( \alpha' + \ell + \frac{1}{2} \right)} a_\ell(\alpha) a_{N - \ell}(\Re(\beta)) \left( \frac{\e^\xi}{\sinh \xi} \right)^\ell \left( \frac{\e^\xi}{\cosh \xi} \right)^{N - \ell} \right| \\ & \times \frac{\left| \Gamma(2\nu + \alpha + \beta + 2) \right|}{\Gamma(\Re(2\nu + \alpha + \beta) + N + 2)} \times 
\begin{cases}
    \left| 1 + \e^{2\xi} \right| \left| \csc(2 \Im(\xi)) \right|, & \text{if }\; -1 < \Re\left(\e^{- 2\xi}\right) \le 0, \\ 
    \left| 1 + \e^{2\xi} \right|, & \text{if }\; 0 < \Re\left(\e^{- 2\xi}\right),
\end{cases}
\end{split}
\end{gather}
provided $ - \frac{1}{2} < \Re(\alpha) < N + \frac{1}{2}$ and $\left|\Re(\beta)\right| < \frac{1}{2}$,
\begin{gather}\label{Pfacbound2}
\begin{split}
&\left| \widehat{R}_N^{(P_1)}(\nu, \xi, \alpha, \beta) \right| \le \frac{\Gamma(\Re(2\nu+\alpha+\beta-\alpha'-\beta') + 1)}{\left| \Gamma(2\nu+\alpha+\beta-\alpha'-\beta' + 1)\right|} \frac{\Gamma\left( \Re(\alpha') + \frac{1}{2} \right) \Gamma\left( \Re(\beta') + \frac{1}{2} \right)}{\left|\Gamma\left( \alpha' + \frac{1}{2} \right) \Gamma\left( \beta' + \frac{1}{2} \right) \right|}
\\ & \times \left| \frac{\cos(\pi \alpha) \cos(\pi \beta)}{\cos(\pi \Re(\alpha)) \cos(\pi \Re(\beta))} a_N(\Re(\beta)) \left( \frac{\e^\xi}{\cosh \xi} \right)^N \right| \frac{\Gamma(2\nu + \alpha + \beta + 2)}{\Gamma(\Re(2\nu + \alpha + \beta) + N + 2)}
\\ & \times \left| 1 - \e^{4\xi} \right| \left| \csc(2 \Im(\xi)) \right| + \frac{\Gamma(\Re(2\nu+\alpha+\beta-\alpha'-\beta') + 1)}{\left| \Gamma(2\nu+\alpha+\beta-\alpha'-\beta' + 1)\right|} \frac{\Gamma\left( \Re(\alpha') + \frac{1}{2} \right)}{\left|\Gamma\left( \alpha' + \frac{1}{2} \right) \right|}
\\ & \times \sum_{\ell = 1}^N \left| \frac{\cos(\pi \alpha)}{\cos(\pi \Re(\alpha))} \frac{\Gamma\left( \Re(\beta') + N - \ell + \frac{1}{2} \right)}{\Gamma\left( \beta' + N - \ell + \frac{1}{2} \right)} a_\ell(\Re(\alpha)) a_{N - \ell}(\beta) \left( \frac{\e^\xi}{\sinh \xi} \right)^\ell \left( \frac{\e^\xi}{\cosh \xi} \right)^{N - \ell} \right| \\ & \times \frac{\left| \Gamma(2\nu + \alpha + \beta + 2) \right|}{\Gamma(\Re(2\nu + \alpha + \beta) + N + 2)} \times \begin{cases} 
    \left| 1 - \e^{2\xi} \right| \left| \csc(2 \Im(\xi)) \right|, & \text{if }\; 0 \le \Re\left(\e^{- 2\xi}\right) < 1, \\ 
    \left| 1 - \e^{2\xi} \right|, & \text{if }\; \Re\left(\e^{- 2\xi}\right) < 0,
\end{cases}
\end{split}
\end{gather}
provided $\left|\Re(\alpha)\right|< \frac{1}{2}$ and $- \frac{1}{2} < \Re(\beta) < N+\frac{1}{2}$, and
\begin{gather}\label{Pfacbound3}
\begin{split}
& \left| \widehat{R}_N^{(P_1)}(\nu, \xi, \alpha, \beta) \right| \le \frac{\Gamma(\Re(2\nu+\alpha+\beta-\alpha'-\beta') + 1)}{\left| \Gamma(2\nu+\alpha+\beta-\alpha'-\beta' + 1)\right|} \frac{\Gamma\left( \Re(\alpha') + \frac{1}{2} \right) \Gamma\left( \Re(\beta') + \frac{1}{2} \right)}{\left|\Gamma\left( \alpha' + \frac{1}{2} \right) \Gamma\left( \beta' + \frac{1}{2} \right) \right|}
\\ &\times \sum_{\ell = 0}^N \left| \frac{\cos(\pi \alpha) \cos(\pi \beta)}{\cos(\pi \Re(\alpha)) \cos(\pi \Re(\beta))} a_\ell(\Re(\alpha)) a_{N - \ell}(\Re(\beta)) \left( \frac{\e^\xi}{\sinh \xi} \right)^\ell \left( \frac{\e^\xi}{\cosh \xi} \right)^{N - \ell} \right| 
\\ & \times \frac{\left| \Gamma(2\nu + \alpha + \beta + 2) \right|}{\Gamma(\Re(2\nu + \alpha + \beta) + N + 2)} \times \left| 1 - \e^{4\xi} \right| \left| \csc(2 \Im(\xi)) \right|,
\end{split}
\end{gather}
provided $\left|\Re(\alpha)\right|, \left|\Re(\beta)\right| < \frac{1}{2}$. Similarly,
\begin{gather}\label{Pfacbound4}
\begin{split}
&\left| \widehat{R}_M^{(P_2)}(\nu, \xi, \alpha, \beta) \right| \le \frac{\Gamma(\Re(2\nu+\alpha+\beta-\alpha'-\beta') + 1)}{\left| \Gamma(2\nu+\alpha+\beta-\alpha'-\beta' + 1)\right|} \frac{\Gamma\left( \Re(\alpha') + \frac{1}{2} \right) \Gamma\left( \Re(\beta') + \frac{1}{2} \right)}{\left|\Gamma\left( \alpha' + \frac{1}{2} \right) \Gamma\left( \beta' + \frac{1}{2} \right) \right|}
\\ & \times \left| \frac{\cos(\pi \alpha) \cos(\pi \beta)}{\cos(\pi \Re(\alpha)) \cos(\pi \Re(\beta))} a_M(\Re(\alpha)) \left( \frac{\e^{-\xi}}{\sinh \xi} \right)^M \right| \frac{\left| \Gamma(2\nu + \alpha + \beta + 2) \right|}{\Gamma(\Re(2\nu + \alpha + \beta) + M + 2)} 
\\ & \times \begin{cases}
   \left| 1 \pm \e^{-2\xi} \right|, & \text{if }\; 1 \le \pm \Re \left( \e^{2\xi} \right), \\ 
   \left| 1 - \e^{-4\xi} \right| \left| \csc (2 \Im (\xi)) \right|, & \text{if }\; \left| \Re \left( \e^{2\xi} \right) \right| < 1 
\end{cases}  + \frac{\Gamma(\Re(2\nu+\alpha+\beta-\alpha'-\beta') + 1)}{\left| \Gamma(2\nu+\alpha+\beta-\alpha'-\beta' + 1)\right|} \frac{\Gamma\left( \Re(\beta') + \frac{1}{2} \right)}{\left|\Gamma\left( \beta' + \frac{1}{2} \right) \right|} \\ & \times \sum_{\ell = 0}^{M - 1} \left| \frac{\cos(\pi \beta)}{\cos(\pi \Re(\beta))} \frac{\Gamma\left( \Re(\alpha') + \ell + \frac{1}{2} \right)}{\Gamma\left( \alpha' + \ell + \frac{1}{2} \right)} a_\ell(\alpha) a_{M - \ell}(\Re(\beta)) \left( \frac{\e^{-\xi}}{\sinh \xi} \right)^\ell \left( \frac{\e^{-\xi}}{\cosh \xi} \right)^{M - \ell} \right| \\ & \times \frac{\left| \Gamma(2\nu + \alpha + \beta + 2) \right|}{\Gamma(\Re(2\nu + \alpha + \beta) + M + 2)} \times 
\begin{cases} 
    1, & \text{if }\; \Re\left(\e^{2\xi}\right) \le -1, \\ 
    \left| 1 + \e^{- 2\xi} \right| \left| \csc(2 \Im(\xi)) \right|, & \text{if }\; -1 < \Re\left(\e^{2\xi}\right) \le 0, \\ 
    \left| 1 + \e^{- 2\xi} \right|, & \text{if }\; 0 < \Re\left(\e^{2\xi}\right),
\end{cases}
\end{split}
\end{gather}
provided $ - \frac{1}{2} < \Re(\alpha) < M + \frac{1}{2}$ and $\left|\Re(\beta)\right| < \frac{1}{2}$,
\begin{gather}\label{Pfacbound5}
\begin{split}
&\left| \widehat{R}_M^{(P_2)}(\nu, \xi, \alpha, \beta) \right| \le \frac{\Gamma(\Re(2\nu+\alpha+\beta-\alpha'-\beta') + 1)}{\left| \Gamma(2\nu+\alpha+\beta-\alpha'-\beta' + 1)\right|} \frac{\Gamma\left( \Re(\alpha') + \frac{1}{2} \right) \Gamma\left( \Re(\beta') + \frac{1}{2} \right)}{\left|\Gamma\left( \alpha' + \frac{1}{2} \right) \Gamma\left( \beta' + \frac{1}{2} \right) \right|}
\\ & \times \left| \frac{\cos(\pi \alpha) \cos(\pi \beta)}{\cos(\pi \Re(\alpha)) \cos(\pi \Re(\beta))} a_M(\Re(\beta)) \left( \frac{\e^{-\xi}}{\cosh \xi} \right)^M \right| \frac{\Gamma(2\nu + \alpha + \beta + 2)}{\Gamma(\Re(2\nu + \alpha + \beta) + M + 2)}
\\ & \times \begin{cases}
   \left| 1 \pm \e^{-2\xi} \right|, & \text{if }\; 1 \le \pm \Re \left( \e^{2\xi} \right), \\ 
   \left| 1 - \e^{-4\xi} \right| \left| \csc (2 \Im (\xi)) \right|, & \text{if }\; \left| \Re \left( \e^{2\xi} \right) \right| < 1 
\end{cases} + \frac{\Gamma(\Re(2\nu+\alpha+\beta-\alpha'-\beta') + 1)}{\left| \Gamma(2\nu+\alpha+\beta-\alpha'-\beta' + 1)\right|} \frac{\Gamma\left( \Re(\alpha') + \frac{1}{2} \right)}{\left|\Gamma\left( \alpha' + \frac{1}{2} \right) \right|}
\\ & \times \sum_{\ell = 1}^M \left| \frac{\cos(\pi \alpha)}{\cos(\pi \Re(\alpha))} \frac{\Gamma\left( \Re(\beta') + M - \ell + \frac{1}{2} \right)}{\Gamma\left( \beta' + M - \ell + \frac{1}{2} \right)} a_\ell(\Re(\alpha)) a_{M - \ell}(\beta) \left( \frac{\e^{-\xi}}{\sinh \xi} \right)^\ell \left( \frac{\e^{-\xi}}{\cosh \xi} \right)^{M - \ell} \right| \\ & \times \frac{\left| \Gamma(2\nu + \alpha + \beta + 2) \right|}{\Gamma(\Re(2\nu + \alpha + \beta) + M + 2)} \times 
\begin{cases} 
    1, & \text{if }\; 1 \le \Re\left(\e^{2\xi}\right), \\ 
    \left| 1 - \e^{- 2\xi} \right| \left| \csc(2 \Im(\xi)) \right|, & \text{if }\; 0 \le \Re\left(\e^{2\xi}\right) < 1, \\ 
    \left| 1 - \e^{- 2\xi} \right|, & \text{if }\; \Re\left(\e^{2\xi}\right) < 0,
\end{cases}
\end{split}
\end{gather}
provided $\left|\Re(\alpha)\right|< \frac{1}{2}$ and $- \frac{1}{2} < \Re(\beta) < M+\frac{1}{2}$, and
\begin{gather}\label{Pfacbound6}
\begin{split}
& \left| \widehat{R}_M^{(P_2)}(\nu, \xi, \alpha, \beta) \right| \le \frac{\Gamma(\Re(2\nu+\alpha+\beta-\alpha'-\beta') + 1)}{\left| \Gamma(2\nu+\alpha+\beta-\alpha'-\beta' + 1)\right|} \frac{\Gamma\left( \Re(\alpha') + \frac{1}{2} \right) \Gamma\left( \Re(\beta') + \frac{1}{2} \right)}{\left|\Gamma\left( \alpha' + \frac{1}{2} \right) \Gamma\left( \beta' + \frac{1}{2} \right) \right|}
\\ &\times \sum_{\ell = 0}^M \left| \frac{\cos(\pi \alpha) \cos(\pi \beta)}{\cos(\pi \Re(\alpha)) \cos(\pi \Re(\beta))} a_\ell(\Re(\alpha)) a_{M - \ell}(\Re(\beta)) \left( \frac{\e^{-\xi}}{\sinh \xi} \right)^\ell \left( \frac{\e^{-\xi}}{\cosh \xi} \right)^{M - \ell} \right| 
\\ & \times \frac{\left| \Gamma(2\nu + \alpha + \beta + 2) \right|}{\Gamma(\Re(2\nu + \alpha + \beta) + M + 2)} \times \begin{cases}
   \left| 1 \pm \e^{-2\xi} \right|, & \text{if }\; 1 \le \pm \Re \left( \e^{2\xi} \right), \\ 
   \left| 1 - \e^{-4\xi} \right| \left| \csc (2 \Im (\xi)) \right|, & \text{if }\; \left| \Re \left( \e^{2\xi} \right) \right| < 1,
\end{cases}
\end{split}
\end{gather}
provided $\left|\Re(\alpha)\right|, \left|\Re(\beta)\right| < \frac{1}{2}$. If $2\Re(\alpha)$ or $2\Re(\beta)$ is an odd integer, then the limiting values must to be taken in these bounds. The fractional powers are defined to be positive for positive real $\xi$ and are defined by continuity elsewhere.
\end{theorem}

\begin{theorem}\label{Qfacthm} Let $\xi\in \mathcal{D}$ and let $N$ be an arbitrary non-negative integer. Let $\nu$, $\alpha$, and $\beta$ be complex numbers satisfying $\Re(2\nu + \alpha + \beta + 1) > \Re(\alpha' + \beta')$. Then we have
\begin{gather}\label{Qfac}
\begin{split}
 &\frac{Q_\nu^{(\alpha ,\beta )} (\cosh (2\xi ))}{2^{2\nu + \alpha  + \beta } B(\nu + \alpha  + 1,\nu + \beta  + 1)}
 = \\ & \left( \frac{\e^{ - \xi }}{\sinh \xi} \right)^{\alpha  + 1/2} \left( \frac{\e^{ - \xi }}{\cosh \xi} \right)^{\beta  + 1/2} \e^{ - 2\nu \xi } \left( \sum_{n= 0}^{N - 1} (-1)^n g_n(-\xi,\alpha,\beta)\frac{\Gamma (2\nu+\alpha+\beta+2)}{\Gamma(2\nu+\alpha+\beta+n+2)} + \widehat{R}_N^{(Q)} (\nu,\xi ,\alpha ,\beta ) \right),
\end{split}
\end{gather}
where the remainder term satisfies the following estimates:
\begin{align*}
&\left| \widehat{R}_N^{(Q)}(\nu, \xi, \alpha, \beta) \right| \le \frac{\Gamma(\Re(2\nu+\alpha+\beta-\alpha'-\beta') + 1)}{\left| \Gamma(2\nu+\alpha+\beta-\alpha'-\beta' + 1)\right|} \frac{\Gamma\left( \Re(\alpha') + \frac{1}{2} \right) \Gamma\left( \Re(\beta') + \frac{1}{2} \right)}{\left|\Gamma\left( \alpha' + \frac{1}{2} \right) \Gamma\left( \beta' + \frac{1}{2} \right) \right|}
\\ & \times \left| \frac{\cos(\pi \alpha) \cos(\pi \beta)}{\cos(\pi \Re(\alpha)) \cos(\pi \Re(\beta))} a_N(\Re(\alpha)) \left( \frac{\e^{-\xi}}{\sinh \xi} \right)^N \right| \frac{\left| \Gamma(2\nu + \alpha + \beta + 2) \right|}{\Gamma(\Re(2\nu + \alpha + \beta) + N + 2)} 
\\ & \times \begin{cases}
   \left| 1 \pm \e^{-2\xi} \right|, & \text{if }\; 1 \le \pm \Re \left( \e^{2\xi} \right), \\ 
   \left| 1 - \e^{-4\xi} \right| \left| \csc (2 \Im (\xi)) \right|, & \text{if }\; \left| \Re \left( \e^{2\xi} \right) \right| < 1 
\end{cases}  + \frac{\Gamma(\Re(2\nu+\alpha+\beta-\alpha'-\beta') + 1)}{\left| \Gamma(2\nu+\alpha+\beta-\alpha'-\beta' + 1)\right|} \frac{\Gamma\left( \Re(\beta') + \frac{1}{2} \right)}{\left|\Gamma\left( \beta' + \frac{1}{2} \right) \right|} \\ & \times \sum_{\ell = 0}^{N - 1} \left| \frac{\cos(\pi \beta)}{\cos(\pi \Re(\beta))} \frac{\Gamma\left( \Re(\alpha') + \ell + \frac{1}{2} \right)}{\Gamma\left( \alpha' + \ell + \frac{1}{2} \right)} a_\ell(\alpha) a_{N - \ell}(\Re(\beta)) \left( \frac{\e^{-\xi}}{\sinh \xi} \right)^\ell \left( \frac{\e^{-\xi}}{\cosh \xi} \right)^{N - \ell} \right| \\ & \times \frac{\left| \Gamma(2\nu + \alpha + \beta + 2) \right|}{\Gamma(\Re(2\nu + \alpha + \beta) + N + 2)} \times 
\begin{cases} 
    1, & \text{if }\; \Re\left(\e^{2\xi}\right) \le -1, \\ 
    \left| 1 + \e^{- 2\xi} \right| \left| \csc(2 \Im(\xi)) \right|, & \text{if }\; -1 < \Re\left(\e^{2\xi}\right) \le 0, \\ 
    \left| 1 + \e^{- 2\xi} \right|, & \text{if }\; 0 < \Re\left(\e^{2\xi}\right),
\end{cases}
\end{align*}
provided $ - \frac{1}{2} < \Re(\alpha) <N + \frac{1}{2}$ and $\left|\Re(\beta)\right| < \frac{1}{2}$,
\begin{align*}
&\left| \widehat{R}_N^{(Q)}(\nu, \xi, \alpha, \beta) \right| \le \frac{\Gamma(\Re(2\nu+\alpha+\beta-\alpha'-\beta') + 1)}{\left| \Gamma(2\nu+\alpha+\beta-\alpha'-\beta' + 1)\right|} \frac{\Gamma\left( \Re(\alpha') + \frac{1}{2} \right) \Gamma\left( \Re(\beta') + \frac{1}{2} \right)}{\left|\Gamma\left( \alpha' + \frac{1}{2} \right) \Gamma\left( \beta' + \frac{1}{2} \right) \right|}
\\ & \times \left| \frac{\cos(\pi \alpha) \cos(\pi \beta)}{\cos(\pi \Re(\alpha)) \cos(\pi \Re(\beta))} a_N(\Re(\beta)) \left( \frac{\e^{-\xi}}{\cosh \xi} \right)^N \right| \frac{\Gamma(2\nu + \alpha + \beta + 2)}{\Gamma(\Re(2\nu + \alpha + \beta) + N + 2)}
\\ & \times \begin{cases}
   \left| 1 \pm \e^{-2\xi} \right|, & \text{if }\; 1 \le \pm \Re \left( \e^{2\xi} \right), \\ 
   \left| 1 - \e^{-4\xi} \right| \left| \csc (2 \Im (\xi)) \right|, & \text{if }\; \left| \Re \left( \e^{2\xi} \right) \right| < 1 
\end{cases} + \frac{\Gamma(\Re(2\nu+\alpha+\beta-\alpha'-\beta') + 1)}{\left| \Gamma(2\nu+\alpha+\beta-\alpha'-\beta' + 1)\right|} \frac{\Gamma\left( \Re(\alpha') + \frac{1}{2} \right)}{\left|\Gamma\left( \alpha' + \frac{1}{2} \right) \right|}
\\ & \times \sum_{\ell = 1}^N \left| \frac{\cos(\pi \alpha)}{\cos(\pi \Re(\alpha))} \frac{\Gamma\left( \Re(\beta') + N - \ell + \frac{1}{2} \right)}{\Gamma\left( \beta' + N - \ell + \frac{1}{2} \right)} a_\ell(\Re(\alpha)) a_{N - \ell}(\beta) \left( \frac{\e^{-\xi}}{\sinh \xi} \right)^\ell \left( \frac{\e^{-\xi}}{\cosh \xi} \right)^{N - \ell} \right| \\ & \times \frac{\left| \Gamma(2\nu + \alpha + \beta + 2) \right|}{\Gamma(\Re(2\nu + \alpha + \beta) + N + 2)} \times 
\begin{cases} 
    1, & \text{if }\; 1 \le \Re\left(\e^{2\xi}\right), \\ 
    \left| 1 - \e^{- 2\xi} \right| \left| \csc(2 \Im(\xi)) \right|, & \text{if }\; 0 \le \Re\left(\e^{2\xi}\right) < 1, \\ 
    \left| 1 - \e^{- 2\xi} \right|, & \text{if }\; \Re\left(\e^{2\xi}\right) < 0,
\end{cases}
\end{align*}
provided $\left|\Re(\alpha)\right|< \frac{1}{2}$ and $- \frac{1}{2} < \Re(\beta) < N+\frac{1}{2}$, and
\begin{gather}\label{numericsbound2}
\begin{split}
& \left| \widehat{R}_N^{(Q)}(\nu, \xi, \alpha, \beta) \right| \le \frac{\Gamma(\Re(2\nu+\alpha+\beta-\alpha'-\beta') + 1)}{\left| \Gamma(2\nu+\alpha+\beta-\alpha'-\beta' + 1)\right|} \frac{\Gamma\left( \Re(\alpha') + \frac{1}{2} \right) \Gamma\left( \Re(\beta') + \frac{1}{2} \right)}{\left|\Gamma\left( \alpha' + \frac{1}{2} \right) \Gamma\left( \beta' + \frac{1}{2} \right) \right|}
\\ &\times \sum_{\ell = 0}^N \left| \frac{\cos(\pi \alpha) \cos(\pi \beta)}{\cos(\pi \Re(\alpha)) \cos(\pi \Re(\beta))} a_\ell(\Re(\alpha)) a_{N - \ell}(\Re(\beta)) \left( \frac{\e^{-\xi}}{\sinh \xi} \right)^\ell \left( \frac{\e^{-\xi}}{\cosh \xi} \right)^{N - \ell} \right| 
\\ & \times \frac{\left| \Gamma(2\nu + \alpha + \beta + 2) \right|}{\Gamma(\Re(2\nu + \alpha + \beta) +N + 2)} \times \begin{cases}
   \left| 1 \pm \e^{-2\xi} \right|, & \text{if }\; 1 \le \pm \Re \left( \e^{2\xi} \right), \\ 
   \left| 1 - \e^{-4\xi} \right| \left| \csc (2 \Im (\xi)) \right|, & \text{if }\; \left| \Re \left( \e^{2\xi} \right) \right| < 1,
\end{cases}
\end{split}
\end{gather}
provided $\left|\Re(\alpha)\right|, \left|\Re(\beta)\right| < \frac{1}{2}$. If $2\Re(\alpha)$ or $2\Re(\beta)$ is an odd integer, then the limiting values must to be taken in these bounds. The fractional powers are defined to be positive for positive real $\xi$ and are defined by continuity elsewhere. Furthermore, the remainder term $\widehat{R}_N^{(Q)} (\nu,\xi ,\alpha ,\beta )$ is bounded in absolute value by twice the corresponding first neglected term and retains the same sign, provided that $\xi$ is positive, $\nu$, $\alpha$, and $\beta$ are real, $|\alpha|, |\beta| < \frac{1}{2}$, and $2\nu + \alpha + \beta + 1  > \alpha' + \beta'$.
\end{theorem}

The proofs of Theorems \ref{Pfacthm} and \ref{Qfacthm} are provided in Section \ref{fproof1}.

In the following two theorems, we present alternative, less explicit error bounds for the factorial expansions \eqref{Pfacasymp} and \eqref{Qfacasymp} of the Jacobi functions $P_\nu^{(\alpha, \beta)}(z)$ and $Q_\nu^{(\alpha, \beta)}(z)$. In these theorems, the function $\mathsf{g}(w)$ is defined as follows:
\begin{equation}\label{gdef}
\mathsf{g}(w)=\begin{cases} 1  & \text{if }\;
\Re(w) \le  0, \\ \left|\frac{w}{\Im (w)}\right| & \text{if }\;
0<\Re (w) \le  \left| w \right|^2, \\  \frac{1}{|1-w|} & \text{if }\;  \Re(w)> \left| w \right|^2,\end{cases}
\end{equation}
for $w\in \mathbb{C}\setminus [1,+\infty)$ (cf. \cite[Eq. (5.5)]{Hahn1980}). In the special case where $\nu$, $\alpha$, and $\beta$ are real and $|\alpha|, |\beta| < \frac{1}{2}$, Theorem \ref{Pfacthm2} was also proved by Hahn \cite{Hahn1980}.

\begin{theorem}\label{Pfacthm2} Let $\xi\in \mathcal{D}\setminus \mathbb{R}^+$ and let $N$, $M$ both be arbitrary non-negative integers. Suppose that $\nu$, $\alpha$, and $\beta$ are complex numbers satisfying $\Re(2\nu + \alpha + \beta + 1) > \Re(\alpha' + \beta')$, and that either $0 < \Re(\alpha' + \beta')$ or $\alpha = \beta = 0$. Then the remainder terms in \eqref{Pfac} satisfy the following estimates:
\begin{gather}\label{Hahn1}
\begin{split}
& \left|\widehat{R}_N^{(P_1)}(\nu, \xi, \alpha, \beta)\right| \le \frac{\Gamma(\Re(2\nu + \alpha + \beta - \alpha' - \beta') + 1)}{\left| \Gamma(2\nu + \alpha + \beta - \alpha' - \beta' + 1) \right|}
\frac{\left| \Gamma(-\alpha' - \beta' + N + 1) \right|}{\Gamma(\Re(-\alpha' - \beta') + N + 1)} \\ 
&  \times \sum_{\ell = 0}^N \left| \frac{\cos(\pi \alpha) \cos(\pi \beta)}{\cos(\pi \Re(\alpha)) \cos(\pi \Re(\beta))}  a_\ell (\Re(\alpha)) a_{N - \ell} (\Re(\beta)) \left( \frac{\e^\xi}{\sinh \xi} \right)^\ell \left( \frac{\e^\xi}{\cosh \xi} \right)^{N - \ell} \right| \\ 
&  \times \frac{\left|\Gamma(2\nu + \alpha + \beta + 2)\right|}{\Gamma(\Re(2\nu + \alpha + \beta) + N + 2)} \times \max_{0 \le u \le \frac{\pi}{2}} \mathsf{g}\left( \frac{\sin^2 u}{1 + \sin(2u)} \frac{\e^\xi}{2 \sinh \xi} + \frac{\cos^2 u}{1 + \sin(2u)} \frac{\e^\xi}{2 \cosh \xi} \right),
\end{split}
\end{gather}
and
\begin{gather}\label{Hahn2}
\begin{split}
& \left|\widehat{R}_M^{(P_2)}(\nu, \xi, \alpha, \beta)\right| \le \frac{\Gamma(\Re(2\nu + \alpha + \beta - \alpha' - \beta') + 1)}{\left| \Gamma(2\nu + \alpha + \beta - \alpha' - \beta' + 1) \right|}
\frac{\left| \Gamma(-\alpha' - \beta' + M + 1) \right|}{\Gamma(\Re(-\alpha' - \beta') + M + 1)} \\ 
&  \times \sum_{\ell = 0}^M \left| \frac{\cos(\pi \alpha) \cos(\pi \beta)}{\cos(\pi \Re(\alpha)) \cos(\pi \Re(\beta))}  a_\ell (\Re(\alpha)) a_{M - \ell} (\Re(\beta)) \left( \frac{\e^{-\xi}}{\sinh \xi} \right)^\ell \left( \frac{\e^{-\xi}}{\cosh \xi} \right)^{M - \ell} \right| \\ 
&  \times \frac{\left|\Gamma(2\nu + \alpha + \beta + 2)\right|}{\Gamma(\Re(2\nu + \alpha + \beta) + M + 2)} \times \max_{0 \le u \le \frac{\pi}{2}} \mathsf{g}\left(-\frac{\sin^2 u}{1 + \sin(2u)} \frac{\e^{-\xi}}{2 \sinh \xi} + \frac{\cos^2 u}{1 + \sin(2u)} \frac{\e^{-\xi}}{2 \cosh \xi} \right),
\end{split}
\end{gather}
provided that $\left|\Re(\alpha)\right|, \left|\Re(\beta)\right| < \frac{1}{2}$. The fractional powers are defined to be positive for positive real $\xi$ and are defined by continuity elsewhere.
\end{theorem}

\begin{theorem}\label{Qfacthm2} Let $\xi\in \mathcal{D}$ and let $N$ be an arbitrary non-negative integer. Suppose that $\nu$, $\alpha$, and $\beta$ are complex numbers satisfying $\Re(2\nu + \alpha + \beta + 1) > \Re(\alpha' + \beta')$, and that either $0 < \Re(\alpha' + \beta')$ or $\alpha = \beta = 0$. Then the remainder term in \eqref{Qfac} satisfies the following estimate:
\begin{gather}\label{Hahn3}
\begin{split}
& \left|\widehat{R}_N^{(Q)}(\nu, \xi, \alpha, \beta)\right| \le \frac{\Gamma(\Re(2\nu + \alpha + \beta - \alpha' - \beta') + 1)}{\left| \Gamma(2\nu + \alpha + \beta - \alpha' - \beta' + 1) \right|}
\frac{\left| \Gamma(-\alpha' - \beta' + N + 1) \right|}{\Gamma(\Re(-\alpha' - \beta') + N + 1)} \\ 
&  \times \sum_{\ell = 0}^N \left| \frac{\cos(\pi \alpha) \cos(\pi \beta)}{\cos(\pi \Re(\alpha)) \cos(\pi \Re(\beta))}  a_\ell (\Re(\alpha)) a_{N - \ell} (\Re(\beta)) \left( \frac{\e^{-\xi}}{\sinh \xi} \right)^\ell \left( \frac{\e^{-\xi}}{\cosh \xi} \right)^{N - \ell} \right| \\ 
&  \times \frac{\left|\Gamma(2\nu + \alpha + \beta + 2)\right|}{\Gamma(\Re(2\nu + \alpha + \beta) + N + 2)} \times \max_{0 \le u \le \frac{\pi}{2}} \mathsf{g}\left(-\frac{\sin^2 u}{1 + \sin(2u)} \frac{\e^{-\xi}}{2 \sinh \xi} + \frac{\cos^2 u}{1 + \sin(2u)} \frac{\e^{-\xi}}{2 \cosh \xi} \right),
\end{split}
\end{gather}
provided that $\left|\Re(\alpha)\right|, \left|\Re(\beta)\right| < \frac{1}{2}$. The fractional powers are defined to be positive for positive real $\xi$ and are defined by continuity elsewhere.
\end{theorem}

The proofs of Theorems \ref{Pfacthm2} and \ref{Qfacthm2} are provided in Section \ref{fproof2}.

Our final theorem provides computable error bounds for the factorial expansions \eqref{Pfacasymp2}, \eqref{Qfacasymp2}, and \eqref{Qfacasymp3} of the Jacobi function $P_\nu^{(\alpha, \beta)}(x)$ and the associated functions $\SzQ_\nu^{(\alpha, \beta)}(x)$ and $\DQ_\nu^{(\alpha, \beta)}(x)$. In the special case where $\nu$, $\alpha$, and $\beta$ are real and $|\alpha|, |\beta| < \frac{1}{2}$, the error bound \eqref{cutfacbound4} for $\widehat{R}_N^{(P)}(\nu, \zeta, \alpha, \beta)$ was also proved by Hahn \cite{Hahn1980}.

\begin{theorem}\label{cutinfacthm2} Let $\xi\in \mathcal{D}$ and let $N$ be an arbitrary non-negative integer. Let $\nu$, $\alpha$, and $\beta$ be complex numbers satisfying $\Re(2\nu + \alpha + \beta + 1) > \Re(\alpha' + \beta')$. Then we have
\begin{gather}\label{Pfac2}
\begin{split}
& \frac{ \pi P_\nu ^{(\alpha ,\beta )} (\cos (2\zeta) )}{2^{2\nu  + \alpha  + \beta +1} B(\nu  + \alpha  + 1,\nu  + \beta  + 1)} = \frac{1}{\sin ^{\alpha  + 1/2} \zeta \cos ^{\beta  + 1/2} \zeta }\\& 
\times\left(\sum_{n = 0}^{N-1}  ( - 1)^n \left(\sum\limits_{\ell = 0}^n \frac{a_\ell (\alpha) a_{n - \ell} (\beta)}{\sin^\ell \zeta \cos^{n- \ell} \zeta} \cos \left(\zeta^{(3)}_{\nu ,n,\ell }\right)\right)\frac{\Gamma (2\nu  + \alpha  + \beta  + 2)}{\Gamma (2\nu  + \alpha  + \beta  + n + 2)} +\widehat{R}_N^{(P)}(\nu, \zeta, \alpha, \beta)\right),
\end{split}
\end{gather}
\begin{gather}\label{Qfac2}
\begin{split}
& \frac{ \SzQ_\nu ^{(\alpha ,\beta )} (\cos (2\zeta) )}{2^{2\nu  + \alpha  + \beta } B(\nu  + \alpha  + 1,\nu  + \beta  + 1)} = \frac{1}{\sin ^{\alpha  + 1/2} \zeta \cos ^{\beta  + 1/2} \zeta } \\& 
\times\left(\sum_{n = 0}^{N-1} ( - 1)^n \left(\sum\limits_{\ell = 0}^n \frac{a_\ell (\alpha) a_{n - \ell} (\beta)}{\sin^\ell \zeta \cos^{n- \ell} \zeta} \cos \left(\zeta^{(4)}_{\nu ,n,\ell }\right)\right)\frac{\Gamma (2\nu  + \alpha  + \beta  + 2)}{\Gamma (2\nu  + \alpha  + \beta  + n + 2)} +\widehat{R}_N^{(\SzQ)}(\nu, \zeta, \alpha, \beta)\right),
\end{split}
\end{gather}
\begin{gather}\label{Qfac3}
\begin{split}
& \frac{ \DQ_\nu ^{(\alpha ,\beta )} (\cos (2\zeta) )}{2^{2\nu  + \alpha  + \beta } B(\nu  + \alpha  + 1,\nu  + \beta  + 1)} = \frac{-1}{\sin ^{\alpha  + 1/2} \zeta \cos ^{\beta  + 1/2} \zeta } \\& 
\times\left(\sum_{n = 0}^{N-1} ( - 1)^n \left(\sum\limits_{\ell = 0}^n \frac{a_\ell (\alpha) a_{n - \ell} (\beta)}{\sin^\ell \zeta \cos^{n- \ell} \zeta} \sin \left(\zeta^{(3)}_{\nu ,n,\ell }\right)\right)\frac{\Gamma (2\nu  + \alpha  + \beta  + 2)}{\Gamma (2\nu  + \alpha  + \beta  + n + 2)} +\widehat{R}_N^{(\DQ)}(\nu, \zeta, \alpha, \beta)\right),
\end{split}
\end{gather}
where $\zeta^{(3)}_{\nu,n,\ell}$ and $\zeta^{(4)}_{\nu,n,\ell}$ are defined in \eqref{zetadef2}, and the remainder terms satisfy the following estimates:
\begin{gather}\label{cutfacbound1}
\begin{split}
& \left|\widehat{R}_N^{(P)}(\nu, \zeta, \alpha, \beta)\right|,\left|\widehat{R}_N^{(\DQ)}(\nu, \zeta, \alpha, \beta)\right| \le 2 \frac{\Gamma \left( \Re(2\nu + \alpha + \beta - \alpha' - \beta') + 1 \right)}{\left| \Gamma \left( 2\nu + \alpha + \beta - \alpha' - \beta' + 1 \right) \right|}
\frac{\Gamma \left( \Re(\alpha') + \frac{1}{2} \right) \Gamma \left( \Re(\beta') + \frac{1}{2} \right)}{\left| \Gamma \left( \alpha' + \frac{1}{2} \right) \Gamma \left( \beta' + \frac{1}{2} \right) \right|}
\\ &\times \left| \frac{\cos(\pi \alpha) \cos(\pi \beta)}{\cos(\pi \Re(\alpha)) \cos(\pi \Re(\beta))} \frac{a_N(\Re(\alpha))}{\sin^N \zeta} \right| \cosh \left(\Im\left( \zeta_{\nu, N, N}^{(3)} \right)\right) 
\frac{\left| \Gamma \left( 2\nu + \alpha + \beta + 2 \right) \right|}{\Gamma \left( \Re(2\nu + \alpha + \beta) + N + 2 \right)}
\\ &+ \frac{\Gamma \left( \Re(2\nu + \alpha + \beta - \alpha' - \beta') + 1 \right)}{\left| \Gamma \left( 2\nu + \alpha + \beta - \alpha' - \beta' + 1 \right) \right|}
\frac{\Gamma \left( \Re(\beta') + \frac{1}{2} \right)}{\left| \Gamma \left( \beta' + \frac{1}{2} \right) \right|}
\\ &\times \sum_{\ell = 0}^{N - 1} \left| \frac{\cos(\pi \beta)}{\cos(\pi \Re(\beta))} 
\frac{\Gamma \left( \Re(\alpha') + \ell + \frac{1}{2} \right)}{\Gamma \left( \alpha' + \ell + \frac{1}{2} \right)} 
\frac{a_\ell(\alpha) a_{N - \ell}(\Re(\beta))}{\sin^\ell \zeta \cos^{N - \ell} \zeta} \right|
\cosh \left(\Im\left( \zeta_{\nu, N, \ell}^{(3)}\right) \right) 
\frac{\left| \Gamma \left( 2\nu + \alpha + \beta + 2 \right) \right|}{\Gamma \left( \Re(2\nu + \alpha + \beta) + N + 2 \right)}
\\ &\times 
\begin{cases} 
     2\cos \zeta, & \text{if }\; 0 < \zeta  < \frac{\pi }{4}, \\ 
    \csc \zeta, & \text{if }\; \frac{\pi }{4} \le \zeta  < \frac{\pi }{2},
\end{cases}
\end{split}
\end{gather}
provided $ - \frac{1}{2} < \Re(\alpha) <N + \frac{1}{2}$ and $\left|\Re(\beta)\right| < \frac{1}{2}$,
\begin{gather}\label{cutfacbound2}
\begin{split}
& \left|\widehat{R}_N^{(P)}(\nu, \zeta, \alpha, \beta)\right|,\left|\widehat{R}_N^{(\DQ)}(\nu, \zeta, \alpha, \beta)\right| \le 2 \frac{\Gamma \left( \Re(2\nu + \alpha + \beta - \alpha' - \beta') + 1 \right)}{\left| \Gamma \left( 2\nu + \alpha + \beta - \alpha' - \beta' + 1 \right) \right|}
\frac{\Gamma \left( \Re(\alpha') + \frac{1}{2} \right) \Gamma \left( \Re(\beta') + \frac{1}{2} \right)}{\left| \Gamma \left( \alpha' + \frac{1}{2} \right) \Gamma \left( \beta' + \frac{1}{2} \right) \right|}
\\ &\times \left| \frac{\cos(\pi \alpha) \cos(\pi \beta)}{\cos(\pi \Re(\alpha)) \cos(\pi \Re(\beta))} \frac{a_N(\Re(\beta))}{\cos^N \zeta} \right| \cosh \left(\Im\left( \zeta_{\nu, N, 0}^{(3)} \right)\right) 
\frac{\left| \Gamma \left( 2\nu + \alpha + \beta + 2 \right) \right|}{\Gamma \left( \Re(2\nu + \alpha + \beta) + N + 2 \right)}
\\ &+ \frac{\Gamma \left( \Re(2\nu + \alpha + \beta - \alpha' - \beta') + 1 \right)}{\left| \Gamma \left( 2\nu + \alpha + \beta - \alpha' - \beta' + 1 \right) \right|}
\frac{\Gamma \left( \Re(\alpha') + \frac{1}{2} \right)}{\left| \Gamma \left( \alpha' + \frac{1}{2} \right) \right|}
\\ &\times \sum_{\ell = 1}^{N} \left| \frac{\cos(\pi \alpha)}{\cos(\pi \Re(\alpha))} 
\frac{\Gamma \left( \Re(\beta') + \ell + \frac{1}{2} \right)}{\Gamma \left( \beta' + \ell + \frac{1}{2} \right)} 
\frac{a_\ell(\Re(\alpha)) a_{N - \ell}(\beta)}{\sin^\ell \zeta \cos^{N - \ell} \zeta} \right|
\cosh \left(\Im\left( \zeta_{\nu, N, \ell}^{(3)} \right)\right) 
\frac{\left| \Gamma \left( 2\nu + \alpha + \beta + 2 \right) \right|}{\Gamma \left( \Re(2\nu + \alpha + \beta) + N + 2 \right)}
\\ &\times \begin{cases} 
     \sec\zeta, & \text{if }\; 0 < \zeta  < \frac{\pi }{4}, \\ 
    2\sin \zeta, & \text{if }\; \frac{\pi }{4} \le \zeta  < \frac{\pi }{2},
\end{cases}
\end{split}
\end{gather}
provided $\left|\Re(\alpha)\right|< \frac{1}{2}$ and $- \frac{1}{2} < \Re(\beta) < N+\frac{1}{2}$,
\begin{gather}\label{cutfacbound3}
\begin{split}
& \left|\widehat{R}_N^{(P)}(\nu, \zeta, \alpha, \beta)\right|,\left|\widehat{R}_N^{(\DQ)}(\nu, \zeta, \alpha, \beta)\right|\leq 2 \frac{\Gamma \left( \Re(2\nu + \alpha + \beta - \alpha' - \beta') + 1 \right)}{\left| \Gamma \left( 2\nu + \alpha + \beta - \alpha' - \beta' + 1 \right) \right|}
\frac{\Gamma \left( \Re(\alpha') + \frac{1}{2} \right) \Gamma \left( \Re(\beta') + \frac{1}{2} \right)}{\left| \Gamma \left( \alpha' + \frac{1}{2} \right) \Gamma \left( \beta' + \frac{1}{2} \right) \right|}
\\ &\times \sum_{\ell = 0}^N \left| \frac{\cos(\pi \alpha) \cos(\pi \beta)}{\cos(\pi \Re(\alpha)) \cos(\pi \Re(\beta))} 
\frac{a_\ell(\Re(\alpha)) a_{N - \ell}(\Re(\beta))}{\sin^\ell \zeta \cos^{N - \ell} \zeta} \right| 
\cosh \left(\Im\left( \zeta_{\nu, N, \ell}^{(3)} \right)\right) 
\\ &\times \frac{\left| \Gamma \left( 2\nu + \alpha + \beta + 2 \right) \right|}{\Gamma \left( \Re(2\nu + \alpha + \beta) + N + 2 \right)},
\end{split}
\end{gather}
provided $\left|\Re(\alpha)\right|, \left|\Re(\beta)\right| < \frac{1}{2}$, and
\begin{gather}\label{cutfacbound4}
\begin{split}
& \left|\widehat{R}_N^{(P)}(\nu, \zeta, \alpha, \beta)\right|,\left|\widehat{R}_N^{(\DQ)}(\nu, \zeta, \alpha, \beta)\right|\leq 2 \frac{\Gamma \left( \Re(2\nu + \alpha + \beta - \alpha' - \beta') + 1 \right)}{\left| \Gamma \left( 2\nu + \alpha + \beta - \alpha' - \beta' + 1 \right) \right|}
\frac{\left| \Gamma(-\alpha' - \beta' + N + 1) \right|}{\Gamma(\Re(-\alpha' - \beta') + N + 1)}
\\ &\times \sum_{\ell = 0}^N \left| \frac{\cos(\pi \alpha) \cos(\pi \beta)}{\cos(\pi \Re(\alpha)) \cos(\pi \Re(\beta))} 
\frac{a_\ell(\Re(\alpha)) a_{N - \ell}(\Re(\beta))}{\sin^\ell \zeta \cos^{N - \ell} \zeta} \right| 
\cosh \left(\Im\left( \zeta_{\nu, N, \ell}^{(3)} \right)\right) 
\\ &\times \frac{\left| \Gamma \left( 2\nu + \alpha + \beta + 2 \right) \right|}{\Gamma \left( \Re(2\nu + \alpha + \beta) + N + 2 \right)},
\end{split}
\end{gather}
provided that $\left|\Re(\alpha)\right|, \left|\Re(\beta)\right| < \frac{1}{2}$, and that either $0 < \Re(\alpha' + \beta')$ or $\alpha = \beta = 0$. Analogous bounds hold for $\widehat{R}_N^{(\SzQ)}(\nu, \zeta, \alpha, \beta)$ in which each instance of $\zeta^{(3)}_{\nu, N, \ell}$ ($0\le \ell \le N$) is replaced by $\zeta^{(4)}_{\nu, N, \ell}$. If $2\Re(\alpha)$ or $2\Re(\beta)$ is an odd integer, then the limiting values must to be taken in these bounds.
\end{theorem}

The proof of Theorem \ref{cutinfacthm2} is given in Section \ref{fproof3}.

\section{Proofs of Theorems \ref{thm1}--\ref{thm4}}\label{asympproofs}

The validity of the inverse factorial expansions \eqref{Pifacasymp} and \eqref{Qifacasymp} in the special case where $\left|\Re(\alpha)\right|, \left|\Re(\beta)\right| < \frac{1}{2}$, $\Re(\nu) \to +\infty$, and $\Im(\nu)$ is fixed, follows from Theorems \ref{Pinfacthm} and \ref{Qinfacthm}. In what follows, we will extend the region where these expansions are valid. Define
\begin{equation}\label{Dbar}
\overline{\mathcal{D}} = \left\{ \xi :\Re(\xi ) \ge 0,\; \xi\ne 0,\; \left| \Im(\xi ) \right| < \tfrac{\pi }{2} \right\}.
\end{equation}
Note that $\cosh(2\xi)$ maps $\overline{\mathcal{D}}$ bijectively onto the region $\mathbb{C} \setminus \left(-\infty, 1\right]$, along with both sides of the open segment $(0, 1)$ on the cut along $\left(-\infty, 1\right]$. Using the results of Watson \cite[\S3--9]{Watson1918} applied to the hypergeometric functions in \eqref{Pdef} and \eqref{Qdef}, we find that for $\xi \in \overline{\mathcal{D}}$ and $\alpha, \beta \in \mathbb{C}$, the Jacobi functions can be represented as
\begin{gather}\label{Watson1}
\begin{split}
2^{2\nu + \alpha + \beta + 1} & B(\nu + 1, \nu + \alpha + \beta + 1)  P_\nu^{(\alpha, \beta)} (\cosh (2\xi)) =\\ & \left( \frac{\e^\xi}{\sinh \xi} \right)^{\alpha +1/2} \left( \frac{\e^\xi}{\cosh \xi} \right)^{\beta +1/2} \e^{2\nu \xi} I_+ 
\pm \im \e^{\pm \pi \im \alpha} \left( \frac{\e^{-\xi}}{\sinh \xi} \right)^{\alpha +1/2} \left( \frac{\e^{-\xi}}{\cosh \xi} \right)^{\beta +1/2} \e^{-2\nu \xi} I_-,
\end{split}
\end{gather}
and
\[
\frac{2^{2\nu + \alpha + \beta + 1}}{\pi} B(\nu + 1, \nu + \alpha + \beta + 1) Q_\nu^{(\alpha, \beta)} (\cosh (2\xi)) = \left( \frac{\e^{-\xi}}{\sinh \xi} \right)^{\alpha +1/2} \left( \frac{\e^{-\xi}}{\cosh \xi} \right)^{\beta +1/2} \e^{-2\nu \xi} I_-,
\]
where, for fixed $\alpha$ and $\beta$, the functions $I_\pm$ have asymptotic expansions of the form
\begin{equation}\label{Iasymp}
I_\pm \sim \sum_{n = 0}^\infty \frac{c_n \left(\e^{\pm \xi}\right)}{\Gamma (2\nu + \alpha + \beta + 2)} \frac{2^{2\nu + \alpha + \beta}}{\pi} \Gamma \left( n + \tfrac{1}{2} \right) \frac{\Gamma (\nu + 1) \Gamma (\nu + \alpha + \beta + 1)}{\nu^{n + 1/2}}
\end{equation}
as $\nu\to\infty$, within the sector $\left|\arg \nu \right|\le \frac{\pi}{2}-\delta$ ($<\frac{\pi}{2}$) (for $I_+$) and within $\left|\arg \nu \right|\le \pi-\delta$ ($<\pi$) (for $I_-$). In \eqref{Iasymp}, we assume that $\xi$ is bounded away from the points $0, \pm \frac{\pi}{2}\im$. The coefficients $c_n(w)$ are rational functions of $w$ and polynomial functions of $\alpha$ and $\beta$. In the expansion \eqref{Watson1}, the upper or lower sign is taken according as $\Im(\xi)\gtrless 0$. Watson does not specify the choice of sign when $\Im(\xi) = 0$, noting that the term $\e^{-2\nu \xi} I_-$ is exponentially small compared to $\e^{2\nu \xi} I_+$ and thus can be neglected. To derive an expression for $P_\nu^{(\alpha, \beta)} (\cosh (2\xi))$ when $\Im(\xi) = 0$, we use
\[
P_\nu^{(\alpha, \beta)} (\cosh (2\xi)) = \lim_{\varepsilon \to 0^+} \tfrac{1}{2} \left( P_\nu^{(\alpha, \beta)} (\cosh (2\xi + \im \varepsilon)) + P_\nu^{(\alpha, \beta)} (\cosh (2\xi - \im \varepsilon)) \right)
\]
and obtain
\begin{align*}
2^{2\nu + \alpha + \beta + 1} & B(\nu + 1, \nu + \alpha + \beta + 1)  P_\nu^{(\alpha, \beta)} (\cosh (2\xi)) =\\ & \left( \frac{\e^\xi}{\sinh \xi} \right)^{\alpha +1/2} \left( \frac{\e^\xi}{\cosh \xi} \right)^{\beta +1/2} \e^{2\nu \xi} I_+ 
-C(\xi,\alpha) \left( \frac{\e^{-\xi}}{\sinh \xi} \right)^{\alpha + 1/2} \left( \frac{\e^{-\xi}}{\cosh \xi} \right)^{\beta + 1/2} \e^{-2\nu \xi} I_-,
\end{align*}
where $C(\xi,\alpha)$ is defined in \eqref{Cdef}. This is consistent with both \eqref{Pinfac} and the fact that $P_\nu^{(\alpha, \beta)} (\cosh (2\xi))$ is real when $\xi$, $\nu$, $\alpha$, and $\beta$ are real. For any fixed complex $h$, the gamma function has an asymptotic expansion given by
\begin{equation}\label{Gamma}
\Gamma (\nu+h) \sim \sqrt{2\pi}\nu ^{\nu+h-1/2}\e^{-\nu}\left(1+\frac{6h^2-6h+1}{12\nu} + \frac{36h^4-120h^3+120h^2-36h+1}{288\nu^2}+\ldots\right),
\end{equation}
as $\nu\to\infty$ in the sector $\left|\arg \nu \right|\le \pi-\delta$ ($<\pi$). This result follows from exponentiating the well-known asymptotic expansion of $\log \Gamma(\nu + h)$ \cite[\href{http://dlmf.nist.gov/5.11.E8}{Eq. (5.11.8)}]{NIST:DLMF}. Using this expansion, we can verify the inverse factorial expansion
\[
\frac{2^{2\nu + \alpha + \beta}}{\pi} \Gamma\left(n + \tfrac{1}{2}\right) \frac{\Gamma(\nu + 1) \Gamma(\nu + \alpha + \beta + 1)}{\nu^{n + \frac{1}{2}}} \sim \sum_{k = 0}^{\infty} r_{n,k} \Gamma(2\nu + \alpha + \beta - n - k + 1),
\]
as $\nu\to\infty$ in the sector $\left|\arg \nu \right|\le \pi-\delta$ ($<\pi$), where the coefficients $r_{n,k}$ are polynomials in $\alpha$ and $\beta$. The first two are explicitly given by
\[
r_{n,0} = \frac{(2n)!}{2^n n!}, \quad r_{n,1} = \frac{(2n)!}{2^{n+2} n!} \left( 2(\alpha + \beta)^2 + (2n + 1)(1 + 2\alpha + 2\beta) - 2n^2 \right).
\]
Substituting into \eqref{Iasymp} gives
\[
I_\pm \sim \sum_{n=0}^{\infty} \left( \sum_{k=0}^{n} r_{n-k, k}\, c_{n-k}\left(\e^{\pm\xi}\right) \right) \frac{\Gamma(2\nu + \alpha + \beta - n + 1)}{\Gamma(2\nu + \alpha + \beta + 2)},
\]
as $\nu\to\infty$, within the sector $\left|\arg \nu \right|\le \frac{\pi}{2}-\delta$ ($<\frac{\pi}{2}$) (for $I_+$) and within $\left|\arg \nu \right|\le \pi-\delta$ ($<\pi$) (for $I_-$). Theorem \ref{Pinfacthm} and the uniqueness of the coefficients of an asymptotic expansion imply that
\begin{equation}\label{coeffidentity}
g_n(\mp\xi,\alpha,\beta)=\sum_{k=0}^{n} r_{n-k, k}\, c_{n-k}\left(\e^{\pm\xi}\right)
\end{equation}
when $\xi\in\mathcal{D}$ and $\left|\Re(\alpha)\right|, \left|\Re(\beta)\right| < \frac{1}{2}$. Since both sides are rational functions of $\e^{\pm\xi}$ and polynomial functions of $\alpha$ and $\beta$, an analytic continuation and continuity argument shows that \eqref{coeffidentity} holds for $\xi \in \overline{\mathcal{D}}$ and arbitrary $\alpha, \beta \in \mathbb{C}$. This concludes the proof of Theorem \ref{thm1}.

We now continue with the proof of Theorem \ref{thm2}. Suppose that $0 < \zeta < \frac{\pi}{2}$. Then, by \eqref{PQrel}, we have
\begin{equation}\label{PQrel2}
P_\nu^{(\alpha ,\beta )} (\cos (2\zeta )) = \lim_{\varepsilon \to 0^+} \tfrac{\im}{\pi}\left( \e^{\pi \alpha \im} Q_\nu ^{(\alpha ,\beta )} (\cosh (2\varepsilon  + 2\im\zeta )) - \e^{ - \pi \alpha \im} Q_\nu ^{(\alpha ,\beta )} (\cosh (2\varepsilon  - 2\im\zeta )) \right),
\end{equation}
provided that $\nu+\alpha, \nu+\beta \notin \mathbb{Z}^{-}$. The associated functions $\SzQ_\nu^{(\alpha, \beta)}(x)$ and $\DQ_\nu^{(\alpha, \beta)}(x)$ are defined in terms of the Jacobi function of the second kind, through the limits \eqref{Qfuncdef} and \eqref{Qfuncdef2}, respectively. Specifically,
\begin{equation}\label{Qrel3}
\SzQ_\nu ^{(\alpha ,\beta )} (\cos (2\zeta) ) = \lim_{\varepsilon\to 0^+} \tfrac{1}{2}\left( Q_\nu ^{(\alpha ,\beta )} (\cosh (2\varepsilon  + 2\im\zeta )) + Q_\nu ^{(\alpha ,\beta )} (\cosh (2\varepsilon  - 2\im\zeta )) \right),
\end{equation}
and
\begin{equation}\label{Qrel4}
\DQ_\nu ^{(\alpha ,\beta )} (\cos (2\zeta)) = \lim_{\varepsilon \to 0^+} \tfrac{1}{2}\left( \e^{\pi \alpha \im} Q_\nu ^{(\alpha ,\beta )} (\cosh (2\varepsilon  + 2\im\zeta )) + \e^{ - \pi \alpha \im} Q_\nu ^{(\alpha ,\beta )} (\cosh (2\varepsilon  - 2\im\zeta )) \right),
\end{equation}
provided that $\nu+\alpha, \nu+\beta \notin \mathbb{Z}^{-}$. Since the validity of the inverse factorial series for $Q_\nu^{(\alpha, \beta)}(\cosh(2\xi))$ has been established in $\overline{\mathcal{D}}$, we can apply it to the right-hand sides of \eqref{PQrel2}--\eqref{Qrel4}. After some straightforward algebra, this yields the desired inverse factorial expansions presented in Theorem \ref{thm2}.

The validity of the factorial expansions in Theorem \ref{thm3} for the special case where $\left|\Re(\alpha)\right|, \left|\Re(\beta)\right| < \frac{1}{2}$, $\Re(\nu) \to +\infty$, and $\Im(\nu)$ is fixed, is established by Theorems \ref{Pfacthm} and \ref{Qfacthm}. In the following, we will extend the region of validity for these expansions. Following a similar approach to that used for the inverse factorial series, we find that for $\xi \in \overline{\mathcal{D}}$ and $\alpha, \beta \in \mathbb{C}$, the Jacobi functions can be represented as
\begin{align*}
& \frac{\pi P_\nu ^{(\alpha ,\beta )} (\cosh (2\xi ))}{2^{2\nu  + \alpha  + \beta } B(\nu  + \alpha  + 1,\nu  + \beta  + 1)} =\\ & \left( \frac{\e^\xi}{\sinh \xi} \right)^{\alpha +1/2} \left( \frac{\e^\xi}{\cosh \xi} \right)^{\beta +1/2} \e^{2\nu \xi} J_+ 
-C(\xi,\alpha) \left( \frac{\e^{-\xi}}{\sinh \xi} \right)^{\alpha + 1/2} \left( \frac{\e^{-\xi}}{\cosh \xi} \right)^{\beta + 1/2} \e^{-2\nu \xi} J_-,
\end{align*}
and
\[
\frac{Q_\nu ^{(\alpha ,\beta )} (\cosh (2\xi ))}{2^{2\nu  + \alpha  + \beta } B(\nu  + \alpha  + 1,\nu  + \beta  + 1)} = \left( \frac{\e^{-\xi}}{\sinh \xi} \right)^{\alpha + 1/2} \left( \frac{\e^{-\xi}}{\cosh \xi} \right)^{\beta + 1/2} \e^{-2\nu \xi} J_-,
\]
where, for fixed $\alpha$ and $\beta$, the functions $J_\pm$ have asymptotic expansions of the form
\begin{equation}\label{Jasymp}
J_\pm \sim \sum_{n=0}^{\infty} c_n \left( \e^{\pm \xi} \right) \frac{\Gamma(2\nu + \alpha + \beta + 2)}{2^{2\nu + \alpha + \beta + 1}} \Gamma \left( n + \tfrac{1}{2} \right) \frac{1}{\nu^{n + \frac{1}{2}} \Gamma(\nu + \alpha + 1) \Gamma(\nu + \beta + 1)}
\end{equation}
as $\nu\to\infty$, within the sector $\left|\arg \nu \right|\le \frac{\pi}{2}-\delta$ ($<\frac{\pi}{2}$) (for $J_+$) and within $\left|\arg \nu \right|\le \pi-\delta$ ($<\pi$) (for $J_-$). In \eqref{Jasymp}, we assume that $\xi$ is bounded away from the points $0, \pm \frac{\pi}{2}\im$. Using the expansion \eqref{Gamma}, we can verify the factorial expansion
\[
\frac{1}{2^{2\nu + \alpha + \beta + 1}} \Gamma \left( n + \tfrac{1}{2} \right) \frac{1}{\nu^{n + \frac{1}{2}} \Gamma(\nu + \alpha + 1) \Gamma(\nu + \beta + 1)} \sim \sum_{k=0}^{\infty} \frac{q_{n,k}}{\Gamma(2\nu + \alpha + \beta + n + k + 2)},
\]
as $\nu\to\infty$ in the sector $\left|\arg \nu \right|\le \pi-\delta$ ($<\pi$), where the coefficients $q_{n,k}$ are polynomials in $\alpha$ and $\beta$. The first two are explicitly given by
\[
q_{n,0} = \frac{(2n)!}{2^n n!}, \quad q_{n,1} = - \frac{(2n)!}{2^{n+2} n!} \left( 2(\alpha - \beta)^2 - (2n + 1)(2\alpha + 2\beta + 3) - 2n^2 \right).
\]
Substituting into \eqref{Jasymp}, we obtain
\[
J_\pm \sim \sum_{n=0}^{\infty} \left( \sum_{k=0}^n q_{n-k,k} c_{n-k} \left( \e^{\pm \xi} \right) \right) \frac{\Gamma(2\nu + \alpha + \beta + 2)}{\Gamma(2\nu + \alpha + \beta + n + 2)},
\]
as $\nu\to\infty$, within the sector $\left|\arg \nu \right|\le \frac{\pi}{2}-\delta$ ($<\frac{\pi}{2}$) (for $J_+$) and within $\left|\arg \nu \right|\le \pi-\delta$ ($<\pi$) (for $J_-$). Theorem \ref{Pfacthm} and the uniqueness of the coefficients in an asymptotic expansion imply that
\begin{equation}\label{coeffidentity2}
g_n(\pm\xi,\alpha,\beta)=(-1)^n\sum_{k=0}^n q_{n-k, k}\, c_{n-k}\left(\e^{\pm\xi}\right)
\end{equation}
when $\xi\in\mathcal{D}\setminus \mathbb{R}^+$ and $\left|\Re(\alpha)\right|, \left|\Re(\beta)\right| < \frac{1}{2}$. Since both sides of this equation are rational functions of $\e^{\pm \xi}$ and polynomial functions of $\alpha$ and $\beta$, an analytic continuation and continuity argument show that \eqref{coeffidentity2} holds for $\xi \in \overline{\mathcal{D}}$ and arbitrary $\alpha, \beta \in \mathbb{C}$. This completes the proof of Theorem \ref{thm3}.

Since the factorial series for $Q_\nu^{(\alpha, \beta)}(\cosh(2\xi))$ has been established in $\overline{\mathcal{D}}$, we can now apply it to the right-hand sides of \eqref{PQrel2}--\eqref{Qrel4}. A straightforward algebraic manipulation then leads to the desired factorial expansions, as given in Theorem \ref{thm4}.

\section{Proofs of Theorems \ref{Pinfacthm} and \ref{Qinfacthm}}\label{infproof1}

In this section, we establish the error bounds presented in Theorems \ref{Pinfacthm} and \ref{Qinfacthm}. To this end, we need an estimate for the remainder term in the well-known asymptotic expansion of the modified Bessel function of the second kind, as provided in the following lemma.

\begin{lemma}\label{lemma1} Let $N$ be a non-negative integer, and let $\mu$ be an arbitrary complex number. Define the remainder $r_N^{(K)} (w,\mu )$ by
\begin{equation}\label{eq2}
K_\mu (w) = \sqrt{\frac{\pi }{2w}} \e^{ - w} \left(  \sum\limits_{n = 0}^{N - 1} \frac{a_n (\mu )}{w^n}  + r_N^{(K)} (w,\mu ) \right).
\end{equation}
Then we have the bound
\begin{align*}
\left| r_N^{(K)} (w,\mu ) \right| \le &\;\left|\frac{\cos (\pi \mu )}{\cos (\pi \Re(\mu) )}\right|\frac{\left|a_N (\Re(\mu) )\right|}{\left| w \right|^N } \\
&\times 
\begin{cases} 1 & \text{if }\; |\arg w| \le \frac{\pi }{2}, \\ \min\big(\left|\csc (\arg w) \right|,1 + \chi\big(N+\frac{1}{2}\big)\big) 
& \text{if }\;  \frac{\pi }{2} < |\arg w| \leq \pi, \end{cases}
\end{align*}
provided $\left|\Re(\mu)\right|<N + \frac{1}{2}$. Here, $\chi(p)$ is defined in \eqref{chidef}. If $2\Re(\mu)$ is an odd integer, the limiting value must be taken in this bound. Moreover, the remainder term $r_N^{(K)} (w ,\mu)$ does not exceed the first neglected term in absolute value and has the same sign, provided $w$ is positive, $\mu$ is real, and $\left| \mu \right| <N + \frac{1}{2}$.
\end{lemma}

\begin{proof} The statement for complex variables follows from Theorem 1.8 and Propositions B.1 and B.3 in \cite{Nemes17}. For the real case, see, e.g., \cite[pp.\ 206--207]{Watson1944}.
\end{proof}

We proceed with the proof of Theorem \ref{Pinfacthm}. The Jacobi function of the first kind can be expressed in terms of modified Bessel functions as
\begin{align*}
2^{\nu - 1} & B(\nu + 1, \nu + \alpha + \beta + 1) P_\nu^{(\alpha, \beta)}(z) = \frac{1}{(z - 1)^{\alpha/2} (z + 1)^{\beta/2}} \\ &\times \frac{1}{\Gamma(2\nu + \alpha + \beta + 2)} \int_0^{+\infty} t^{2\nu + \alpha + \beta + 1} I_\alpha \left( t\sqrt{z - 1} \right) K_\beta \left( t\sqrt{z + 1} \right) \id t,
\end{align*}
provided that $\Re(\sqrt{z+1})>|\Re(\sqrt{z-1})|$ and $\Re(2\nu+\alpha+\beta+2) > -\Re(\alpha)+\left|\Re(\beta)\right|$ (combine \cite[Eq. (18)]{Cohl2023} and \cite[Eq. (6.8.44)]{Erdelyi1954}\footnote{Note that the condition $\Re(s + \mu + \nu) > 0$ is missing in reference \cite{Erdelyi1954} for formula (6.8.44) to be valid.}). By applying a suitable change of variables, this can be re-written as
\begin{gather}\label{eq9}
\begin{split}
2^{2\nu + \alpha + \beta}& B(\nu + 1, \nu + \alpha + \beta + 1) P_\nu^{(\alpha, \beta)}(\cosh(2\xi)) = \frac{1}{\sinh^\alpha \xi \cosh^\beta \xi} \\ &\times \frac{1}{\Gamma(2\nu + \alpha + \beta + 2)} \int_0^{+\infty} t^{2\nu + \alpha + \beta + 1} I_\alpha(t \sinh \xi) K_\beta(t \cosh \xi) \id t
\end{split}
\end{gather}
provided that  $\xi  \in \mathcal{D}$ and $\Re(2\nu+\alpha+\beta+2) > -\Re(\alpha)+\left|\Re(\beta)\right|$.

Now, suppose that $|\arg w|\le\pi$. Under this assumption, the modified Bessel functions satisfy the following relations:
\[
I_\mu (w) =  \mp \frac{\im}{\pi }K_\mu (w\e^{ \mp \pi \im} ) \pm \frac{\im}{\pi }\e^{ \pm \pi \im\mu } K_\mu (w),
\]
when $0<\pm \arg w \le\pi$, and
\[
I_\mu (w) = \frac{\im}{2\pi }\left(K_\mu \left(w\e^{\pi \im} \right) - K_\mu \left(w\e^{ - \pi \im} \right)\right) - \frac{\sin (\pi \mu )}{\pi }K_\mu (w),
\]
when $\arg w = 0$ (cf. \cite[\href{https://dlmf.nist.gov/10.34.E3}{Eq. (10.34.3)}]{NIST:DLMF}). By combining these functional equations with the expression \eqref{eq2}, we obtain
\begin{gather}\label{eq5}
\begin{split}
I_\mu (w) = &\;\frac{\e^w }{\sqrt{2\pi w}}\left( \sum_{n = 0}^{N - 1} ( - 1)^n \frac{a_n (\mu )}{w^n }  + \widetilde r_N^{(K)} (w,\mu ) \right) \\
&- \widetilde C(\arg w,\mu )\frac{\e^{ - w} }{\sqrt{2\pi w}}\left( \sum_{m = 0}^{M - 1} \frac{a_m (\mu )}{w^m }  + r_M^{(K)} (w,\mu ) \right),
\end{split}
\end{gather}
with
\begin{equation}\label{eq7}
\widetilde C(\arg w,\mu )  = \begin{cases} \sin (\pi \mu ) & \text{if }\; \arg w = 0, \\ \mp \im \e^{ \pm \pi \im\mu } & \text{if }\; 0<\pm \arg w \le \pi, \end{cases}
\end{equation}
and
\begin{equation}\label{eq11}
\widetilde r_N^{(K)} (w,\mu )  = \begin{cases} \frac{1}{2}\left(r_N^{(K)} \left(w\e^{\pi \im} ,\mu \right) + r_N^{(K)} \left(w\e^{ - \pi \im} ,\mu \right)\right) & \text{if }\; \arg w = 0, \\ r_N^{(K)} \left(w\e^{ \mp \pi \im} ,\mu \right) & \text{if }\; 0<\pm \arg w \le \pi. \end{cases}
\end{equation}
Here, $N$ and $M$ are arbitrary non-negative integers. Multiplying the expressions \eqref{eq2} and \eqref{eq5} together gives
\begin{gather}\label{eq10}
\begin{split}
& I_\alpha ( t\sinh \xi )K_\beta  ( t\cosh \xi  ) = \frac{\exp \left( - t\e^{- \xi } \right)}{2\sqrt{ \sinh \xi\cosh \xi}   t}\Bigg( \sum_{n = 0}^{N - 1} \sum_{\ell = 0}^n (-1)^\ell \frac{a_\ell (\alpha) a_{n - \ell} (\beta)}{\sinh^\ell \xi \cosh^{n - \ell} \xi} \frac{1}{t^n} + R_{N,1} \left( t,\xi,\alpha,\beta \right) \Bigg) \\ & - C(\xi ,\alpha )\frac{\exp \left( - t\e^{\xi }  \right)}{2\sqrt{ \sinh \xi\cosh \xi}   t}\Bigg( \sum_{m = 0}^{M - 1} \sum_{\ell = 0}^m  \frac{a_\ell (\alpha) a_{m - \ell} (\beta)}{\sinh^\ell \xi \cosh^{m - \ell} \xi} \frac{1}{t^m} + R_{M,2} \left( t,\xi,\alpha,\beta \right) \Bigg),
\end{split}
\end{gather}
where
\begin{equation}\label{eq20}
R_{N,1}  ( t,\xi,\alpha,\beta  ) = \widetilde r_N^{(K)} \left( t\sinh \xi,\alpha \right) r_0^{(K)} \left( t\cosh \xi,\beta \right) + \sum_{\ell = 0}^{N - 1} (-1)^\ell\frac{a_\ell (\alpha)}{\sinh^\ell\xi} \frac{1}{t^\ell} r_{N - \ell}^{(K)} \left( t\cosh \xi,\beta \right), 
\end{equation}
or
\begin{equation}\label{eq21}
R_{N,1}  ( t,\xi,\alpha,\beta  ) = \widetilde r_0^{(K)}(t \sinh \xi, \alpha) r_N^{(K)}(t \cosh \xi, \beta) + \sum_{\ell = 1}^N \frac{a_{N-\ell}(\beta)}{\cosh^{N-\ell} \xi} \frac{1}{t^{N-\ell}} \widetilde r_\ell^{(K)}(t \sinh \xi, \alpha),
\end{equation}
and
\begin{equation}\label{eq22}
R_{M,2}  ( t,\xi,\alpha,\beta  ) = r_M^{(K)} \left( t\sinh \xi,\alpha \right) r_0^{(K)} \left( t\cosh \xi,\beta \right) + \sum_{\ell = 0}^{M - 1} \frac{a_\ell (\alpha)}{\sinh^\ell \xi} \frac{1}{t^\ell} r_{M - \ell}^{(K)} \left( t\cosh \xi,\beta \right) ,
\end{equation}
or
\begin{equation}\label{eq23}
R_{M,2}  ( t,\xi,\alpha,\beta  ) = r_0^{(K)} \left( t\sinh \xi,\alpha \right) r_M^{(K)} \left( t\cosh \xi,\beta \right) + \sum_{\ell = 1}^M \frac{a_{M-\ell} (\beta)}{\cosh^{M-\ell} \xi} \frac{1}{t^{M-\ell}} r_\ell^{(K)} \left( t\sinh \xi,\alpha \right) .
\end{equation}
Furthermore, we have defined $C(\xi ,\alpha ) =\widetilde C \left( \arg \left( \sinh \xi \right),\alpha \right)$. The formula \eqref{Cdef} for $C(\xi, \alpha)$ follows directly from \eqref{eq7} and the properties of the hyperbolic sine function. Assume that $\Re(2\nu + \alpha + \beta + 1) > \max(N, M)$ and $\Re(2\nu+\alpha+\beta+2) > -\Re(\alpha)+\left|\Re(\beta)\right|$. Substituting equation \eqref{eq10} into \eqref{eq9}, we obtain  \eqref{Pinfac} with
\[
R_N^{(P1)} (\nu,\xi ,\alpha ,\beta ) = \frac{\e^{ -(2\nu + \alpha  + \beta  + 1) \xi}}{\Gamma (2\nu  + \alpha  + \beta  + 2)}\int_0^{ + \infty }  t^{2\nu + \alpha  + \beta } \exp \left(  - t\e^{ - \xi } \right)R_{N,1} (t,\xi ,\alpha ,\beta )\id t ,
\]
\[
R_M^{(P2)} (\nu,\xi ,\alpha ,\beta ) = \frac{\e^{(2\nu + \alpha  + \beta  + 1) \xi}}{\Gamma (2\nu  + \alpha  + \beta  + 2)} \int_0^{ + \infty }  t^{2\nu + \alpha  + \beta } \exp \left( - t\e^{\xi } \right)R_{M,2} (t,\xi ,\alpha ,\beta )\id t  .
\]
We change the integration variable from $t$ to $s$, setting $t = s \e^\xi$ for the first integral and $t = s \e^{-\xi}$ for the second. By applying Cauchy's theorem, the contours of integration can be deformed back onto the positive real axis, leading to the following representations:
\begin{equation}\label{RP1int}
R_N^{(P1)} (\nu,\xi ,\alpha ,\beta ) = \frac{1}{\Gamma (2\nu  + \alpha  + \beta  + 2)}\int_0^{ + \infty }  s^{2\nu + \alpha  + \beta } \e^{-s}R_{N,1} \left(s\e^{\xi },\xi ,\alpha ,\beta \right)\id s ,
\end{equation}
\begin{equation}\label{RP2int}
R_M^{(P2)} (\nu,\xi ,\alpha ,\beta ) = \frac{1}{\Gamma (2\nu  + \alpha  + \beta  + 2)} \int_0^{ + \infty }  s^{2\nu + \alpha  + \beta } \e^{-s}R_{M,2} \left(s\e^{ - \xi },\xi ,\alpha ,\beta \right)\id s  .
\end{equation}
From Lemma \ref{lemma1}, equations \eqref{eq11} and \eqref{eq20}, and the properties of the function $\e^{\xi} \sinh\xi$, we deduce the following bound:
\begin{gather}\label{bnd1}
\begin{split}
\left|R_{N,1} \left(s\e^{\xi },\xi ,\alpha ,\beta \right)\right|  \le &\; \left| \frac{\cos(\pi \alpha) \cos(\pi \beta)}{\cos(\pi \Re(\alpha)) \cos(\pi \Re(\beta))} a_N(\Re(\alpha)) \left( \frac{\e^{-\xi}}{\sinh \xi} \right)^N \right| \frac{1}{s^N} \\ & \times \begin{cases} 1 & \text{if }\; \frac{\pi }{2} \le \left| \vartheta  \right| \le \pi, \\ \min\big(\left|\csc \vartheta \right|,1 + \chi\big(N+\frac{1}{2}\big)\big) 
& \text{if }\; \left| \vartheta  \right| < \frac{\pi }{2} \end{cases} \\ &
+ \sum_{\ell = 0}^{N-1} \left| \frac{\cos(\pi \beta)}{\cos(\pi \Re(\beta))} a_\ell(\alpha) a_{N-\ell}(\Re(\beta)) \left( \frac{\e^{-\xi}}{\sinh \xi} \right)^\ell \left( \frac{\e^{-\xi}}{\cosh \xi} \right)^{N-\ell} \right| \frac{1}{s^N},
\end{split}
\end{gather}
where $\vartheta = \arg(\e^\xi \sinh \xi) \in \left[-\pi,\pi\right]$, assuming $\xi \in \overline{\mathcal{D}}$, $\left|\Re (\alpha)\right|<N+\frac{1}{2}$, and $\left|\Re(\beta )\right|<\frac{1}{2}$. For the definition of $\overline{\mathcal{D}}$, refer to \eqref{Dbar}. Similarly, using the alternative expression \eqref{eq21}, we obtain the bound
\begin{gather}\label{bnd2}
\begin{split}
\left|R_{N,1} \left(s\e^{\xi },\xi ,\alpha ,\beta \right)\right|  \le &\; \left(\left| \frac{\cos(\pi \alpha) \cos(\pi \beta)}{\cos(\pi \Re(\alpha)) \cos(\pi \Re(\beta))} a_N(\Re(\beta)) \left( \frac{\e^{-\xi}}{\cosh \xi} \right)^N \right| \frac{1}{s^N} \right. \\ &+ \left. \sum_{\ell = 1}^N \left| \frac{\cos(\pi \alpha)}{\cos(\pi \Re(\alpha))} a_\ell(\Re(\alpha)) a_{N-\ell}(\beta) \left( \frac{\e^{-\xi}}{\sinh \xi} \right)^\ell \left( \frac{\e^{-\xi}}{\cosh \xi} \right)^{N-\ell} \right| \frac{1}{s^N}\right)
\\ & \times \begin{cases} 1 & \text{if }\; \frac{\pi }{2} \le \left| \vartheta  \right|\le \pi, \\ \min\big(\left|\csc \vartheta \right|,1 + \chi\big(N+\frac{1}{2}\big)\big) 
& \text{if }\; \left| \vartheta  \right| < \frac{\pi }{2}, \end{cases}
\end{split}
\end{gather}
provided that $\xi \in \overline{\mathcal{D}}$, $\left|\Re (\alpha)\right|<\frac{1}{2}$, and $\left|\Re(\beta )\right|<N+\frac{1}{2}$. Note that if either $\left|\Re(\alpha)\right| < N + \frac{1}{2}$ and $\left|\Re(\beta)\right| < \frac{1}{2}$, or $\left|\Re(\alpha)\right| < \frac{1}{2}$ and $\left|\Re(\beta)\right| < N + \frac{1}{2}$, then the condition $\Re(2\nu + \alpha + \beta + 1) > \max(N, M)$ automatically implies $\Re(2\nu + \alpha + \beta + 2) > -\Re(\alpha) + \left|\Re(\beta)\right|$. Consequently, this latter condition can be omitted. In Section 7 of the paper \cite{Nemes2020}, it is shown that the conditions involving $\vartheta$ can be equivalently expressed in terms of $\xi$ as follows:
\begin{equation}\label{cases}
\begin{cases} 1 & \text{if }\; \Re \left(\e^{2\xi } \right) \le 1, \\ 
\min\big(\left| 1 - \e^{ - 2\xi } \right|\left| \csc (2\Im(\xi) ) \right|,1 + \chi\big(N+\tfrac{1}{2}\big)\big) & \text{if }\; \Re \left(\e^{2\xi } \right) > 1. \end{cases}
\end{equation}
From \eqref{RP1int}, we can infer that
\[
\left| R_N^{(P1)} (\nu,\xi ,\alpha ,\beta )\right| \le \frac{1}{\left|\Gamma (2\nu  + \alpha  + \beta  + 2)\right|}\int_0^{ + \infty }  s^{\Re(2\nu + \alpha  + \beta)} \e^{-s}\left|R_{N,1} \left(s\e^{\xi },\xi ,\alpha ,\beta \right)\right|\id s.
\]
Restricting $\xi$ to $\mathcal{D}$, substituting the estimates from \eqref{bnd1} and \eqref{bnd2}, and applying \eqref{cases} yields the bounds \eqref{eq40} and \eqref{eq41}, respectively. In a similar manner, from Lemma \ref{lemma1}, formulae \eqref{eq22} and \eqref{eq23}, and the properties of the function $\e^{-\xi} \sinh\xi$, we can deduce the following bounds:
\begin{gather}\label{bnd3}
\begin{split}
\left|R_{M,2} \left(s\e^{-\xi },\xi ,\alpha ,\beta \right)\right|  \le &\; \left| \frac{\cos(\pi \alpha) \cos(\pi \beta)}{\cos(\pi \Re(\alpha)) \cos(\pi \Re(\beta))} a_M(\Re(\alpha)) \left( \frac{\e^{\xi}}{\sinh \xi} \right)^M \right| \frac{1}{s^M} \\ &
+ \sum_{\ell = 0}^{M-1} \left| \frac{\cos(\pi \beta)}{\cos(\pi \Re(\beta))} a_\ell(\alpha) a_{M-\ell}(\Re(\beta)) \left( \frac{\e^{\xi}}{\sinh \xi} \right)^\ell \left( \frac{\e^{\xi}}{\cosh \xi} \right)^{M-\ell} \right| \frac{1}{s^M},
\end{split}
\end{gather}
provided that $\xi \in \overline{\mathcal{D}}$, $\left|\Re (\alpha)\right|<M+\frac{1}{2}$, and $\left|\Re(\beta )\right|<\frac{1}{2}$, and
\begin{gather}\label{bnd4}
\begin{split}
\left|R_{M,2} \left(s\e^{-\xi },\xi ,\alpha ,\beta \right)\right|  \le &\; \left| \frac{\cos(\pi \alpha) \cos(\pi \beta)}{\cos(\pi \Re(\alpha)) \cos(\pi \Re(\beta))} a_M(\Re(\beta)) \left( \frac{\e^{\xi}}{\cosh \xi} \right)^M \right| \frac{1}{s^M}  \\ &+ \sum_{\ell = 1}^M \left| \frac{\cos(\pi \alpha)}{\cos(\pi \Re(\alpha))} a_\ell(\Re(\alpha)) a_{M-\ell}(\beta) \left( \frac{\e^{\xi}}{\sinh \xi} \right)^\ell \left( \frac{\e^{\xi}}{\cosh \xi} \right)^{M-\ell} \right| \frac{1}{s^M}
\end{split}
\end{gather}
provided that $\xi \in \overline{\mathcal{D}}$, $\left|\Re (\alpha)\right|<\frac{1}{2}$, and $\left|\Re(\beta )\right|<M+\frac{1}{2}$. The bounds \eqref{eq42} and \eqref{eq43} now follow similarly to those of \eqref{eq40} and \eqref{eq41}. The estimates \eqref{eq44} and \eqref{eq45} can be derived from \eqref{eq41} and \eqref{eq43} by applying the inequality
\begin{equation}\label{coeffineq}
\left| a_\ell(\beta) \right| \le \left| \frac{\cos (\pi \beta)}{\cos (\pi \Re(\beta))}a_\ell (\Re (\beta)) \right|, \quad \left|\Re(\beta) \right|<\ell+\tfrac{1}{2},
\end{equation}
which follows from \cite[Eq. (69)]{Nemes17}.

Now, assume that $ \xi > 0 $, and that $ \nu $, $ \alpha $, and $ \beta $ are real, with $ |\alpha|, |\beta| < \frac{1}{2} $, and $ 2\nu + \alpha + \beta + 1 > M $. From Lemma \ref{lemma1}, it follows that
\[
r_M^{(K)} (w, \mu) = \theta_M (w, \mu) \frac{a_M (\mu)}{w^M},
\]
where $ w > 0 $, $ \mu $ is real, $ |\mu| < M + \frac{1}{2} $, and $ 0 < \theta_M (w, \mu) < 1 $. Applying this result to \eqref{eq22} gives
\begin{align*}
R_{M,2} \left(s\e^{ - \xi }, \xi, \alpha, \beta\right) =&\; \theta_M \left(s\e^{ - \xi } \sinh \xi, \alpha\right) \theta_0 \left(s\e^{ - \xi } \cosh \xi, \beta\right) a_M (\alpha) \left( \frac{\e^\xi}{\sinh \xi} \right)^M \frac{1}{s^M} 
\\ & + \sum_{\ell = 0}^{M - 1} \theta_{M - \ell} \left(s\e^{ - \xi } \cosh \xi, \beta\right) a_\ell (\alpha) a_{M - \ell} (\beta) \left( \frac{\e^\xi}{\sinh \xi} \right)^\ell \left( \frac{\e^\xi}{\cosh \xi} \right)^{M - \ell} \frac{1}{s^M}\\
=&\; \widetilde{\theta}_M (s, \xi,\alpha, \beta) \sum_{\ell = 0}^{M - 1} a_\ell (\alpha) a_{M - \ell} (\beta) \left( \frac{\e^\xi}{\sinh \xi} \right)^\ell \left( \frac{\e^\xi}{\cosh \xi} \right)^{M - \ell} \frac{1}{s^M},
\end{align*}
with a suitable $0<\widetilde{\theta}_M (s, \xi, \alpha,\beta)<1$, since $a_\ell  (\alpha )a_{M - \ell } (\beta ) = ( - 1)^M \left| a_\ell  (\alpha )a_{M - \ell } (\beta ) \right|$ when $|\alpha|, |\beta| < \frac{1}{2}$. Hence, by applying the mean value theorem for improper integrals, we deduce that
\begin{align*}
R_M^{(P2)} (\nu,\xi ,\alpha ,\beta )  = &\;\sum_{\ell = 0}^{M - 1} a_\ell (\alpha) a_{M - \ell} (\beta) 
\left( \frac{\e^\xi}{\sinh \xi} \right)^\ell 
\left( \frac{\e^\xi}{\cosh \xi} \right)^{M - \ell} \\ & \times \frac{1}{\Gamma(2\nu + \alpha + \beta + 2)}
\int_0^{+\infty} \widetilde{\theta}_M (s, \xi, \alpha, \beta) s^{2\nu + \alpha + \beta - M} \e^{-s} \id s 
\\   = &\;\Theta_M (\nu,\xi, \alpha, \beta) \sum_{\ell = 0}^{M - 1} a_\ell (\alpha) a_{M - \ell} (\beta) 
\left( \frac{\e^\xi}{\sinh \xi} \right)^\ell 
\left( \frac{\e^\xi}{\cosh \xi} \right)^{M - \ell} 
\frac{\Gamma(2\nu + \alpha + \beta - M + 1)}{\Gamma(2\nu + \alpha + \beta + 2)}
\end{align*}
where $0<\Theta_M (\nu,\xi, \alpha, \beta)<1$. This concludes the proof of Theorem \ref{Pinfacthm}.

We continue with the proof of Theorem \ref{Qinfacthm}. The Jacobi function of the second kind can be expressed in terms of the modified Bessel function of the second kind as follows:
\begin{align*}
2^{\nu - 1} & B(\nu + 1, \nu + \alpha + \beta + 1) Q_\nu^{(\alpha, \beta)}(z) = \frac{1}{(z - 1)^{\alpha/2} (z + 1)^{\beta/2}} \\ &\times \frac{1}{\Gamma(2\nu + \alpha + \beta + 2)} \int_0^{+\infty} t^{2\nu + \alpha + \beta + 1} K_\alpha \left( t\sqrt{z - 1} \right) K_\beta \left( t\sqrt{z + 1} \right) \id t,
\end{align*}
provided that $\Re(\sqrt{z-1}+\sqrt{z+1})>0$ and $\Re (2\nu+\alpha+\beta+2) > \left|\Re(\alpha)\right|+\left|\Re(\beta)\right|$ (combine \cite[Eq. (21)]{Cohl2023} and \cite[Eq. (6.8.47)]{Erdelyi1954}). By applying an appropriate change of variables, this can be rewritten as \eqref{integralrep1}. Let $N$ be an arbitrary
non-negative integer. Multiplying two copies of the expression in \eqref{eq2} yields
\begin{equation}\label{eq25}
 K_\alpha ( t\sinh \xi )K_\beta  ( t\cosh \xi  ) = \frac{\pi}{2}\frac{\exp \left( - t\e^{\xi }  \right)}{\sqrt{ \sinh \xi\cosh \xi}   t}\Bigg( \sum_{n = 0}^{N - 1} \sum_{\ell = 0}^n  \frac{a_\ell (\alpha) a_{n- \ell} (\beta)}{\sinh^\ell \xi \cosh^{n- \ell} \xi} \frac{1}{t^n} + R_{N,2} \left( t,\xi,\alpha,\beta \right) \Bigg),
\end{equation}
where $R_{N,2} \left( t,\xi,\alpha,\beta \right)$ is given by either \eqref{eq22} or \eqref{eq23}, with $N$ in place of $M$. Assume that $\Re(2\nu + \alpha + \beta + 1) > N$ and $\Re (2\nu+\alpha+\beta+2) > \left|\Re(\alpha)\right|+\left|\Re(\beta)\right|$. By substituting equation \eqref{eq25} into \eqref{integralrep1}, we obtain \eqref{Qinfac} with $R_N^{(Q)}(\nu, \xi, \alpha, \beta)=R_N^{(P2)}(\nu, \xi, \alpha, \beta)$. Therefore, the estimates for $R_N^{(Q)}(\nu, \xi, \alpha, \beta)$ follow directly from those for $R_N^{(P2)}(\nu, \xi, \alpha, \beta)$. Note that if either $\left|\Re(\alpha)\right| < N + \frac{1}{2}$ and $\left|\Re(\beta)\right| < \frac{1}{2}$, or $\left|\Re(\alpha)\right| < \frac{1}{2}$ and $\left|\Re(\beta)\right| < N + \frac{1}{2}$, then the condition $\Re(2\nu + \alpha + \beta + 1) > N$ automatically ensures that $\Re(2\nu + \alpha + \beta + 2) > \left|\Re(\alpha)\right| + \left|\Re(\beta)\right|$. Therefore, the latter condition can be omitted.

We conclude this section by showing that the inverse factorial expansions \eqref{Pifacasymp} and \eqref{Qifacasymp} terminate, providing an exact representation of the corresponding functions when both $2\alpha$ and $2\beta$ are odd integers. Let $L$ denote the positive integer defined by $\max (2\left| \alpha \right|,2\left| \beta \right|) = 2L - 1$. Then, by definition \eqref{andef} and Lemma \ref{lemma1}, we have
\[
a_\ell  \left( L - \tfrac{1}{2} \right) = 0,\quad r_\ell ^{(K)} \left(w,L - \tfrac{1}{2} \right)=0,
\]
for all $\ell \ge L$. Therefore, from \eqref{eq11}, \eqref{eq20}, and \eqref{eq22}, it follows that $R_{2L,1} (t,\xi ,\alpha ,\beta ) = 0$ and $R_{2L,2} (t,\xi ,\alpha ,\beta ) = 0$. Hence, by \eqref{RP1int} and \eqref{RP2int}, we find that $R_{2L}^{(P1)} (\nu, \xi, \alpha, \beta) = 0$ and $R_{2L}^{(P2)} (\nu, \xi, \alpha, \beta) = 0$, provided $\Re(2\nu + \alpha + \beta + 1) > 2L$. Thus, by Theorems \ref{Pinfacthm} and \ref{Qinfacthm}, the expansions \eqref{Pifacasymp} and \eqref{Qifacasymp} terminate, giving exact representations of the corresponding functions. The conditions on the parameters may be relaxed by appealing to analytic continuation.

\section{Proof of Theorem \ref{cutinfacthm}}\label{infproof2}

In this section, we establish the error bounds presented in Theorem \ref{cutinfacthm}. Suppose that $0 < \zeta  < \frac{\pi }{2}$. Assuming that $\Re(2\nu + \alpha + \beta + 1) > N$ and $\Re (2\nu+\alpha+\beta+2) > \left|\Re(\alpha)\right|+\left|\Re(\beta)\right|$, and noting that $R_N^{(Q)}(\nu,\xi,\alpha,\beta)=R_N^{(P2)}(\nu,\xi,\alpha,\beta)$, we substitute \eqref{Qinfac} into the right-hand side of \eqref{PQrel2} and apply the representation \eqref{RP2int}. This results in \eqref{Pinfac2} with
\begin{align*}
R_N^{(P)} (\nu,\zeta, \alpha, \beta) =&\; \tfrac{1}{2}\e^{-\left( (2\nu + \alpha + \beta + 1)\zeta - \left( \alpha + \frac{1}{2} \right) \frac{\pi}{2} \right)\im} \mathop{\lim}\limits_{\varepsilon \to 0^+} R_N^{(Q)} (\nu,\varepsilon + \im\zeta, \alpha, \beta) \\ & + \tfrac{1}{2}\e^{ \left( (2\nu + \alpha + \beta + 1)\zeta - \left( \alpha + \frac{1}{2} \right) \frac{\pi}{2} \right)\im} \mathop{\lim}\limits_{\varepsilon \to 0^+} R_N^{(Q)} (\nu,\varepsilon - \im\zeta, \alpha, \beta)
\\=&\;\frac{\e^{-\left( (2\nu + \alpha + \beta + 1)\zeta - \left( \alpha + \frac{1}{2} \right) \frac{\pi}{2} \right)\im}}{2\Gamma (2\nu+\alpha+\beta+2)}\int_0^{ + \infty } s^{2\nu+\alpha+\beta} \e^{-s} R_{N,2}\left(s\e^{-\im\zeta},\im\zeta,\alpha ,\beta\right)\id s
\\ & + \frac{\e^{ \left( (2\nu + \alpha + \beta + 1)\zeta - \left( \alpha + \frac{1}{2} \right) \frac{\pi}{2} \right)\im}}{2\Gamma (2\nu+\alpha+\beta+2)}\int_0^{ + \infty } s^{2\nu+\alpha+\beta} \e^{-s} R_{N,2}\left(s\e^{\im\zeta},-\im\zeta,\alpha ,\beta\right)\id s.
\end{align*}
A straightforward estimation yields the bound
\begin{gather}\label{eq48}
\begin{split}
\left|R_N^{(P)} (\nu,\zeta, \alpha, \beta)\right| \le&\; \frac{\e^{\Im \left(\zeta^{(1)}_{\nu ,N,0}\right)}}{2\Gamma(2\nu + \alpha + \beta + 2)} \int_0^{+\infty} s^{\operatorname{Re}(2\nu + \alpha + \beta)} \e^{-s} 
\left| R_{N,2}\left( s\e^{-\im \zeta}, \im \zeta, \alpha, \beta \right) \right| \id s
\\ & + \frac{\e^{-\Im \left(\zeta^{(1)}_{\nu ,N,0}\right)}}{2\Gamma(2\nu + \alpha + \beta + 2)} \int_0^{+\infty} s^{\operatorname{Re}(2\nu + \alpha + \beta)} \e^{-s} 
\left| R_{N,2}\left( s\e^{\im \zeta}, -\im \zeta, \alpha, \beta \right) \right| \id s.
\end{split}
\end{gather}
Using the estimates in \eqref{bnd3} and \eqref{bnd4} with $\pm \im \zeta$ replacing $\xi$, we obtain
\begin{align*}
 \left| R_{N,2} \left( s \e^{\mp\im \zeta}, \pm\im \zeta, \alpha, \beta \right) \right| \le&\; \left| \frac{\cos (\pi \alpha) \cos (\pi \beta)}{\cos (\pi \Re(\alpha)) \cos (\pi \Re(\beta))} \frac{a_N (\Re(\alpha))}{\sin^N \zeta} \right| \frac{1}{s^N} \\&+ \sum_{\ell = 0}^{N - 1} \left| \frac{\cos (\pi \beta)}{\cos (\pi \Re(\beta))} \frac{a_\ell (\alpha) a_{N - \ell} (\Re(\beta))}{\sin^\ell \zeta \cos^{N - \ell} \zeta} \right| \frac{1}{s^N},
\end{align*}
provided that $\left|\Re (\alpha)\right|<N+\frac{1}{2}$ and $\left|\Re(\beta )\right|<\frac{1}{2}$, and
\begin{align*}
 \left| R_{N,2} \left( s \e^{\mp\im \zeta}, \pm\im \zeta, \alpha, \beta \right) \right| \le&\; \left| \frac{\cos (\pi \alpha) \cos (\pi \beta)}{\cos (\pi \Re(\alpha)) \cos (\pi \Re(\beta))} \frac{a_N (\Re(\beta))}{\cos^N \zeta} \right| \frac{1}{s^N} \\&+ \sum_{\ell = 1}^{N} \left| \frac{\cos (\pi \alpha)}{\cos (\pi \Re(\alpha))} \frac{a_\ell (\Re(\alpha)) a_{N - \ell} (\beta)}{\sin^\ell \zeta \cos^{N - \ell} \zeta} \right| \frac{1}{s^N},
\end{align*}
assuming $\left|\Re (\alpha)\right|<\frac{1}{2}$ and $\left|\Re(\beta )\right|<N+\frac{1}{2}$. By substituting these estimates into \eqref{eq48} and observing that $ \Im \left(\zeta^{(1)}_{\nu ,N,0}\right) = \Im \left(\zeta^{(1)}_{\nu ,N,\ell}\right) $ for all $ 0 \leq \ell \leq N $, we obtain the bounds \eqref{cutbound1} and \eqref{cutbound2}, respectively. The estimate \eqref{cutbound3} can be derived from \eqref{cutbound2} by applying the inequality \eqref{coeffineq}. The condition $\Re(2\nu + \alpha + \beta + 2) > \left|\Re(\alpha)\right| + \left|\Re(\beta)\right|$ can be omitted for the same reason explained in the previous section.

The bounds for the remainder terms $R_N^{(\SzQ)} (\nu, \zeta, \alpha, \beta)$ and $R_N^{(\DQ)} (\nu, \zeta, \alpha, \beta)$ in the truncated inverse factorial series \eqref{Qinfac2} and \eqref{Qinfac3} can be derived in a similar fashion by combining \eqref{Qinfac} with \eqref{Qrel3} and \eqref{Qrel4}, respectively. The details are left for the reader to verify.

The integral representation of $R_N^{(P)} (\nu, \zeta, \alpha, \beta)$, along with the analogous representations for $R_N^{(\SzQ)} (\nu, \zeta, \alpha, \beta)$ and $R_N^{(\DQ)} (\nu, \zeta, \alpha, \beta)$, as well as the discussion at the end of the previous section, demonstrate that the inverse factorial expansions \eqref{invfexp1}, \eqref{invfexp2}, and \eqref{invfexp3} terminate, providing an exact representation of the corresponding functions when both $2\alpha$ and $2\beta$ are odd integers.

\section{Proof of Theorems \ref{Pfacthm} and \ref{Qfacthm}}\label{fproof1}

In this section, we establish the error bounds given in Theorems \ref{Pfacthm} and \ref{Qfacthm}. To do so, we first need appropriate integral representations of the Jacobi functions. Hahn \cite[Eqs. (9.6) and (14.4)]{Hahn1980} showed that if $\Re(\xi)>0$, $-\frac{\pi}{2}<\Im(\xi)<0$ and $\Re(2\nu + \alpha + \beta + 1) > \Re(\alpha' + \beta')$, then
\begin{align*}
& \frac{\pi P_\nu^{(\alpha ,\beta )} (\cosh (2\xi ))}{2^{2\nu + \alpha  + \beta } B(\nu + \alpha  + 1,\nu + \beta  + 1)} =\\ & \left( \frac{\e^\xi}{\sinh \xi} \right)^{\alpha + 1/2} \left( \frac{\e^\xi}{\cosh \xi} \right)^{\beta + 1/2} \e^{2\nu\xi} Y_+(\xi) - \im \e^{ - \pi \im \alpha } \left( \frac{\e^{ - \xi }}{\sinh \xi} \right)^{\alpha + 1/2} \left( \frac{\e^{ - \xi }}{\cosh \xi} \right)^{\beta + 1/2} \e^{ - 2\nu\xi} Y_-(\xi ),
\end{align*}
where
\begin{gather}\label{Ydef}
\begin{split}
Y_\pm(\xi) =\; & \frac{\Gamma (2\nu + \alpha + \beta + 2)}{\Gamma (2\nu+\alpha+\beta-\alpha'-\beta' + 1)\Gamma \left( \alpha' + \frac{1}{2} \right)\Gamma \left( \beta' + \frac{1}{2} \right)} \\ & \times  \int_0^1 \int_0^1 t^{\alpha' + \beta'} (1 - t)^{2\nu  + \alpha  + \beta  - \alpha ' - \beta '} s^{\alpha' - 1/2} (1 - s)^{\beta' - 1/2}\\ & \times \left( 1 \mp ts\frac{\e^{\pm\xi}}{2\sinh \xi} \right)^{\alpha' - 1/2} \left( 1 - t(1 - s)\frac{\e^{\pm\xi}}{2\cosh \xi} \right)^{\beta' - 1/2} \id s\id t.
\end{split}
\end{gather}
(In Hahn's notation, $Y_\pm(\xi) = y\left( \frac{\e^{\pm\xi}}{2 \cosh \xi} \right)$.) Here, $\alpha'$ and $\beta'$ are defined in \eqref{abprime}. Note that, in Hahn's original formulation, $\nu$, $\alpha$, and $\beta$ are assumed to be real, but the derivation remains valid for complex parameters as well. By the Schwarz reflection principle, we have $P_\nu^{(\alpha, \beta)}(\cosh(2\xi)) = \overline{P_\nu^{(\alpha, \beta)}(\cosh(2\bar{\xi}))}$, assuming that $\nu$, $\alpha$, and $\beta$ are real. Using this observation along with analytic continuation, we obtain
\begin{gather}\label{PHahn}
\begin{split}
& \frac{\pi P_\nu^{(\alpha ,\beta )} (\cosh (2\xi ))}{2^{2\nu + \alpha  + \beta } B(\nu + \alpha  + 1,\nu + \beta  + 1)} =\\ & \left( \frac{\e^\xi}{\sinh \xi} \right)^{\alpha + 1/2} \left( \frac{\e^\xi}{\cosh \xi} \right)^{\beta + 1/2} \e^{2\nu\xi} Y_+(\xi) - C(\xi,\alpha) \left( \frac{\e^{ - \xi }}{\sinh \xi} \right)^{\alpha + 1/2} \left( \frac{\e^{ - \xi }}{\cosh \xi} \right)^{\beta + 1/2} \e^{ - 2\nu\xi} Y_-(\xi ),
\end{split}
\end{gather}
where $C(\xi ,\alpha )$ is defined in \eqref{Cdef}, $\xi \in \mathcal{D} \setminus \mathbb{R}^+$ and $\Re(2\nu + \alpha + \beta + 1) > \Re(\alpha' + \beta')$.

From the hypergeometric representation \eqref{Qhypergeom} and Hahn's results \cite[Eqs. (8.2) and (8.21)]{Hahn1980}, we can assert that
\begin{equation}\label{QHahn}
\frac{ Q_\nu ^{(\alpha ,\beta )} (\cosh (2\xi) )}{2^{2\nu  + \alpha  + \beta } B(\nu  + \alpha  + 1,\nu  + \beta  + 1)} = \left( \frac{\e^{ - \xi }}{\sinh \xi} \right)^{\alpha  + 1/2} \left( \frac{\e^{ - \xi }}{\cosh \xi} \right)^{\beta  + 1/2} \e^{ - 2\nu \xi } Y_-(\xi ),
\end{equation}
provided that $\xi \in \mathcal{D}$ and $\Re(2\nu + \alpha + \beta + 1) > \Re(\alpha' + \beta')$.

To continue, we first need to establish the following lemma. Here, $\overline{\mathcal{D}}$ is defined by \eqref{Dbar}.

\begin{lemma}\label{lemma2} Let $N$ be a non-negative integer, and let $A$ and $B$ be complex numbers. Let $u$, $v_1$, and $v_2$ be real numbers such that $0 < u, v_1, v_2 < 1$. For $\xi \in \overline{\mathcal{D}} \setminus \mathbb{R}^+$, define the remainder $\widehat{R}_N^+(u, v_1, v_2, \xi, A, B)$, and for $\xi \in \overline{\mathcal{D}}$, define the remainder $\widehat{R}_N^-(u, v_1, v_2, \xi, A, B)$ by
\begin{gather}\label{binomprod}
\begin{split}
&\left( 1 \mp uv_1 \frac{\e^{\pm \xi}}{2 \sinh \xi} \right)^{-A} \left( 1 - uv_2 \frac{\e^{\pm \xi}}{2 \cosh \xi} \right)^{-B} =\\ & \sum_{n = 0}^{N - 1} \left( \sum_{\ell = 0}^n \frac{(A)_\ell (B)_{n - \ell}}{\ell!(n - \ell)!} \left( \pm v_1 \frac{\e^{\pm \xi}}{2 \sinh \xi} \right)^\ell \left( v_2 \frac{\e^{\pm \xi}}{2 \cosh \xi} \right)^{n - \ell} \right) u^n + \widehat{R}_N^\pm (u,v_1,v_2,\xi,A,B).
\end{split}
\end{gather}
Then, we have the following bounds:
\begin{gather}\label{Rhatbound1}
\begin{split}
\left| \widehat{R}_N^+ (u,v_1,v_2,\xi, A, B) \right| \le \; & \left| \frac{\sin(\pi A) \sin(\pi B)}{\sin(\pi \Re(A)) \sin(\pi \Re(B))} \frac{(\Re(A))_N}{N!} \left( v_1 \frac{\e^\xi}{2 \sinh \xi} \right)^N\right| u^N \\ & \times \left| 1 - \e^{4\xi} \right| \left| \csc(2 \Im(\xi)) \right|
\\ & + \sum_{\ell = 0}^{N - 1} \left| \frac{\sin(\pi B)}{\sin(\pi \Re(B))} \frac{(A)_\ell (\Re(B))_{N - \ell}}{\ell! (N - \ell)!} \left( v_1 \frac{\e^\xi}{2 \sinh \xi} \right)^\ell \left( v_2 \frac{\e^\xi}{2 \cosh \xi} \right)^{N - \ell} \right| u^N
\\ &\times 
\begin{cases}
    \left| 1 + \e^{2\xi} \right| \left| \csc(2 \Im(\xi)) \right|, & \text{if }\; -1 < \Re\left(\e^{- 2\xi}\right) \le 0, \\ 
    \left| 1 + \e^{2\xi} \right|, & \text{if }\; 0 < \Re\left(\e^{- 2\xi}\right),
\end{cases}
\end{split}
\end{gather}
provided that $-N <\Re(A) < 1$ and $0<\Re(B)<1$,
\begin{gather}\label{Rhatbound2}
\begin{split}
\left| \widehat{R}_N^+ (u,v_1,v_2,\xi, A, B) \right| \le \; &\left| \frac{\sin(\pi A) \sin(\pi B)}{\sin(\pi \Re(A)) \sin(\pi \Re(B))} \frac{(\Re(B))_N}{N!} \left(v_2 \frac{\e^\xi}{2 \cosh \xi}\right)^N  \right|u^N \\ & \times \left| 1 - \e^{4\xi} \right| \left| \csc(2 \Im(\xi)) \right|
\\ & + \sum_{\ell = 1}^N \left| \frac{\sin(\pi A)}{\sin(\pi \Re(A))} \frac{(\Re(A))_\ell (B)_{N - \ell}}{\ell! (N - \ell)!} \left( v_1 \frac{\e^\xi}{2 \sinh \xi} \right)^\ell \left( v_2 \frac{\e^\xi}{2 \cosh \xi} \right)^{N - \ell} \right| u^N
\\ &\times \begin{cases} 
    \left| 1 - \e^{2\xi} \right| \left| \csc(2 \Im(\xi)) \right|, & \text{if }\; 0 \le \Re\left(\e^{- 2\xi}\right) < 1, \\ 
    \left| 1 - \e^{2\xi} \right|, & \text{if }\; \Re\left(\e^{- 2\xi}\right) < 0,
\end{cases}
\end{split}
\end{gather}
provided that $0<\Re(A)<1$ and $-N <\Re(B) < 1$,
\begin{gather}\label{Rhatbound3}
\begin{split}
\left| \widehat{R}_N^- (u,v_1,v_2,\xi, A, B) \right| \le \; &\left| \frac{\sin(\pi A) \sin(\pi B)}{\sin(\pi \Re(A)) \sin(\pi \Re(B))}  \frac{(\Re(A))_N}{N!} \left( v_1 \frac{\e^{-\xi}}{2 \sinh \xi} \right)^N\right| u^N \\ &\times \begin{cases}
   \left| 1 \pm \e^{-2\xi} \right|, & \text{if }\; 1 \le \pm \Re \left( \e^{2\xi} \right), \\ 
   \left| 1 - \e^{-4\xi} \right| \left| \csc (2 \Im (\xi)) \right|, & \text{if }\; \left| \Re \left( \e^{2\xi} \right) \right| < 1 
\end{cases}
\\ & +\sum_{\ell = 0}^{N - 1} \left| \frac{\sin(\pi B)}{\sin(\pi \Re(B))} \frac{(A)_\ell (\Re(B))_{N - \ell}}{\ell! (N - \ell)!} \left( v_1 \frac{\e^{-\xi}}{2 \sinh \xi} \right)^\ell \left( v_2 \frac{\e^{-\xi}}{2 \cosh \xi} \right)^{N - \ell} \right| u^N
\\ &\times 
\begin{cases} 
    1, & \text{if }\; \Re\left(\e^{2\xi}\right) \le -1, \\ 
    \left| 1 + \e^{- 2\xi} \right| \left| \csc(2 \Im(\xi)) \right|, & \text{if }\; -1 < \Re\left(\e^{2\xi}\right) \le 0, \\ 
    \left| 1 + \e^{- 2\xi} \right|, & \text{if }\; 0 < \Re\left(\e^{2\xi}\right),
\end{cases}
\end{split}
\end{gather}
provided that $-N <\Re(A) < 1$ and $0<\Re(B)<1$, and
\begin{gather}\label{Rhatbound4}
\begin{split}
\left| \widehat{R}_N^- (u,v_1,v_2,\xi, A, B) \right| \le \; &\left| \frac{\sin(\pi A) \sin(\pi B)}{\sin(\pi \Re(A)) \sin(\pi \Re(B))} \frac{(\Re(B))_N}{N!} \left( v_2 \frac{\e^{-\xi}}{2 \cosh \xi} \right)^N\right| u^N \\ &\times \begin{cases}
   \left| 1 \pm \e^{-2\xi} \right|, & \text{if }\; 1 \le \pm \Re \left( \e^{2\xi} \right), \\ 
   \left| 1 - \e^{-4\xi} \right| \left| \csc (2 \Im (\xi)) \right|, & \text{if }\; \left| \Re \left( \e^{2\xi} \right) \right| < 1 
\end{cases}
\\ & +\sum_{\ell = 1}^N \left| \frac{\sin(\pi A)}{\sin(\pi \Re(A))} \frac{(\Re(A))_\ell (B)_{N - \ell}}{\ell! (N - \ell)!} \left( v_1 \frac{\e^{-\xi}}{2 \sinh \xi} \right)^\ell \left( v_2 \frac{\e^{-\xi}}{2 \cosh \xi} \right)^{N - \ell} \right| u^N
\\ &\times 
\begin{cases} 
    1, & \text{if }\; 1 \le \Re\left(\e^{2\xi}\right), \\ 
    \left| 1 - \e^{- 2\xi} \right| \left| \csc(2 \Im(\xi)) \right|, & \text{if }\; 0 \le \Re\left(\e^{2\xi}\right) < 1, \\ 
    \left| 1 - \e^{- 2\xi} \right|, & \text{if }\; \Re\left(\e^{2\xi}\right) < 0,
\end{cases}
\end{split}
\end{gather}
provided that $0<\Re(A)<1$ and $-N <\Re(B) < 1$. If $A$ or $B$ is an integer, then the limiting values must to be taken in these bounds. Moreover, when $\xi$, $A$, and $B$ are real with $\xi > 0$ and $0 < A, B < 1$, we have
\begin{equation}\label{Rhatbound5}
0 < \widehat{R}_N^-(u, v_1, v_2, \xi, A, B) < 2 \sum_{\ell = 0}^N \left( \frac{(A)_\ell (B)_{N - \ell}}{\ell!(N - \ell)!} \left(- v_1 \frac{\e^{-\xi}}{2 \sinh \xi}\right)^\ell \left(v_2 \frac{\e^{-\xi}}{2 \cosh \xi}\right)^{N - \ell} \right) u^N.
\end{equation}
\end{lemma}

\begin{proof} Let $N$ be a non-negative integer, let $A$ and $B$ be complex numbers, and let $u$ be a real number such that $0 < u < 1$. For $\xi \in \overline{\mathcal{D}} \setminus \mathbb{R}^+$, define the remainder $\widehat{r}_{N,1}^+(u, \xi, A)$, and for $\xi \in \overline{\mathcal{D}}$, define the remainders $\widehat{r}_{N,1}^-(u, \xi, A)$ and $\widehat{r}_{N,2}^\pm(u, \xi, B)$ by
\[
\left( 1 \mp u \frac{\e^{\pm \xi}}{2 \sinh \xi} \right)^{-A} = \sum_{n = 0}^{N - 1} \frac{(A)_n}{n!} \left( \pm \frac{\e^{\pm \xi}}{2 \sinh \xi} \right)^n u^n + \widehat{r}_{N,1}^\pm (u, \xi,A)
\]
and
\[
\left( 1 - u \frac{\e^{\pm \xi}}{2 \cosh \xi} \right)^{-B} = \sum_{n = 0}^{N - 1} \frac{(B)_n}{n!} \left( \frac{\e^{\pm \xi}}{2 \cosh \xi} \right)^n u^n + \widehat{r}_{N,2}^\pm (u, \xi,B),
\]
respectively. Using this notation, the remainders $\widehat{R}_N^\pm (u,v_1,v_2, \xi, A, B)$ can be expressed as
\begin{gather}\label{Rhatexpr1}
\begin{split}
\widehat{R}_N^\pm (u,v_1,v_2, \xi, A, B) =\; & \widehat{r}_{N,1}^\pm (uv_1, \xi, A) \widehat{r}_{0,2}^\pm (uv_2, \xi, B) \\ &+ \sum_{\ell = 0}^{N - 1} \frac{(A)_\ell}{\ell!} \left( \pm v_1 \frac{\e^{\pm \xi}}{2 \sinh \xi} \right)^\ell u^\ell \widehat{r}_{N - \ell,2}^\pm (uv_2, \xi, B),
\end{split}
\end{gather}
or
\begin{gather}\label{Rhatexpr2}
\begin{split}
\widehat{R}_N^\pm (u,v_1,v_2, \xi, A, B) =\; & \widehat{r}_{0,1}^\pm (uv_1, \xi, A) \widehat{r}_{N,2}^\pm (uv_2, \xi, B) \\ &+ \sum_{\ell = 1}^N \frac{(B)_{N - \ell}}{(N- \ell)!} \left( v_2 \frac{\e^{\pm \xi}}{2 \cosh \xi} \right)^{N - \ell} u^{N - \ell} \widehat{r}_{\ell,1}^\pm (uv_1, \xi, A).
\end{split}
\end{gather}
To obtain the required bounds, we will estimate the remainders $\widehat{r}_{N,1}^\pm(u, \xi, A)$ and $\widehat{r}_{N,2}^\pm(u, \xi, B)$. Using the standard integral representation of the beta function \cite[\href{http://dlmf.nist.gov/5.12.E1}{Eq. (5.12.1)}]{NIST:DLMF}, we have
\[
\frac{(A)_n}{n!}  = \frac{\sin(\pi A)}{\pi}   \int_0^1 t^{A + n - 1} (1 - t)^{-A} \id t,
\]
valid for $-n < \Re(A) < 1$. Summation and analytic continuation yields
\[
(1 - uw)^{-A} = \sum_{n = 0}^{N - 1} \frac{(A)_n}{n!} u^n w^n + M_N (u,w,A),
\]
where
\[
M_N (u,w,A)= u^N w^N \frac{\sin(\pi A)}{\pi} \int_0^1 \frac{t^{A + N - 1} (1 - t)^{-A}}{1 - tuw} \id t,
\]
provided that $0 < u < 1$, $w \in \mathbb{C} \setminus \left[1, +\infty\right)$, and $-N < \Re(A) < 1$. For $0 < t, u < 1$, we note that
\[
\left| 1 - tuw \right|^2  = \left| w \right|^2 t^2u^2  - 2\Re(w)tu + 1 \ge \begin{cases} 1  & \text{if }\;
\Re (w) \le  0, \\     \frac{(\Im (w))^2}{ \left| w \right|^2} & \text{if }\;
0<\Re (w) \le  \left| w \right|^2, \\  \left|1-w\right|^2 & \text{if }\;  \Re (w) > \left| w \right|^2.\end{cases}
\]
Consequently, we obtain
\begin{gather}\label{Mbound}
\begin{split}
\left| M_N (u,w,A) \right| & \le u^N|w|^N \frac{|\sin(\pi A)|}{\pi} \int_0^1 \frac{t^{\Re(A) + N - 1} (1 - t)^{-\Re(A)}}{|1 - tuw|} \id t \\ & \le \left| \frac{\sin(\pi A)}{\sin(\pi \Re(A))} \frac{(\Re(A))_N}{N!} w^N\right| u^N \times \begin{cases} 1  & \text{if }\;
\Re(w) \le  0, \\ \left|\frac{w}{\Im (w)}\right| & \text{if }\;
0<\Re (w) \le  \left| w \right|^2, \\  \frac{1}{|1-w|} & \text{if }\;  \Re(w)> \left| w \right|^2.\end{cases}
\end{split}
\end{gather}
If $A$ is an integer, then the limiting value must to be taken in this bound. Additionally, when $0 < u < 1$, $0 < w \le \frac{1}{2}$, and $0 < A < 1$, we have
\begin{equation}\label{Mineq1}
0 < M_N (u,w,A) < 2 u^N w^N \frac{\sin(\pi A)}{\pi} \int_0^1 t^{A + N - 1} (1 - t)^{-A} \id t = 2 \frac{(A)_N}{N!} w^N u^N.
\end{equation}
Similarly, when $0 < u < 1$, $w < 0$, and $0 < A < 1$, we obtain
\begin{equation}\label{Mineq2}
0 < M_N (u,w,A) < u^N w^N \frac{\sin(\pi A)}{\pi} \int_0^1 t^{A + N - 1} (1 - t)^{-A} \id t =   \frac{(A)_N}{N!} w^N u^N.
\end{equation}
Observing that
\begin{equation}\label{rhatM}
\widehat{r}_{N,1}^{\pm}(u, \xi, A) = M_N\left(u, \pm \frac{\e^{\pm \xi}}{2 \sinh \xi}, A\right), \quad 
\widehat{r}_{N,2}^{\pm}(u, \xi, B) = M_N\left(u, \frac{\e^{\pm \xi}}{2 \cosh \xi}, B\right)
\end{equation}
and applying \eqref{Mbound} to the right-hand sides, we obtain the following bounds:
\begin{gather}\label{rbound1}
\begin{split}
\left| \widehat{r}_{N,1}^+ (u, \xi,A) \right| \le \; &\left| \frac{\sin(\pi A)}{\sin(\pi \Re(A))} \frac{(\Re(A))_N}{N!} \left( \frac{\e^{\xi}}{2 \sinh \xi} \right)^N\right| u^N \\ &\times 
\begin{cases} 
    \left| 1 - \e^{2\xi} \right| \left| \csc(2 \Im(\xi)) \right|, & \text{if }\; 0 \le \Re\left(\e^{- 2\xi}\right) < 1, \\ 
    \left| 1 - \e^{2\xi} \right|, & \text{if }\; \Re\left(\e^{- 2\xi}\right)  < 0,
\end{cases}
\end{split}
\end{gather}
\begin{gather}\label{rbound2}
\begin{split}
\left| \widehat{r}_{N,1}^- (u, \xi,A) \right| \le \; &\left| \frac{\sin(\pi A)}{\sin(\pi \Re(A))} \frac{(\Re(A))_N}{N!} \left( \frac{\e^{- \xi}}{2 \sinh \xi} \right)^N\right| u^N \\ &\times 
\begin{cases} 
    1, & \text{if }\; 1 \le \Re\left(\e^{2\xi}\right), \\ 
    \left| 1 - \e^{- 2\xi} \right| \left| \csc(2 \Im(\xi)) \right|, & \text{if }\; 0 \le \Re\left(\e^{2\xi}\right) < 1, \\ 
    \left| 1 - \e^{- 2\xi} \right|, & \text{if }\; \Re\left(\e^{2\xi}\right) < 0,
\end{cases}
\end{split}
\end{gather}
provided $-N<\Re(A)<1$, and
\begin{gather}\label{rbound3}
\begin{split}
\left| \widehat{r}_{N,2}^+ (u, \xi,B) \right| \le \; &\left| \frac{\sin(\pi B)}{\sin(\pi \Re(B))} \frac{(\Re(B))_N}{N!} \left( \frac{\e^{\xi}}{2 \cosh \xi} \right)^N\right| u^N \\ &\times 
\begin{cases}
    \left| 1 + \e^{2\xi} \right| \left| \csc(2 \Im(\xi)) \right|, & \text{if }\; -1 < \Re\left(\e^{- 2\xi}\right) \le 0, \\ 
    \left| 1 + \e^{2\xi} \right|, & \text{if }\; 0< \Re\left(\e^{- 2\xi}\right),
\end{cases}
\end{split}
\end{gather}
\begin{gather}\label{rbound4}
\begin{split}
\left| \widehat{r}_{N,2}^- (u, \xi,B) \right| \le \; &\left| \frac{\sin(\pi B)}{\sin(\pi \Re(B))} \frac{(\Re(B))_N}{N!} \left( \frac{\e^{- \xi}}{2 \cosh \xi} \right)^N\right| u^N \\ &\times 
\begin{cases} 
    1, & \text{if }\; \Re\left(\e^{2\xi}\right) \le -1, \\ 
    \left| 1 + \e^{- 2\xi} \right| \left| \csc(2 \Im(\xi)) \right|, & \text{if }\; -1 < \Re\left(\e^{2\xi}\right) \le 0, \\ 
    \left| 1 + \e^{- 2\xi} \right|, & \text{if }\; 0< \Re\left(\e^{2\xi}\right),
\end{cases}
\end{split}
\end{gather}
provided $-N<\Re(B)<1$.

Straightforward estimations in \eqref{Rhatexpr1} and \eqref{Rhatexpr2} provide the bounds
\begin{gather}\label{Rhatbound01}
\begin{split}
\left|\widehat{R}_N^\pm (u,v_1,v_2, \xi, A, B)\right| \le \; & \left|\widehat{r}_{N,1}^\pm (uv_1, \xi, A) \widehat{r}_{0,2}^\pm (uv_2, \xi, B)\right| \\ &+ \sum_{\ell = 0}^{N - 1} \left|\frac{(A)_\ell}{\ell!} \left( \pm v_1 \frac{\e^{\pm \xi}}{2 \sinh \xi} \right)^\ell\right| u^\ell  \left|\widehat{r}_{N - \ell,2}^\pm (uv_2, \xi, B)\right|,
\end{split}
\end{gather}
and
\begin{gather}\label{Rhatbound02}
\begin{split}
\left|\widehat{R}_N^\pm (u,v_1,v_2, \xi, A, B)\right| \le \; & \left|\widehat{r}_{0,1}^\pm (uv_1, \xi, A) \widehat{r}_{N,2}^\pm (uv_2, \xi, B)\right| \\ &+ \sum_{\ell = 1}^N\left| \frac{(B)_{N - \ell}}{(N- \ell)!} \left( v_2 \frac{\e^{\pm \xi}}{2 \cosh \xi} \right)^{N - \ell}\right| u^{N - \ell} \left|\widehat{r}_{\ell,1}^\pm (uv_1, \xi, A)\right|,
\end{split}
\end{gather}
respectively. Applying the inequalities in \eqref{rbound1}--\eqref{rbound4} to the right-hand side of \eqref{Rhatbound01} yields the bounds \eqref{Rhatbound1} and \eqref{Rhatbound3}. Similarly, applying the same inequalities to the right-hand side of \eqref{Rhatbound02} yields the bounds \eqref{Rhatbound2} and \eqref{Rhatbound4}.

Using \eqref{Mineq1} and \eqref{Mineq2} on the right-hand sides of \eqref{rhatM}, we obtain the inequalities
\[
0 < \widehat{r}_{N,1}^-(u, \xi, A) < \frac{(A)_N}{N!} \left(-\frac{\e^{-\xi}}{2 \sinh \xi}\right)^N u^N, \quad 0 < \widehat{r}_{N,2}^-(u, \xi, B) < 2 \frac{(B)_N}{N!} \left(\frac{\e^{-\xi}}{2 \cosh \xi}\right)^N u^N,
\]
for $\xi > 0$ and $0 < A, B < 1$. Applying these inequalities to the right-hand side of \eqref{Rhatexpr1} yields \eqref{Rhatbound5}.
\end{proof}

We now proceed with the proof of Theorem \ref{Pfacthm}. Assume that $\Re(2\nu + \alpha + \beta + 1) > \Re(\alpha' + \beta')$. Substituting into \eqref{Ydef} and using \eqref{binomprod} with the values $u = t$, $v_1 = s$, $v_2 = 1 - s$, $A = \frac{1}{2} - \alpha'$, and $B = \frac{1}{2} - \beta'$, while taking into account \eqref{PHahn}, we derive \eqref{Pfac} with
\begin{gather}\label{Rhatexpr0}
\begin{split}
&\widehat{R}_N^{(P_1)}(\nu, \xi, \alpha, \beta) = \frac{\Gamma(2\nu + \alpha + \beta + 2)}{\Gamma(2\nu+\alpha+\beta-\alpha'-\beta' + 1) \Gamma\left(\alpha' + \frac{1}{2}\right) \Gamma\left(\beta' + \frac{1}{2}\right)} \\ & \times  \int_0^1 \int_0^1  t^{\alpha' + \beta'} (1 - t)^{2\nu + \alpha + \beta - \alpha' - \beta'} s^{\alpha' - 1/2} (1 - s)^{\beta' - 1/2} \widehat{R}_N^+ \left(t, s, 1 - s, \xi, \tfrac{1}{2} - \alpha', \tfrac{1}{2} - \beta' \right) \id s \id t,
\end{split}
\end{gather}
\begin{gather}\label{Rhatexpr}
\begin{split}
&\widehat{R}_M^{(P_2)}(\nu, \xi, \alpha, \beta) = \frac{\Gamma(2\nu + \alpha + \beta + 2)}{\Gamma(2\nu+\alpha+\beta-\alpha'-\beta' + 1) \Gamma\left(\alpha' + \frac{1}{2}\right) \Gamma\left(\beta' + \frac{1}{2}\right)} \\ & \times \int_0^1 \int_0^1 t^{\alpha' + \beta'} (1 - t)^{2\nu + \alpha + \beta - \alpha' - \beta'} s^{\alpha' - 1/2} (1 - s)^{\beta' - 1/2} \widehat{R}_M^- \left(t, s, 1 - s, \xi, \tfrac{1}{2} - \alpha', \tfrac{1}{2} - \beta' \right) \id s \id t.
\end{split}
\end{gather}
In deriving the form \eqref{Pfac}, we used the fact that $a_\ell(\alpha') = a_\ell(\alpha)$ and $a_\ell(\beta') = a_\ell(\beta)$ for all $\ell$. Straightforward estimation gives
\begin{gather}\label{Rhatineq1}
\begin{split}
\left|\widehat{R}_N^{(P_1)}(\nu, \xi, \alpha, \beta)\right| \le\;& \left|\frac{\Gamma(2\nu + \alpha + \beta + 2)}{\Gamma(2\nu+\alpha+\beta-\alpha'-\beta' + 1) \Gamma\left(\alpha' + \frac{1}{2}\right) \Gamma\left(\beta' + \frac{1}{2}\right)}\right| \\ & \times \int_0^1 \int_0^1  t^{\Re(\alpha' + \beta')} (1 - t)^{\Re(2\nu + \alpha + \beta - \alpha' - \beta')} s^{\Re(\alpha') - 1/2} (1 - s)^{\Re(\beta') - 1/2}\\ & \times  \left|\widehat{R}_N^+ \left(t, s, 1 - s, \xi, \tfrac{1}{2} - \alpha', \tfrac{1}{2} - \beta' \right)\right| \id s \id t,
\end{split}
\end{gather}
\begin{gather}\label{Rhatineq2}
\begin{split}
\left|\widehat{R}_M^{(P_2)}(\nu, \xi, \alpha, \beta)\right| \le\;&  \left|\frac{\Gamma(2\nu + \alpha + \beta + 2)}{\Gamma(2\nu+\alpha+\beta-\alpha'-\beta' + 1) \Gamma\left(\alpha' + \frac{1}{2}\right) \Gamma\left(\beta' + \frac{1}{2}\right)}\right| \\ & \times \int_0^1 \int_0^1 t^{\Re(\alpha' + \beta')} (1 - t)^{\Re(2\nu + \alpha + \beta - \alpha' - \beta')} s^{\Re(\alpha') - 1/2} (1 - s)^{\Re(\beta') - 1/2}\\ & \times  \left|\widehat{R}_M^- \left(t, s, 1 - s, \xi, \tfrac{1}{2} - \alpha', \tfrac{1}{2} - \beta' \right)\right| \id s \id t.
\end{split}
\end{gather}
Restricting $\xi$ to $\mathcal{D}\setminus \mathbb{R}^+$, we substitute the estimates from \eqref{Rhatbound1}, \eqref{Rhatbound2}, \eqref{Rhatbound3}, and \eqref{Rhatbound4}. Noting that $a_\ell(\alpha') = a_\ell(\alpha)$, $a_\ell(\Re(\alpha')) = a_\ell(\Re(\alpha))$, $a_\ell(\beta') = a_\ell(\beta)$ and $a_\ell(\Re(\beta')) = a_\ell(\Re(\beta))$ for all $\ell$, this yields the bounds \eqref{Pfacbound1}, \eqref{Pfacbound2}, \eqref{Pfacbound4}, and \eqref{Pfacbound5}, respectively.
Note that, for example, in deriving \eqref{Pfacbound1}, the conditions $-\frac{1}{2} < \Re(\alpha') < N + \frac{1}{2}$ and $|\Re(\beta')| < \frac{1}{2}$ apply. However, these are equivalent to $-\frac{1}{2} < \Re(\alpha) < N + \frac{1}{2}$ and $|\Re(\beta)| < \frac{1}{2}$. A similar remark applies to the other bounds.

The estimates \eqref{Pfacbound3} and \eqref{Pfacbound6} can be derived from \eqref{Pfacbound1} and \eqref{Pfacbound4} by applying the following inequality:
\begin{gather}\label{aineq}
\begin{split}
& \left| \frac{\Gamma \left( \Re(\alpha') + \ell + \frac{1}{2} \right)}{\Gamma \left( \alpha' + \ell + \frac{1}{2} \right)} a_\ell (\alpha) \right|
= \left| \frac{\Gamma \left( \Re(\alpha') + \ell + \frac{1}{2} \right)}{\Gamma \left( \alpha' + \ell + \frac{1}{2} \right)} a_\ell (\alpha') \right|
\\ & = \left| \frac{\cos(\pi \alpha)}{\cos(\pi \Re(\alpha))} \frac{\Gamma \left( -\alpha' + \ell + \frac{1}{2} \right)}{\Gamma \left( -\Re(\alpha') + \ell + \frac{1}{2} \right)} a_\ell (\Re(\alpha')) \right|
= \left| \frac{\cos(\pi \alpha)}{\cos(\pi \Re(\alpha))} \frac{\Gamma \left( -\alpha' + \ell + \frac{1}{2} \right)}{\Gamma \left( -\Re(\alpha') + \ell + \frac{1}{2} \right)} a_\ell (\Re(\alpha)) \right|
\\ & \le \frac{\Gamma \left( \Re(\alpha') + \frac{1}{2} \right)}{\left| \Gamma \left( \alpha' + \frac{1}{2} \right) \right|} \left| \frac{\cos(\pi \alpha)}{\cos(\pi \Re(\alpha))} a_\ell (\Re(\alpha)) \right|
\end{split}
\end{gather}
which holds when $\ell \ge 0$ and $|\Re(\alpha')| < \frac{1}{2}$ (or, equivalently, $|\Re(\alpha)| < \frac{1}{2}$). We then use the following observations:
\begin{enumerate}[(i)]
    \item If $-1 < \Re\left(\e^{-2\xi}\right) \le 0$, then $\left|1 + \e^{2\xi}\right| \le \left|1 - \e^{4\xi}\right|$.
    \item If $\Re\left(\e^{-2\xi}\right) > 0$, then $\left|1 + \e^{2\xi}\right| \le \left|1 - \e^{4\xi}\right| \left|\csc\left(2 \Im(\xi)\right)\right|$.
    \item If $1 \le \Re\left(\e^{2\xi}\right)$, then $1 \le \left|1 + \e^{-2\xi}\right|$.
    \item If $0 \le \Re\left(\e^{2\xi}\right) < 1$, then $\left|1 - \e^{-2\xi}\right| \le \left|1 - \e^{-4\xi}\right|$.
    \item If $-1 < \Re\left(\e^{2\xi}\right) < 0$, then $\left|1 - \e^{-2\xi}\right| \le \left|1 - \e^{-4\xi}\right| \left|\csc\left(2 \Im(\xi)\right)\right|$.
\end{enumerate}
This completes the proof of Theorem \ref{Pfacthm}.

To prove Theorem \ref{Qfacthm}, assume that $\Re(2\nu + \alpha + \beta + 1) > \Re(\alpha' + \beta')$. Substituting into \eqref{Ydef} and using \eqref{binomprod} with the values $u = t$, $v_1 = s$, $v_2 = 1 - s$, $A = \frac{1}{2} - \alpha'$, and $B = \frac{1}{2} - \beta'$, while taking into account \eqref{QHahn}, we obtain \eqref{Qfac}, with $\widehat{R}_N^{(Q)}(\nu, \xi, \alpha, \beta)=\widehat{R}_N^{(P_2)}(\nu, \xi, \alpha, \beta)$ (cf. \eqref{Rhatexpr}). Thus, the estimates for $\widehat{R}_N^{(Q)}(\nu, \xi, \alpha, \beta)$ directly follow from those for $\widehat{R}_N^{(P_2)}(\nu, \xi, \alpha, \beta)$, the latter of which are also valid for $\xi > 0$.

Now, assume that $\xi > 0$, and that $\nu$, $\alpha$, and $\beta$ are real, with $|\alpha|, |\beta| < \frac{1}{2}$ and $2\nu + \alpha + \beta + 1  > \alpha' + \beta'$. Since $\widehat{R}_N^{(Q)}(\nu, \xi, \alpha, \beta)=\widehat{R}_N^{(P_2)}(\nu, \xi, \alpha, \beta)$, substituting \eqref{Rhatbound5} into \eqref{Rhatexpr}, and observing that $a_\ell(\alpha') = a_\ell(\alpha)$ and $a_\ell(\beta') = a_\ell(\beta)$ for all $\ell$, we obtain
\begin{align*}
0 & <\widehat{R}_N^{(Q)}(\nu, \xi, \alpha, \beta) \\ & < 2 \times (-1)^N \sum_{\ell=0}^N a_\ell(\alpha) a_{N - \ell}(\beta) \left( - \frac{\e^{-\xi}}{\sinh \xi} \right)^\ell \left( \frac{\e^{-\xi}}{\cosh \xi} \right)^{N - \ell} \frac{\Gamma(2\nu + \alpha + \beta + 2)}{\Gamma(2\nu + \alpha + \beta + N + 2)}
\\ & = 2 \times (-1)^N g_N(-\xi, \alpha, \beta) \frac{\Gamma(2\nu + \alpha + \beta + 2)}{\Gamma(2\nu + \alpha + \beta + N + 2)}.
\end{align*}
This concludes the proof of Theorem \ref{Qfacthm}.

We proceed by showing that the factorial expansions \eqref{Pfacasymp} and \eqref{Qfacasymp} terminate, providing an exact representation of the corresponding functions when both $2\alpha$ and $2\beta$ are odd integers. Let $L$ denote the positive integer defined by $\max (2\left| \alpha \right|,2\left| \beta \right|) = 2L - 1$. Then, by definition \eqref{binomprod}, we have
\[
\widehat{R}_{2L}^ \pm  \left( t,s,1 - s,\xi ,\tfrac{1}{2}-\alpha',\tfrac{1}{2}-\beta' \right) = 0.
\]
Thus, from \eqref{Rhatexpr0} and \eqref{Rhatexpr}, it follows that $\widehat{R}_{2L}^{(P_1)}(\nu, \xi, \alpha, \beta) = 0$ and $\widehat{R}_{2L}^{(P_2)}(\nu, \xi, \alpha, \beta) = 0$. Consequently, by Theorems \ref{Pfacthm} and \ref{Qfacthm}, the expansions \eqref{Pfacasymp} and \eqref{Qfacasymp} terminate, yielding exact representations of the corresponding functions. These parameter conditions may be relaxed by invoking analytic continuation.

Finally, observe that when
\[
\left| \frac{\e^{\pm \xi}}{2 \sinh \xi} \right|,\quad \left| \frac{\e^{\pm \xi}}{2 \cosh \xi} \right| < 1 \quad \Longleftrightarrow \quad \left| \Re \left(\e^{\pm 2\xi} \right) \right| < \tfrac{1}{2},
\]
the convergence of the binomial expansion implies $\lim_{N \to +\infty} \widehat{R}_N^\pm(u, v_1, v_2, \xi, A, B) = 0$. Therefore, by \eqref{Rhatexpr0} and \eqref{Rhatexpr}, it follows that $\lim_{N \to +\infty} \widehat{R}_N^{(P_1)}(\nu, \xi, \alpha, \beta) = 0$ and $\lim_{M \to +\infty} \widehat{R}_M^{(P_2)}(\nu, \xi, \alpha, \beta) = 0$, provided that $\Re(2\nu + \alpha + \beta + 1) > \Re(\alpha' + \beta')$. As a result, under these conditions, the infinite series in \eqref{Pfacasymp} converges to its left-hand side. Since $\widehat{R}_N^{(Q)}(\nu, \xi, \alpha, \beta) = \widehat{R}_N^{(P_2)}(\nu, \xi, \alpha, \beta)$, the series in \eqref{Qfacasymp} converges to its left-hand side under the same assumptions.

\section{Proof of Theorems \ref{Pfacthm2} and \ref{Qfacthm2}}\label{fproof2}.

In this section, we establish the error bounds stated in Theorems \ref{Pfacthm2} and \ref{Qfacthm2}. To this end, we require alternative estimates for the remainder terms $\widehat{R}_N^\pm(u, v_1, v_2, \xi, A, B)$ defined in Lemma \ref{lemma2}.

\begin{lemma}\label{lemma3} Let $N$ be a non-negative integer, and let $A$ and $B$ be complex numbers. Let $u$, $v_1$, and $v_2$ be real numbers such that $0 < u, v_1, v_2 < 1$. For $\xi \in \overline{\mathcal{D}} \setminus \mathbb{R}^+$, define the remainder $\widehat{R}_N^+(u, v_1, v_2, \xi, A, B)$, and for $\xi \in \overline{\mathcal{D}}$, define the remainder $\widehat{R}_N^-(u, v_1, v_2, \xi, A, B)$ by \eqref{binomprod}. Then the following bounds hold:
\begin{gather}\label{Hahnbound}
\begin{split}
\left|\widehat{R}_N^\pm(u, v_1, v_2, \xi, A, B)\right| \le\; & \frac{\Gamma(\Re(A)) \Gamma(\Re(B))}{\left| \Gamma(A) \Gamma(B) \right|}
\frac{\left|\Gamma(A + B + N)\right|}{\Gamma(\Re(A + B) + N)} \\ &\times
\sum_{\ell = 0}^N \frac{(\Re(A))_\ell (\Re(B))_{N - \ell}}{\ell! (N - \ell)!} 
\left|\left( v_1 \frac{\e^{\pm \xi}}{2 \sinh \xi} \right)^\ell 
\left( v_2 \frac{\e^{\pm \xi}}{2 \cosh \xi} \right)^{N - \ell} \right|u^N 
\\ &\times \max_{0 \le q \le 1} \mathsf{g}\left( \pm quv_1 \frac{\e^{\pm \xi}}{2 \sinh \xi} + (1 - q) uv_2 \frac{\e^{\pm \xi}}{2 \cosh \xi} \right),
\end{split}
\end{gather}
under the conditions $\Re(A), \Re(B) > 0$ and $\Re(A + B) < 1$. Here, $\mathsf{g}$ is defined by \eqref{gdef}.
\end{lemma}

\begin{proof} The proof proceeds similarly to the argument given in \cite[Lemma 3]{Hahn1980} for the special case where $A$ and $B$ are real. We omit the details.
\end{proof}

We now proceed with the proof of Theorem \ref{Pfacthm2}. Assume that $\Re(\alpha'), \Re(\beta') < \frac{1}{2}$ and that $0<\Re(\alpha' + \beta') < \Re(2\nu+\alpha + \beta+1)$. By restricting $\xi$ to $\mathcal{D}\setminus \mathbb{R}^+$, we substitute the bound \eqref{Hahnbound} into the right-hand side of \eqref{Rhatineq1}. Noting that $a_\ell(\Re(\alpha')) = a_\ell(\Re(\alpha))$ and $a_\ell(\Re(\beta')) = a_\ell(\Re(\beta))$ for all $\ell$, we obtain the estimate
\begin{align*}
& \left|\widehat{R}_N^{(P_1)}(\nu, \xi, \alpha, \beta)\right| \le \frac{\Gamma(\Re(2\nu + \alpha + \beta - \alpha' - \beta') + 1)}{\left| \Gamma(2\nu + \alpha + \beta - \alpha' - \beta' + 1) \right|}
\frac{\left| \Gamma(-\alpha' - \beta' + N + 1) \right|}{\Gamma(\Re(-\alpha' - \beta') + N + 1)} \\ 
&  \times \sum_{\ell = 0}^N \left| \frac{\cos(\pi \alpha) \cos(\pi \beta)}{\cos(\pi \Re(\alpha)) \cos(\pi \Re(\beta))}  a_\ell (\Re(\alpha)) a_{N - \ell} (\Re(\beta)) \left( \frac{\e^\xi}{\sinh \xi} \right)^\ell \left( \frac{\e^\xi}{\cosh \xi} \right)^{N - \ell} \right| \\ 
&  \times \frac{\left|\Gamma(2\nu + \alpha + \beta + 2)\right|}{\Gamma(\Re(2\nu + \alpha + \beta) + N + 2)} \times \max_{0 \le q, t, s \le 1} \mathsf{g}\left( qst \frac{\e^\xi}{2\sinh \xi} + (1 - q)(1 - s)t \frac{\e^\xi}{2\cosh \xi} \right).
\end{align*}
Using \cite[Eq. (16.6)]{Hahn1980}, we find
\[
\max_{0 \le q, t, s \le 1} \mathsf{g}\left( qst \frac{\e^\xi}{2 \sinh \xi} + (1 - q)(1 - s)t \frac{\e^\xi}{2 \cosh \xi} \right) = \max_{0 \le q, s \le 1} \mathsf{g}\left( qs \frac{\e^\xi}{2 \sinh \xi} + (1 - q)(1 - s) \frac{\e^\xi}{2 \cosh \xi} \right).
\]
By \cite[Eqs. (16.7) and (16.17)]{Hahn1980}, this further simplifies to
\[
\max_{0 \le u \le \frac{\pi}{2}} \mathsf{g}\left( \frac{\sin^2 u}{1 + \sin(2u)} \frac{\e^\xi}{2 \sinh \xi} + \frac{\cos^2 u}{1 + \sin(2u)} \frac{\e^\xi}{2 \cosh \xi} \right),
\]
which proves the bound \eqref{Hahn1}. The case $\alpha = \beta = 0$ follows from a continuity argument. Note that the condition $\Re(\alpha'), \Re(\beta') < \frac{1}{2}$ is equivalent to $\left|\Re(\alpha)\right|,\left|\Re(\beta)\right| < \frac{1}{2}$. The bound \eqref{Hahn2} can be established in a completely analogous manner by combining \eqref{Hahnbound} and \eqref{Rhatineq2}. This completes the proof of Theorem \ref{Pfacthm2}.

For Theorem \ref{Qfacthm2}, note that $\widehat{R}_N^{(Q)}(\nu, \xi, \alpha, \beta) = \widehat{R}_N^{(P_2)}(\nu, \xi, \alpha, \beta)$. Consequently, the estimate \eqref{Hahn3} follows directly from \eqref{Hahn2}, which is also valid for $\xi > 0$.

\section{Proof of Theorem \ref{cutinfacthm2}}\label{fproof3}

In this section, we establish the error bounds presented in Theorem \ref{cutinfacthm2}. Suppose that $0 < \zeta  < \frac{\pi}{2}$. Assuming that $\Re(2\nu + \alpha + \beta + 1) > \Re(\alpha' +\beta')$, and noting that $\widehat{R}_N^{(Q)}(\nu,\xi,\alpha,\beta)=\widehat{R}_N^{(P2)}(\nu,\xi,\alpha,\beta)$, we substitute \eqref{Qfac} into the right-hand side of \eqref{PQrel2}. This results in \eqref{Pfac2} with
\begin{align*}
\widehat{R}_N^{(P)}(\nu, \zeta, \alpha, \beta) =\; & \tfrac{1}{2} \e^{-\left( (2\nu + \alpha + \beta + 1)\zeta - \left( \alpha + \frac{1}{2} \right) \frac{\pi}{2} \right) \im} \lim_{\varepsilon \to 0^+} \widehat{R}_N^{(Q)}(\nu, \varepsilon + \im\zeta, \alpha, \beta) \\ & + \tfrac{1}{2} \e^{\left( (2\nu + \alpha + \beta + 1)\zeta - \left( \alpha + \frac{1}{2} \right) \frac{\pi}{2} \right) \im} \lim_{\varepsilon \to 0^+} \widehat{R}_N^{(Q)}(\nu, \varepsilon - \im\zeta, \alpha, \beta).
\end{align*}
Applying the representation \eqref{Rhatexpr} and passing to the limit, we obtain
\begin{gather}\label{Pcutremainder}
\begin{split}
& \widehat{R}_N^{(P)}(\nu, \zeta, \alpha, \beta) = \frac{\Gamma(2\nu + \alpha + \beta + 2)}{2\Gamma(2\nu+\alpha+\beta-\alpha'-\beta' + 1) \Gamma\left(\alpha' + \frac{1}{2}\right) \Gamma\left(\beta' + \frac{1}{2}\right)}\e^{-\left( (2\nu + \alpha + \beta + 1)\zeta - \left( \alpha + \frac{1}{2} \right) \frac{\pi}{2} \right) \im} \\ & \times \int_0^1 \int_0^1 t^{\alpha' + \beta'} (1 - t)^{2\nu + \alpha + \beta - \alpha' - \beta'} s^{\alpha' - 1/2} (1 - s)^{\beta' - 1/2} \widehat{R}_N^- \left(t, s, 1 - s, \im\zeta, \tfrac{1}{2} - \alpha', \tfrac{1}{2} - \beta' \right) \id s \id t\\
\\ & +\frac{\Gamma(2\nu + \alpha + \beta + 2)}{2\Gamma(2\nu+\alpha+\beta-\alpha'-\beta' + 1) \Gamma\left(\alpha' + \frac{1}{2}\right) \Gamma\left(\beta' + \frac{1}{2}\right)}\e^{\left( (2\nu + \alpha + \beta + 1)\zeta - \left( \alpha + \frac{1}{2} \right) \frac{\pi}{2} \right) \im} \\ & \times \int_0^1 \int_0^1 t^{\alpha' + \beta'} (1 - t)^{2\nu + \alpha + \beta - \alpha' - \beta'} s^{\alpha' - 1/2} (1 - s)^{\beta' - 1/2} \widehat{R}_N^- \left(t, s, 1 - s, -\im\zeta, \tfrac{1}{2} - \alpha', \tfrac{1}{2} - \beta' \right) \id s \id t.
\end{split}
\end{gather}
A straightforward estimation yields the bound
\begin{gather}\label{RhatPbound}
\begin{split}
\left|\widehat{R}_N^{(P)}(\nu, \zeta, \alpha, \beta)\right| \le \;& \left|\frac{\Gamma(2\nu + \alpha + \beta + 2)}{2\Gamma(2\nu+\alpha+\beta-\alpha'-\beta' + 1) \Gamma\left(\alpha' + \frac{1}{2}\right) \Gamma\left(\beta' + \frac{1}{2}\right)}\right|\e^{\Im\left(\zeta_{\nu ,N,0}^{(3)}\right)} \\ & \times \int_0^1 \int_0^1 t^{\Re(\alpha' + \beta')} (1 - t)^{\Re(2\nu + \alpha + \beta - \alpha' - \beta')} s^{\Re(\alpha') - 1/2} (1 - s)^{\Re(\beta') - 1/2}\\ & \times  \left|\widehat{R}_N^- \left(t, s, 1 - s, \im\zeta, \tfrac{1}{2} - \alpha', \tfrac{1}{2} - \beta' \right)\right| \id s \id t
\\ & +\left|\frac{\Gamma(2\nu + \alpha + \beta + 2)}{2\Gamma(2\nu+\alpha+\beta-\alpha'-\beta' + 1) \Gamma\left(\alpha' + \frac{1}{2}\right) \Gamma\left(\beta' + \frac{1}{2}\right)}\right|\e^{-\Im\left(\zeta_{\nu ,N,0}^{(3)}\right)} \\ & \times \int_0^1 \int_0^1 t^{\Re(\alpha' + \beta')} (1 - t)^{\Re(2\nu + \alpha + \beta - \alpha' - \beta')} s^{\Re(\alpha') - 1/2} (1 - s)^{\Re(\beta') - 1/2}\\ & \times  \left|\widehat{R}_N^- \left(t, s, 1 - s, -\im\zeta, \tfrac{1}{2} - \alpha', \tfrac{1}{2} - \beta' \right)\right| \id s \id t.
\end{split}
\end{gather}
From Lemma \ref{lemma2}, we can deduce that
\begin{align*}
\left| \widehat{R}_N^- (u,v_1,v_2,\pm\im\zeta, A, B) \right| \le \; &2\left| \frac{\sin(\pi A) \sin(\pi B)}{\sin(\pi \Re(A)) \sin(\pi \Re(B))}  \frac{(\Re(A))_N}{N!} \left(  \frac{v_1}{2 \sin \zeta} \right)^N\right| u^N
\\ & +\sum_{\ell = 0}^{N - 1} \left| \frac{\sin(\pi B)}{\sin(\pi \Re(B))} \frac{(A)_\ell (\Re(B))_{N - \ell}}{\ell! (N - \ell)!} \left( \frac{v_1 }{2 \sin\zeta} \right)^\ell \left( \frac{v_2}{2 \cos\zeta} \right)^{N - \ell} \right| u^N
\\ &\times 
\begin{cases} 
     2\cos \zeta, & \text{if }\; 0 < \zeta  < \frac{\pi }{4}, \\ 
    \csc \zeta, & \text{if }\; \frac{\pi }{4} \le \zeta  < \frac{\pi }{2},
\end{cases}
\end{align*}
provided that $-N <\Re(A) < 1$ and $0<\Re(B)<1$, and
\begin{align*}
\left| \widehat{R}_N^- (u,v_1,v_2,\pm\im\zeta, A, B) \right| \le \; &2\left| \frac{\sin(\pi A) \sin(\pi B)}{\sin(\pi \Re(A)) \sin(\pi \Re(B))} \frac{(\Re(B))_N}{N!} \left(\frac{v_2}{2 \cos\zeta} \right)^N\right| u^N 
\\ & +\sum_{\ell = 1}^N \left| \frac{\sin(\pi A)}{\sin(\pi \Re(A))} \frac{(\Re(A))_\ell (B)_{N - \ell}}{\ell! (N - \ell)!} \left( \frac{v_1}{2 \sin\zeta} \right)^\ell \left( \frac{v_2}{2 \cos\zeta} \right)^{N - \ell} \right| u^N
\\ &\times 
\begin{cases} 
     \sec\zeta, & \text{if }\; 0 < \zeta  < \frac{\pi }{4}, \\ 
    2\sin \zeta, & \text{if }\; \frac{\pi }{4} \le \zeta  < \frac{\pi }{2},
\end{cases}
\end{align*}
provided that $0<\Re(A)<1$ and $-N <\Re(B) < 1$. If $A$ or $B$ is an integer, then the limiting values must to be taken in these bounds. Similarly, from Lemma \ref{lemma3} and \cite[Eq. (16.19)]{Hahn1980}, we obtain
\begin{align*}
\left|\widehat{R}_N^-(u, v_1, v_2, \pm \im \zeta, A, B)\right| \le\; & 2\frac{\Gamma(\Re(A)) \Gamma(\Re(B))}{\left| \Gamma(A) \Gamma(B) \right|}
\frac{\left|\Gamma(A + B + N)\right|}{\Gamma(\Re(A + B) + N)} \\ &\times
\sum_{\ell = 0}^N \frac{(\Re(A))_\ell (\Re(B))_{N - \ell}}{\ell! (N - \ell)!} 
\left|\left(\frac{v_1}{2 \sin\zeta} \right)^\ell 
\left(\frac{v_2}{2 \cos\zeta} \right)^{N - \ell} \right|u^N,
\end{align*}
under the conditions $\Re(A), \Re(B) > 0$ and $\Re(A + B) < 1$. Substituting these estimates into \eqref{RhatPbound} and observing that $\Im\left(\zeta_{\nu ,N,0}^{(3)}\right)=\Im\left(\zeta_{\nu ,N,\ell}^{(3)}\right)$, $a_\ell(\alpha') = a_\ell(\alpha)$, $a_\ell(\Re(\alpha')) = a_\ell(\Re(\alpha))$, $a_\ell(\beta') = a_\ell(\beta)$, and $a_\ell(\Re(\beta')) = a_\ell(\Re(\beta))$ for all $0\le\ell\le N$, we obtain the bounds \eqref{cutfacbound1}, \eqref{cutfacbound2}, and \eqref{cutfacbound4}, respectively. The case $\alpha = \beta = 0$ in \eqref{cutfacbound4} follows by a continuity argument. Note that, for instance, in deriving \eqref{cutfacbound1}, the conditions $-\frac{1}{2} < \Re(\alpha') < N + \frac{1}{2}$ and $|\Re(\beta')| < \frac{1}{2}$ are required. However, these are equivalent to $-\frac{1}{2} < \Re(\alpha) < N + \frac{1}{2}$ and $|\Re(\beta)| < \frac{1}{2}$. A similar observation holds for the other bounds. The estimate \eqref{cutfacbound3} can be derived from \eqref{cutfacbound1} by applying the inequality \eqref{aineq}.

The bounds for the remainder terms $\widehat{R}_N^{(\SzQ)} (\nu, \zeta, \alpha, \beta)$ and $\widehat{R}_N^{(\DQ)} (\nu, \zeta, \alpha, \beta)$ in the truncated factorial series \eqref{Qfac2} and \eqref{Qfac3} can be derived in a similar fashion by combining \eqref{Qfac} with \eqref{Qrel3} and \eqref{Qrel4}, respectively. The details are left for the reader to verify.

The integral representation of $\widehat{R}_N^{(P)} (\nu, \zeta, \alpha, \beta)$, along with the analogous representations for $\widehat{R}_N^{(\SzQ)} (\nu, \zeta, \alpha, \beta)$ and $\widehat{R}_N^{(\DQ)} (\nu, \zeta, \alpha, \beta)$, as well as the discussion at the end of Section \ref{fproof1}, demonstrate that the factorial expansions \eqref{Pfacasymp2}, \eqref{Qfacasymp2}, and \eqref{Qfacasymp3} terminate, providing an exact representation of the corresponding functions when both $2\alpha$ and $2\beta$ are odd integers.

Finally, note that when
\[
\frac{1}{|2\sin \zeta|}, \quad \frac{1}{|2\cos \zeta|} < 1 \quad \Longleftrightarrow \quad \tfrac{\pi}{6} < \zeta < \tfrac{\pi}{3},
\]
the convergence of the binomial expansion implies $\lim_{N \to +\infty} \widehat{R}_N^-(u, v_1, v_2, \pm\im\zeta, A, B) = 0$. Consequently, from \eqref{Pcutremainder}, we have $\lim_{N \to +\infty} \widehat{R}_N^{(P)}(\nu, \zeta, \alpha, \beta) = 0$, provided that $\Re(2\nu + \alpha + \beta + 1) > \Re(\alpha' + \beta')$. Thus, under these conditions, the infinite series in \eqref{Pfacasymp2} converges to its left-hand side. The convergence of the series \eqref{Qfacasymp2} and \eqref{Qfacasymp3} follows in a similar manner, based on the analogous integral representations of the remainder terms $\widehat{R}_N^{(\SzQ)}(\nu, \zeta, \alpha, \beta)$ and $\widehat{R}_N^{(\DQ)}(\nu, \zeta, \alpha, \beta)$.

\section{Numerical examples}\label{numerics}

In this section, we present numerical examples to illustrate the sharpness of our error bounds and the accuracy of the asymptotic approximations.

Figure \ref{fig1} shows the ratio of the left-hand side to the right-hand side of \eqref{boundnum} for $\nu = 200 + \im t$, $\alpha = 2 + \frac{\im}{2}$, $\beta = \frac{1}{4}$, and $\xi = \frac{3}{4}$, with $N = 2$ (black) and $N = 4$ (grey). The plot demonstrates that the error bound \eqref{boundnum} is particularly sharp when the real part of $\nu$ is large compared to its imaginary part.

\begin{figure}[!ht]
  \centering
  \includegraphics[width=0.5\textwidth]{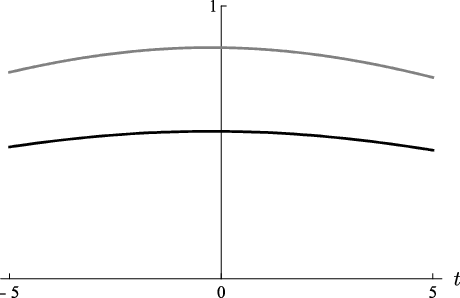}
  \caption{Ratio of the left-hand side to the right-hand side of \eqref{boundnum} for $\nu = 200 + \im t$, $\alpha = 2 + \frac{\im}{2}$, $\beta = \frac{1}{4}$, $\xi = \frac{3}{4}$, with $N = 2$ (black) and $N = 4$ (grey).}
  \label{fig1}
\end{figure}

Figure \ref{fig2} depicts the ratio of the left-hand side (with $R_N^{(P)}$) to the right-hand side of \eqref{cutbound3} for $\nu = 200 + \im t$, $\alpha = \frac{1}{5}$, $\beta = \frac{1}{4}$, and $\zeta = \frac{\pi}{3}$, with $N = 3$ (black) and $N = 5$ (grey). The plot shows that the error bound \eqref{boundnum} is especially sharp when the imaginary part of $\nu$ is not close to $0$. This behaviour arises because the $\cosh$ factor in the bound \eqref{cutbound3} becomes significant compared to the $\cos$ factors appearing in the inverse factorial series \eqref{invfexp1} when $\nu$ is real.

\begin{figure}[!ht]
  \centering
  \includegraphics[width=0.5\textwidth]{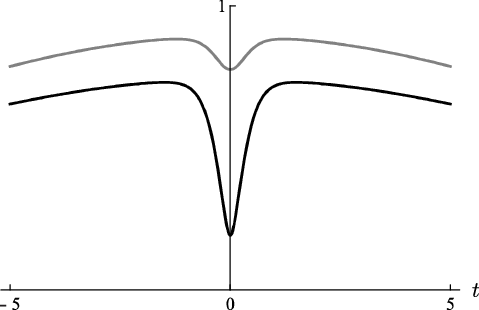}
  \caption{Ratio of the left-hand side (with $R_N^{(P)}$) to the right-hand side of \eqref{cutbound3} for $\nu = 200 + \im t$, $\alpha = \frac{1}{5}$, $\beta = \frac{1}{4}$, $\zeta = \frac{\pi}{3}$, with $N = 3$ (black) and $N = 5$ (grey).}
  \label{fig2}
\end{figure}

Figure \ref{fig3} shows the ratio of the left-hand side to the right-hand side of \eqref{numericsbound2} for $\alpha = \frac{1}{4} + \frac{\im}{4}$, $\beta = \frac{1}{3}$, $\xi = 0.01 + \im t$, and $N = 4$, with $\nu = 150$ (black) and $\nu = 150 + 50 \im$ (grey). The plot demonstrates that the bound remains accurate near the critical points $\xi = 0$ and $\xi = \pm \frac{\pi}{2}\im$. Moreover, the bound \eqref{numericsbound2} remains realistic even when the real part of $\nu$ is not significantly larger than its imaginary part.

\begin{figure}[!ht]
  \centering
  \includegraphics[width=0.5\textwidth]{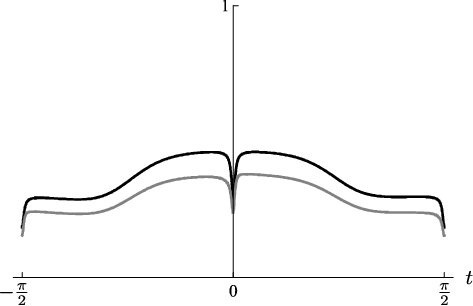}
  \caption{Ratio of the left-hand side to the right-hand side of \eqref{numericsbound2} for $\alpha = \frac{1}{4}+\frac{\im}{4}$, $\beta = \frac{1}{3}$, $\xi = 0.01+\im t$, $N=4$, with $\nu = 150$ (black) and $\nu = 150+50 \im$ (grey).}
  \label{fig3}
\end{figure}

Figure \ref{fig4} depicts the ratio of the left-hand side (with $\widehat{R}_N^{(P)}$) to the right-hand side of \eqref{cutfacbound3} for $\nu = 200 + \im t$, $\alpha = \frac{1}{5} + \frac{\im}{3}$, $\beta = \frac{1}{4}$, and $\zeta = \frac{\pi}{3}$, with $N = 3$ (black) and $N = 5$ (grey). The plot confirms that the error bound \eqref{cutfacbound3} remains especially sharp when the imaginary part of $\nu$ is away from zero, as previously explained in connection with the $\cosh$ and $\cos$ factors. Additionally, the bound stays realistic even when the real part of $\nu$ is not significantly larger than its imaginary part.

\begin{figure}[!ht]
 \centering
  \includegraphics[width=0.5\textwidth]{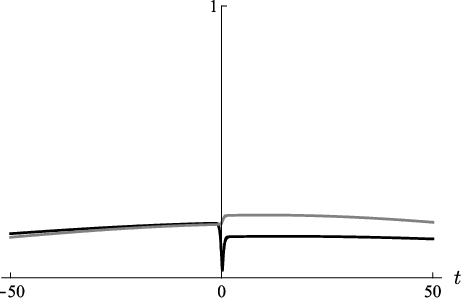}
\caption{Ratio of the left-hand side (with $\widehat{R}_N^{(P)}$) to the right-hand side of \eqref{cutfacbound3} for $\nu = 200 + \im t$, $\alpha = \frac{1}{5}+\frac{\im}{3}$, $\beta = \frac{1}{4}$, $\zeta = \frac{\pi}{3}$, with $N = 3$ (black) and $N = 5$ (grey).}
  \label{fig4}
\end{figure}

Finally, Figure \ref{fig5} shows the absolute value of the remainder $\widehat{R}_N^{(P)}(\nu, \zeta, \alpha, \beta)$ in \eqref{Pfac2} for $\alpha = \frac{1}{5}$, $\beta = \frac{1}{3}+\frac{\im}{6}$, $0<\zeta <\frac{1}{10}$, $N = 3$, with $\nu = 50$ (black) and $\nu = 100$ (grey). The figure suggests that the factorial expansion \eqref{Pfacasymp2} remains useful near the endpoint $\zeta = 0$, particularly for larger values of $\nu$.

\begin{figure}[!ht]
 \centering
  \includegraphics[width=0.5\textwidth]{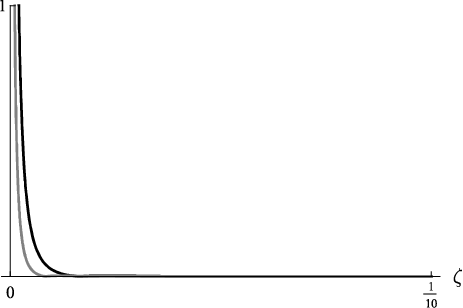}
\caption{The absolute value of the remainder $\widehat{R}_N^{(P)}(\nu, \zeta, \alpha, \beta)$ in \eqref{Pfac2} for $\alpha = \frac{1}{5}$, $\beta = \frac{1}{3}+\frac{\im}{6}$, $0<\zeta <\frac{1}{10}$, $N = 3$, with $\nu = 50$ (black) and $\nu = 100$ (grey).}
  \label{fig5}
\end{figure}

\section*{Acknowledgement} The author gratefully acknowledges the referees for their insightful comments and valuable suggestions, which have significantly improved the clarity and presentation of the paper.

\end{document}